\documentclass{amsart}

\usepackage[all]{xy}
\xyoption{poly}
\xyoption{arc}
\xyoption{curve}

\CompileMatrices

\begin{document}

\renewcommand{\theequation}{\thesection.\arabic{equation}}
\newcommand{\nc}{\newcommand}

\nc{\pr}{\noindent{\em Proof. }} \nc{\g}{\mathfrak g}
\nc{\n}{\mathfrak n} \nc{\opn}{\overline{\n}}\nc{\h}{\mathfrak h}
\renewcommand{\b}{\mathfrak b}
\nc{\Ug}{U(\g)} \nc{\Uh}{U(\h)} \nc{\Un}{U(\n)}
\nc{\Uopn}{U(\opn)}\nc{\Ub}{U(\b)} \nc{\p}{\mathfrak p}
\renewcommand{\l}{\mathfrak l}
\nc{\z}{\mathfrak z} \renewcommand{\h}{\mathfrak h}
\nc{\m}{\mathfrak m}
\renewcommand{\k}{\mathfrak k}
\nc{\opk}{\overline{\k}}
\nc{\opb}{\overline{\b}}
\nc{\e}{{\epsilon}}
\nc{\ke}{{\bf k}_\e}
\nc{\Hk}{{\rm Hk}^{\gr}(A,A_0,\e )}
\nc{\gr}{\bullet}
\nc{\ra}{\rightarrow}
\nc{\Alm}{A-{\rm mod}}
\nc{\DAl}{{D}^-(A)}
\nc{\HA}{{\rm Hom}_A}

\newtheorem{theorem}{Theorem}{}
\newtheorem{lemma}[theorem]{Lemma}{}
\newtheorem{corollary}[theorem]{Corollary}{}
\newtheorem{conjecture}[theorem]{Conjecture}{}
\newtheorem{proposition}[theorem]{Proposition}{}
\newtheorem{axiom}{Axiom}{}
\newtheorem{remark}{Remark}{}
\newtheorem{example}{Example}{}
\newtheorem{exercise}{Exercise}{}
\newtheorem{definition}{Definition}{}

\renewcommand{\thetheorem}{\thesection.\arabic{theorem}}

\renewcommand{\thelemma}{\thesection.\arabic{lemma}}

\renewcommand{\theproposition}{\thesection.\arabic{proposition}}

\renewcommand{\thecorollary}{\thesection.\arabic{corollary}}

\renewcommand{\theremark}{\thesection.\arabic{remark}}

\renewcommand{\thedefinition}{\thesection.\arabic{definition}}

\title{Localization of quantum biequivariant $\mathcal{D}$--modules and q-W algebras}

\author{A. Sevostyanov}

\address{Institute of Pure and Applied Mathematics,
University of Aberdeen \\ Aberdeen AB24 3UE, United Kingdom \\ e-mail: a.sevastyanov@abdn.ac.uk }

\begin{abstract}
We present a biequivariant version of Kremnizer--Tanisaki localization theorem for quantum $\mathcal{D}$--modules. We also obtain an equivalence between a category of finitely generated equivariant modules over a quantum group and a category of finitely generated modules over a q-W algebra which can be regarded as an equivariant quantum group version of Skryabin equivalence. The biequivariant localization theorem for quantum $\mathcal{D}$--modules together with the equivariant quantum group version of Skryabin equivalence yield an equivalence between a certain category of quantum biequivariant $\mathcal{D}$--modules and a category of finitely generated modules over a q-W algebra.
\end{abstract}


\maketitle

\section{Introduction}

\setcounter{equation}{0}

Let $G$ be a complex simple connected simply connected algebraic group with Lie algebra $\g$, $B$ a Borel subgroup of $G$, $\b$ the Lie algebra of $B$. Denote by $U(\g)$ the universal enveloping algebra of $\g$. Let $\lambda$ be a weight of $\g$, $M_\lambda$ the Verma module over $\g$ with highest weight $\lambda$ with respect to the system of positive roots of the pair $(\g,\b)$. Denote by $I_\lambda$ the annihilator of $M_\lambda$ in $U(\g)$, and let $U(\g)^\lambda=U(\g)/I_\lambda$. note that $I_\lambda$ is generated by a maximal ideal of the center $Z(U(\g))$ of $U(\g)$ which is the kernel of a character $\chi_\lambda:Z(U(\g))\rightarrow \mathbb{C}$. By the celebrated Beilinson--Bernstein theorem, if $\lambda$ is regular dominant then the category of $U(\g)^\lambda$--modules is equivalent to the category of modules over the sheaf $\mathcal{D}^\lambda_{G/B}$ of $\lambda$--twisted differential operators on the flag variety $G/B$ which are quasi--coherent over the sheaf of regular functions $\mathbb{C}[G/B]$ on $G/B$. The functor providing the equivalence is simply the global section functor.

This result was generalized to the case of quantum groups in \cite{Kr, Tan}. The main observation used in \cite{Kr} is that $\mathcal{D}_\lambda$ can be regarded as a quantization of the $\lambda$--twisted cotangent bundle $T^*(G/B)_\lambda$ which a symplectic leaf in the quotient $(T^*G)/B$ of the symplectic variety $T^*G$, equipped with the canonical symplectic structure of the cotangent bundle, by the Hamiltonian action induced by the $B$--action by right translations on $G$. Note that $\lambda$ naturally gives rise to a character $\lambda:\b \rightarrow \mathbb{C}$, and $T^*(G/B)_\lambda$ corresponds to the value $\lambda\in \b^*$ of the moment map $\mu:T^*G\rightarrow \b^*$ for the $B$--action, $T^*(G/B)_\lambda=\mu^{-1}(\lambda)/B$. Using this observation at the quantum level one can replace the category of  $\mathcal{D}^\lambda_{G/B}$--modules with a category $\mathcal{D}_{B}^\lambda$ of modules over the sheaf of differential operators $\mathcal{D}_G$ on $G$ which are equivariant with respect to a left $B$--action. Objects of this category are $\mathcal{D}_G$--modules $M$ equipped with the structure of $B$--modules in such a way that the action map $\mathcal{D}_G\otimes M \rightarrow M$ is a morphism of $B$--modules, where the action of $B$ on $\mathcal{D}_G$ is induced by the action on $G$ by right translations, and the differential of the action of $B$ on $M$ coincides with the action of the Lie algebra $\b$ on the tensor product $M\otimes \mathbb{C}_\lambda$, where $\b$ acts on $M$ via the natural embedding $\b \rightarrow \mathcal{D}_G$, and $\mathbb{C}_\lambda$ is the one--dimensional representation of $\b$ corresponding to the character $\lambda$. The Beilinson--Bernstein localization theorem for equivariant $\mathcal{D}_G$--modules was already formulated in \cite{BBer} (see also \cite{Kash} for some further details).

Note that $T^*G$ is naturally equipped with a $G$--action induced by the $G$--action by left translations on $G$. This action also preserves the canonical symplectic structure on $T^*G$ and commutes with the right $B$--action. Hence it induces a Hamiltonian $G$--action on $(T^*G)/B$ and on all its symplectic leaves. In particular, the natural $G$--action on  $T^*(G/B)_\lambda$ is Hamiltonian. One can restrict this action to various subgroups of $G$. Let $N$ be such a subgroup with Lie algebra $\n$ equipped with a character $\chi:\n\rightarrow \mathbb{C}$. Similarly to the case of $B$--equaivariant $\mathcal{D}_G$--modules one can consider the category $_{N}^{\chi}\mathcal{D}^\lambda_{G/B}$ of $N$--equivariant $\mathcal{D}^\lambda_{G/B}$--modules. By Beilinson--Bernstein localization theorem this category is equivalent to the category $_{N}^{~~\chi} U(\g)^\lambda-{\rm mod}$ of equivariant $(\g,N)$--modules on which the center $Z(U(\g))$ acts by the character $\chi_\lambda$. This category is defined similarly to the category $\mathcal{D}_{B}^\lambda$. Its objects are left $\g$--modules $V$ equipped with the structure of left $N$--modules in such a way that the action map
$\g \otimes V \rightarrow V$ is a morphism of $N$--modules, where the action of $N$ on $\g$ is induced by the adjoint representation, and the differential of the action of $N$ on $V$ coincides with the action of the Lie algebra $\n$ on the tensor product $V\otimes \mathbb{C}_\chi$, where $\n$ acts on $V$ via the natural embedding $\b \rightarrow \g$, and $\mathbb{C}_\chi$ is the one--dimensional representation of $\n$ corresponding to the character $\chi$.

Now let $\mu_1:T^*(G/B)_\lambda\rightarrow \n^*$ be the moment map corresponding to the Hamiltonian group action of $N$ on $T^*(G/B)_\lambda$, and $_{\chi}T^*(G/B)_\lambda=\mu_1^{-1}(\chi)/N$ the corresponding reduced Poisson manifold. Following the philosophy presented before in case of equivariant $\mathcal{D}_G$--modules one can expect that the category $_{N}^{~~\chi}\mathcal{D}^\lambda_{G/B}$ is equivalent to the category of $\mathcal{D}$--modules related to certain quantization of $_{\chi}T^*(G/B)_\lambda$, and the category $_{N}^{~~\chi} U(\g)^\lambda-{\rm mod}$ is equivalent to the category of modules over an associative algebra $^\chi U(\g)^\lambda$ which is a quantization of $_{\chi}T^*(G/B)_\lambda$.
Putting the two equivariance conditions together this would yield an equivalence between a category $_{N}^{~~\chi}\mathcal{D}_{B}^\lambda$ of $\mathcal{D}_G$--modules equipped with the two equivariance conditions with respect to actions of $B$ and $N$ and a category of $^\chi U(\g)^\lambda$--modules.

Such equivalence was established, for instance, in case of modules over $W$--algebras in \cite{Kr1} when the subgroup $N$ and its character $\chi$ are chosen in such a way that $^\chi U(\g)^\lambda$ is a quotient of a finitely generated W--algebra over a central ideal. In this paper we are going to obtain a similar categorial equivalence in case of q-W--algebras introduced in \cite{S10}.

The definition of q-W--algebras is given in terms of quantum groups and we shall need an analogue of Beilinson--Bernstein localization for quantum groups.
First of all there is a natural analogue of the algebra of differential operators on $G$ for quantum groups called the Heisenberg double $\mathcal{D}_q$ (see \cite{dual}). $\mathcal{D}_q$ is a smash product of a quantum group $U_q(\g)$ and of the dual Hopf algebra generated by matrix elements of finite--dimensional representations of the quantum group. Similarly to the case of Lie algebras one can consider the category of $\mathcal{D}_q$--modules which are equivariant, in a sense similar to the Lie algebra case, with respect to a locally finite action of a quantum group analogue $U_q(\b_+)$ of the universal enveloping algebra of a Borel sublagebra, $U_q(\b_+)$ being equipped with a character $\lambda$ as well. The main statement of \cite{Kr} is that if $\lambda$ is regular dominant the category of such $\mathcal{D}_q$--modules is equivalent to the category of $U_q(\g)^\lambda$--modules, where $U_q(\g)^\lambda=U_q(\g)^{fin}/J_\lambda$, and $U_q(\g)^{fin}$ is the locally finite part with respect to the adjoint action of the Hopf algebra $U_q(\g)$ on itself, $J_\lambda$ is the annihilator of the Verma module with highest weight $\lambda$ in $U_q(\g)^{fin}$.

In Section \ref{Dmod} we give a quantum group analogue of the localization theorem for the category $_{N}^{~~\chi}\mathcal{D}_{B}^\lambda$. Our construction is a straightforward generalization of the classical result. Such an easy generalization is possible because the Heisenberg double is equipped with natural analogues of the $G$--actions on the algebra of differential operators on $G$ induced by left and right translations on $G$.

This result can be applied in case of q-W--algebras if a quantum analogue of the group $N$ and of its character are chosen in a proper way. Appropriate subalgebras $U_q^s(\m_+)$ of $U_q(\g)$ with characters $\chi_q^s$ were defined in terms of certain new realizations $U_q^s(\g)$ of the quantum group $U_q(\g)$ associated to Weyl group elements $s$ of the Weyl group $W$ of $\g$. The definition of subalgebras $U_q^s(\m_+)$ requires a deep study of the algebraic structure of $U_q(\g)$ presented in \cite{S10}. We recall the main results of \cite{S10} in Sections \ref{wqreal}--\ref{wpsred}. However, the definition of the category of $U_q^s(\m_+)$--equivariant modules over $U_q^s(\g)$ requires some further investigation presented in Sections \ref{qplproups}--\ref{wpsred}. The problem is that a proper definition of this category formulated in Section \ref{skryabin} can only be given in terms of the locally finite part $U_q^s(\g)^{fin}$ of $U_q^s(\g)$, and the definition of the corresponding q-W--algebras associated to characters $\chi_q^s:U_q^s(\m_+)\rightarrow \mathbb{C}$ given in Section \ref{qplproups} in terms of $U_q^s(\g)^{fin}$ also becomes more complicated comparing to the one suggested in \cite{S10}. The use of the locally finite part $U_q^s(\g)^{fin}$ is related to the fact that $U_q^s(\g)^{fin}$ is a deformation of the algebra of regular functions on the algebraic group $G$ which follows from Proposition \ref{locfin}. Implicitly this result is also contained in \cite{Jos}.

The most difficult part of our construction is the proof of the equivalence between the category of finitely generated modules over $U_q^s(\g)^{fin}$ equivariant over $U_q^s(\m_+)$  and the category of finitely generated modules over the corresponding q-W--algebra $W_q^s(G)$ which can be regarded as an equivariant version of Skryabin equivalence for quantum groups (see Appendix to \cite{Pr}). We use the idea of the proof of a similar fact for W--algebras as it appears in \cite{GG}. However, technical difficulties in case of quantum groups become obscure. Our proof is presented in Section \ref{skryabin}. It heavily relies on the behavior of all ingredients of the construction in the classical limit $q\rightarrow 1$. In particular, the key step is to use the cross--section theorem for the action of a unipotent algebraic subgroup $N\subset G$ on a subvariety of $G$ obtained in \cite{S6}. Let $U(\m_+)$ be $q=1$ specialization of the $U_q^s(\m_+)$. The cross--section theorem implies in particular that as a $U(\m_+)$--module the $q=1$ specialization of any $U_q^s(\m_+)$--equivariant $U_q^s(\g)^{fin}$--module $V$ is isomorphic to the space of homomorphisms ${\rm hom}_{\mathbb{C}}(U(\m_+),V')$ of $U(\m_+)$ into a vector space $V'$ vanishing on some power of the natural augmentation ideal of $U(\m_+)$.

The quantum group analogue of the localization theorem for the category
$_{N}^{~~\chi} U(\g)^\lambda-{\rm mod}$ easily gives an equivalence between a category of modules over $\mathcal{D}_q$ equivariant with respect to a $U_q(\b_+)$--action and to a $U_q^s(\m_+)$--action and the category of $U_q^s(\m_+)$--equivariant modules over $U_q^s(\g)^{fin}$ with central character $\chi_\lambda$. This equivalence together with the equivariant Skryabin equivalence for quantum groups yield an equivalence between a category of finitely generated modules over $\mathcal{D}_q$ equivariant with respect to a $U_q(\b_+)$--action and to a $U_q^s(\m_+)$--action and the category of finitely generated modules over the quotient $W_q^s(G)_\lambda$ of the corresponding q-W--algebra $W_q^s(G)$ by a central ideal. This agrees with the general philosophy that $W_q^s(G)_\lambda$, or more generally $W_q^s(G)$, can be regarded as a quantization of the algebra of regular functions on a reduced Poisson manifold. In case of the algebra $W_q^s(G)$ the corresponding manifold is an algebraic group analogue of Slodowy slices associated to Weyl group element $s$ (see Theorem \ref{var}). Such slices transversal to conjugacy classes in $G$ were defined in \cite{S6}.

{\bf Acknowledgement}

The author is grateful to Y. Kremnizer for useful
discussions.


\setcounter{equation}{0}
\setcounter{theorem}{0}

\section{Notation}\label{notation}

Fix the notation used throughout the text.
Let $G$ be a
connected finite--dimensional complex simple Lie group, $
{\frak g}$ its Lie algebra. Fix a Cartan subalgebra ${\frak h}\subset {\frak
g}\ $and let $\Delta $ be the set of roots of $\left( {\frak g},{\frak h}
\right)$.  Let $\alpha_i,~i=1,\ldots, l,~~l=rank({\frak g})$ be a system of
simple roots, $\Delta_+=\{ \beta_1, \ldots ,\beta_N \}$
the set of positive roots.
Let $H_1,\ldots ,H_l$ be the set of simple root generators of $\frak h$.

Let $a_{ij}$ be the corresponding Cartan matrix,
and let $d_1,\ldots , d_l$ be coprime positive integers such that the matrix
$b_{ij}=d_ia_{ij}$ is symmetric. There exists a unique non--degenerate invariant
symmetric bilinear form $\left( ,\right) $ on ${\frak g}$ such that
$(H_i , H_j)=d_j^{-1}a_{ij}$. It induces an isomorphism of vector spaces
${\frak h}\simeq {\frak h}^*$ under which $\alpha_i \in {\frak h}^*$ corresponds
to $d_iH_i \in {\frak h}$. We denote by $\alpha^\vee$ the element of $\frak h$ that
corresponds to $\alpha \in {\frak h}^*$ under this isomorphism.
The induced bilinear form on ${\frak h}^*$ is given by
$(\alpha_i , \alpha_j)=b_{ij}$.

Let $W$ be the Weyl group of the root system $\Delta$. $W$ is the subgroup of $GL({\frak h})$
generated by the fundamental reflections $s_1,\ldots ,s_l$,
$$
s_i(h)=h-\alpha_i(h)H_i,~~h\in{\frak h}.
$$
The action of $W$ preserves the bilinear form $(,)$ on $\frak h$.
We denote a representative of $w\in W$ in $G$ by
the same letter. For $w\in W, g\in G$ we write $w(g)=wgw^{-1}$.
For any root $\alpha\in \Delta$ we also denote by $s_\alpha$ the corresponding reflection.

For every element $w\in W$ one can introduce the set $\Delta_w=\{\alpha \in \Delta_+: w(\alpha)\in -\Delta_+\}$, and the number of the elements in the set $\Delta_w$ is equal to the length $l(w)$ of the element $w$ with respect to the system $\Gamma$ of simple roots in $\Delta_+$.

Let ${{\frak b}_+}$ be the positive Borel subalgebra and ${\frak b}_-$
the opposite Borel subalgebra; let ${\frak n}_+=[{{\frak b}_+},{{\frak b}_+}]$ and $%
{\frak n}_-=[{\frak b}_-,{\frak b}_-]$ be their
nilradicals. Let $H=\exp {\frak h},N_+=\exp {{\frak n}_+},
N_-=\exp {\frak n}_-,B_+=HN_+,B_-=HN_-$ be
the Cartan subgroup, the maximal unipotent subgroups and the Borel subgroups
of $G$ which correspond to the Lie subalgebras ${\frak h},{{\frak n}_+},%
{\frak n}_-,{\frak b}_+$ and ${\frak b}_-,$ respectively.

We identify $\frak g$ and its dual by means of the canonical invariant bilinear form.
Then the coadjoint
action of $G$ on ${\frak g}^*$ is naturally identified with the adjoint one. We also identify
${{\frak n}_+}^*\cong {\frak n}_-,~{{\frak b}_+}^*\cong {\frak b}_-$.

Let ${\frak g}_\beta$ be the root subspace corresponding to a root $\beta \in \Delta$,
${\frak g}_\beta=\{ x\in {\frak g}| [h,x]=\beta(h)x \mbox{ for every }h\in {\frak h}\}$.
${\frak g}_\beta\subset {\frak g}$ is a one--dimensional subspace.
It is well known that for $\alpha\neq -\beta$ the root subspaces ${\frak g}_\alpha$ and ${\frak g}_\beta$ are orthogonal with respect
to the canonical invariant bilinear form. Moreover ${\frak g}_\alpha$ and ${\frak g}_{-\alpha}$
are non--degenerately paired by this form.

Root vectors $X_{\alpha}\in {\frak g}_\alpha$ satisfy the following relations:
$$
[X_\alpha,X_{-\alpha}]=(X_\alpha,X_{-\alpha})\alpha^\vee.
$$

Note also that in this paper we denote by $\mathbb{N}$ the set of nonnegative integer numbers, $\mathbb{N}=\{0,1,\ldots \}$.


\section{Quantum groups}

\setcounter{equation}{0}
\setcounter{theorem}{0}

In this paper we shall consider various specializations of the standard Drinfeld-Jimbo quantum group $U_h({\frak g})$ defined over the ring of formal power series ${\Bbb C}[[h]]$, where $h$ is an indeterminate.
We follow the notation of \cite{ChP}.

Let $V$ be a ${\Bbb C}[[h]]$--module equipped with the $h$--adic
topology. This topology is characterized by requiring that
$\{ h^nV ~|~n\geq 0\}$ is a base of the neighborhoods of $0$ in $V$, and that translations
in $V$ are continuous.

A topological Hopf algebra over ${\Bbb C}[[h]]$ is a complete ${\Bbb C}[[h]]$--module $A$
equipped with a structure of ${\Bbb C}[[h]]$--Hopf algebra (see \cite{ChP}, Definition 4.3.1),
the algebraic tensor products entering the axioms of the Hopf algebra are replaced by their
completions in the $h$--adic topology.
Let $\mu , \imath , \Delta , \varepsilon , S$ be the multiplication, the unit, the comultiplication,
the counit and the antipode of $A$, respectively.

The standard quantum group $U_h({\frak g})$ associated to a complex finite--dimensional simple Lie algebra
$\frak g$ is a topological Hopf algebra over ${\Bbb C}[[h]]$ topologically generated by elements
$H_i,~X_i^+,~X_i^-,~i=1,\ldots ,l$, subject to the following defining relations:
$$
\begin{array}{l}
[H_i,H_j]=0,~~ [H_i,X_j^\pm]=\pm a_{ij}X_j^\pm, ~~X_i^+X_j^- -X_j^-X_i^+ = \delta _{i,j}{K_i -K_i^{-1} \over q_i -q_i^{-1}},\\
\\
\sum_{r=0}^{1-a_{ij}}(-1)^r
\left[ \begin{array}{c} 1-a_{ij} \\ r \end{array} \right]_{q_i}
(X_i^\pm )^{1-a_{ij}-r}X_j^\pm(X_i^\pm)^r =0 ,~ i \neq j ,
\end{array}
$$
where
$$
K_i=e^{d_ihH_i},~~e^h=q,~~q_i=q^{d_i}=e^{d_ih},
$$
$$
\left[ \begin{array}{c} m \\ n \end{array} \right]_q={[m]_q! \over [n]_q![n-m]_q!} ,~
[n]_q!=[n]_q\ldots [1]_q ,~ [n]_q={q^n - q^{-n} \over q-q^{-1} },
$$
with comultiplication defined by
$$
\Delta_h(H_i)=H_i\otimes 1+1\otimes H_i,~~
\Delta_h(X_i^+)=X_i^+\otimes K_i+1\otimes X_i^+,~~\Delta_h(X_i^-)=X_i^-\otimes 1 +K_i^{-1}\otimes X_i^-,
$$
antipode defined by
$$
S_h(H_i)=-H_i,~~S_h(X_i^+)=-X_i^+K_i^{-1},~~S_h(X_i^-)=-K_iX_i^-,
$$
and counit defined by
$$
\varepsilon_h(H_i)=\varepsilon_h(X_i^\pm)=0.
$$

We shall also use the weight--type generators
$$
Y_i=\sum_{j=1}^l d_i(a^{-1})_{ij}H_j,
$$
and the elements $L_i=e^{hY_i}$.

The Hopf algebra $U_h({\frak g})$ is a quantization of the standard bialgebra structure on $\frak g$ in the sense that $U_h({\frak g})/hU_h({\frak g})=U({\frak g}),~~ \Delta_h=\Delta~(\mbox{mod }h)$, where $\Delta$ is
the standard comultiplication on $U({\frak g})$, and
$$
{\Delta_h -\Delta_h^{opp} \over h}~(\mbox{mod }h)=\delta.
$$
Here
$\delta: {\frak g}\rightarrow {\frak g}\otimes {\frak g}$ is the standard cocycle on $\frak g$, and $\Delta^{opp}_h=\sigma \Delta_h$, $\sigma$ is the permutation in $U_h({\frak g})^{\otimes 2}$,
$\sigma (x\otimes y)=y\otimes x$.
Recall that
$$
\delta (x)=({\rm ad}_x\otimes 1+1\otimes {\rm ad}_x)2r_+,~~ r_+\in {\frak g}\otimes {\frak g},
$$
\begin{equation}\label{rcl}
r_+=\frac 12 \sum_{i=1}^lY_i \otimes H_i + \sum_{\beta \in \Delta_+}(X_{\beta},X_{-\beta})^{-1} X_{\beta}\otimes X_{-\beta}.
\end{equation}
Here $X_{\pm \beta}\in {\frak g}_{\pm \beta}$ are root vectors of $\frak g$.
The element $r_+\in {\frak g}\otimes {\frak g}$ is called a classical r--matrix.

$U_h({\frak g})$ is a quasitriangular Hopf algebra, i.e. there exists an invertible element
${\mathcal R}\in U_h({\frak g})\otimes U_h({\frak g})$, called a universal R--matrix, such that
\begin{equation}\label{quasitr}
\Delta^{opp}_h(a)={\mathcal R}\Delta_h(a){\mathcal R}^{-1}\mbox{ for all } a\in U_h({\frak g}).
\end{equation}

We recall an explicit description of the element ${\mathcal R}$.
Firstly, one can define root vectors of $U_h({\frak g})$  in terms of  a braid group action on $U_h({\frak g})$ (see \cite{ChP}). Let $m_{ij}$, $i\neq j$ be equal to $2,3,4,6$ if $a_{ij}a_{ji}$ is equal to $0,1,2,3$, respectively. The braid group $\mathcal{B}_\g$ associated to $\g$ has generators $T_i$, $i=1,\ldots, l$, and defining relations
$$
T_iT_jT_iT_j\ldots=T_jT_iT_jT_i\ldots
$$
for all $i\neq j$, where there are $m_{ij}$ $T$'s on each side of the equation.

$\mathcal{B}_\g$ acts by algebra automorphisms of $U_h({\frak g})$ as follows:
\begin{eqnarray*}
T_i(X_i^+)=-X_i^-e^{hd_iH_i},~T_i(X_i^-)=-e^{-hd_iH_i}X_i^+,~T_i(H_j)=H_j-a_{ji}H_i, \\
\\
T_i(X_j^+)=\sum_{r=0}^{-a_{ij}}(-1)^{r-a_{ij}}q_i^{-r}
(X_i^+ )^{(-a_{ij}-r)}X_j^+(X_i^+)^{(r)},~i\neq j,\\
\\
T_i(X_j^-)=\sum_{r=0}^{-a_{ij}}(-1)^{r-a_{ij}}q_i^{r}
(X_i^-)^{(r)}X_j^-(X_i^-)^{(-a_{ij}-r)},~i\neq j,
\end{eqnarray*}
where
$$
(X_i^+)^{(r)}=\frac{(X_i^+)^{r}}{[r]_{q_i}!},~(X_i^-)^{(r)}=\frac{(X_i^-)^{r}}{[r]_{q_i}!},~r\geq 0,~i=1,\ldots,l.
$$

Recall that an ordering of a set of positive roots $\Delta_+$ is called normal if all simple roots are written in an arbitrary order, and for any three roots $\alpha,~\beta,~\gamma$ such that
$\gamma=\alpha+\beta$ we have either $\alpha<\gamma<\beta$ or $\beta<\gamma<\alpha$.

For any reduced decomposition $w_0=s_{i_1}\ldots s_{i_D}$ of the longest element $w_0$ of the Weyl group $W$ of $\g$ the ordering
$$
\beta_1=\alpha_{i_1},\beta_2=s_{i_1}\alpha_{i_2},\ldots,\beta_D=s_{i_1}\ldots s_{i_{D-1}}\alpha_{i_D}
$$
is a normal ordering in $\Delta_+$, and there is a one--to--one correspondence between normal orderings of $\Delta_+$ and reduced decompositions of $w_0$ (see \cite{Z1}).

Fix a reduced decomposition $w_0=s_{i_1}\ldots s_{i_D}$ of $w_0$ and define the corresponding root vectors in $U_h({\frak g})$ by
\begin{equation}\label{rootvect}
X_{\beta_k}^\pm=T_{i_1}\ldots T_{i_{k-1}}X_{i_k}^\pm.
\end{equation}

The root vectors $X_{\beta}^+$ satisfy the following relations:
\begin{equation}\label{qcom}
X_{\alpha}^+X_{\beta}^+ - q^{(\alpha,\beta)}X_{\beta}^+X_{\alpha}^+= \sum_{\alpha<\delta_1<\ldots<\delta_n<\beta}C(k_1,\ldots,k_n)
{(X_{\delta_1}^+)}^{(k_1)}{(X_{\delta_2}^+)}^{(k_2)}\ldots {(X_{\delta_n}^+)}^{(k_n)},~\alpha<\beta,
\end{equation}
where for $\alpha \in \Delta_+$ we put ${(X_{\alpha}^\pm)}^{(k)}=\frac{(X_\alpha^\pm)^{k}}{[k]_{q_\alpha}!}$, $k\geq 0$, $q_\alpha =q^{d_i}$ if the positive root $\alpha$ is Weyl group conjugate to the simple root $\alpha_i$, $C(k_1,\ldots,k_n)\in {\Bbb C}[q,q^{-1}]$.

Note that by construction
$$
\begin{array}{l}
X_\beta^+~(\mbox{mod }h)=X_\beta \in {\frak g}_\beta,\\
\\
X_\beta^-~(\mbox{mod }h)=X_{-\beta} \in {\frak g}_{-\beta}
\end{array}
$$
are root vectors of $\frak g$.

Denote by $U_h({\frak n}_+),U_h({\frak n}_-)$ and $U_h(\h)$ the ${\Bbb C}[[h]]$--subalgebras of $U_h({\frak g})$ topologically generated by the
$X_i^+$, by the $X_i^-$ and by the $H_i$, respectively. For any $\alpha\in \Delta_+$ one has $X_{\alpha}^\pm\in U_h({\frak n}_\pm)$.

An explicit expression for $\mathcal R$ may be written by making use of the q--exponential
$$
exp_q(x)=\sum_{k=0}^\infty q^{\frac{1}{2}k(k+1)}{x^k \over [k]_q!}
$$
in terms of which the element $\mathcal R$ takes the form:
\begin{equation}\label{univr}
{\mathcal R}=exp\left[ h\sum_{i=1}^l(Y_i\otimes H_i)\right]\prod_{\beta}
exp_{q_{\beta}}[(1-q_{\beta}^{-2})X_{\beta}^+\otimes X_{\beta}^-],
\end{equation}
where
the product is over all the positive roots of $\frak g$, and the order of the terms is such that
the $\alpha$--term appears to the left of the $\beta$--term if $\alpha >\beta$ with respect to the normal
ordering of $\Delta_+$.

The r--matrix $r_+=\frac 12 h^{-1}({\mathcal R}-1\otimes 1)~~(\mbox{mod }h)$, which is the classical limit of $\mathcal R$,
coincides with the classical r--matrix (\ref{rcl}).

One can calculate the action of the comultiplication on the root vectors $X_{\beta_k}^\pm$ in terms of the universal R--matrix. For instance for $\Delta_h(X_{\beta_k}^+)$ one has
\begin{equation}\label{comult}
\Delta_h(X_{\beta_k}^+)=\widetilde{\mathcal R}^{-1}_{<\beta_k}(X_{\beta_k}^+\otimes e^{h\beta^\vee}+1\otimes X_{\beta_k}^+)\widetilde{\mathcal R}_{<\beta_k},
\end{equation}
where
$$
\widetilde{\mathcal R}_{<\beta_k}=\widetilde{\mathcal R}_{\beta_{k-1}}\ldots \widetilde{\mathcal R}_{\beta_1},~\widetilde{\mathcal R}_{\beta_r}=exp_{q_{\beta_r}}[(1-q_{\beta_r}^{-2})X_{\beta_r}^+\otimes X_{\beta_r}^-].
$$


\section{Realizations of quantum groups associated to Weyl group elements}\label{wqreal}

\setcounter{equation}{0}
\setcounter{theorem}{0}

Our main objects of study are certain specializations of realizations of quantum groups associated to Weyl group elements.
Let $s$ be an element of the Weyl group $W$ of the pair $(\g,\h)$, and $\h'$ the orthogonal complement, with respect to the Killing form, to the subspace of $\h$ fixed by the natural action of $s$ on $\h$. Let $\h'^*$ be the image of $\h'$ in $\h^*$ under the identification $\h^*\simeq \h$ induced by the canonical bilinear form on $\g$.
The restriction of the natural action of $s$ on $\h^*$ to the subspace $\h'^*$ has no fixed points. Therefore one can define the Cayley transform ${1+s \over 1-s }P_{{\h'}^*}$ of the restriction of $s$ to ${\h'}^*$, where $P_{{\h'}^*}$ is the orthogonal projection operator onto ${{\h'}^*}$ in $\h^*$, with respect to the Killing form.

Let
$U_h^{s}({\frak g})$ be the topological algebra over ${\Bbb C}[[h]]$ topologically generated by elements
$e_i , f_i , H_i,~i=1, \ldots l$ subject to the relations:
$$
\begin{array}{l}
[H_i,H_j]=0,~~ [H_i,e_j]=a_{ij}e_j, ~~ [H_i,f_j]=-a_{ij}f_j,~~e_i f_j -q^{ c_{ij}} f_j e_i = \delta _{i,j}{K_i -K_i^{-1} \over q_i -q_i^{-1}} ,\\
\\
c_{ij}=\left( {1+s \over 1-s }P_{{\h'}^*}\alpha_i , \alpha_j \right),~~K_i=e^{d_ihH_i}, \\
\\
\sum_{r=0}^{1-a_{ij}}(-1)^r q^{r c_{ij}}
\left[ \begin{array}{c} 1-a_{ij} \\ r \end{array} \right]_{q_i}
(e_i )^{1-a_{ij}-r}e_j (e_i)^r =0 ,~ i \neq j , \\
\\
\sum_{r=0}^{1-a_{ij}}(-1)^r q^{r c_{ij}}
\left[ \begin{array}{c} 1-a_{ij} \\ r \end{array} \right]_{q_i}
(f_i )^{1-a_{ij}-r}f_j (f_i)^r =0 ,~ i \neq j .
\end{array}
$$

\begin{proposition} {\bf (\cite{S10}, Theorem 4.1)} \label{newreal}
For every solution $n_{ij}\in {\Bbb C},~i,j=1,\ldots ,l$ of equations
\begin{equation}\label{eqpi}
d_jn_{ij}-d_in_{ji}=c_{ij}
\end{equation}
there exists an algebra
isomorphism $\psi_{\{ n\}} : U_h^{s}({\frak g}) \rightarrow
U_h({\frak g})$ defined  by the formulas:
$$
\psi_{\{ n\}}(e_i)=X_i^+ \prod_{p=1}^lL_p^{n_{ip}},~~
\psi_{\{ n\}}(f_i)=\prod_{p=1}^lL_p^{-n_{ip}}X_i^- ,~~
\psi_{\{ n\}}(H_i)=H_i .
$$
\end{proposition}

The general solution of equation (\ref{eqpi})
is given by
\begin{equation}\label{eq3}
n_{ij}=\frac 1{2d_j} (c_{ij} + {s_{ij}}),
\end{equation}
where $s_{ij}=s_{ji}$.

The algebra $U_h^{s}({\frak g})$ is called the realization of the quantum group $U_h({\frak g})$ corresponding to the element $s\in W$.

Now we recall the definition of certain normal orderings of root systems associated to Weyl group elements introduced in \cite{S10}. These orderings will play a crucial role in the definition of q-W--algebras.

Let $s$ be an element of the Weyl group $W$ of the pair $(\g,\h)$.
By Theorem C in \cite{C} $s$ can be represented as a product of two involutions,
\begin{equation}\label{inv}
s=s^1s^2,
\end{equation}
where $s^1=s_{\gamma_1}\ldots s_{\gamma_n}$, $s^2=s_{\gamma_{n+1}}\ldots s_{\gamma_{l'}}$, the roots in each of the sets $\gamma_1, \ldots, \gamma_n$ and ${\gamma_{n+1}}, \ldots, {\gamma_{l'}}$ are positive and mutually orthogonal, and
the roots $\gamma_1, \ldots, \gamma_{l'}$ form a linear basis of $\h'^*$.

\begin{proposition}{\bf (\cite{S10}, Peoposition 5.1)}\label{pord}
Let $s\in W$ be an element of the Weyl group $W$ of the pair $(\g,\h)$, $\Delta$ the root system of the pair $(\g,\h)$ Then there is a system of positive roots $\Delta_+$ associated to (the conjugacy class of) $s$ such that
the decomposition $s=s^1s^2$ is reduced in the sense that ${l}(s)={l}(s^2)+{l}(s^1)$, where ${l}(\cdot)$ is the length function in $W$ with respect to the system of simple roots in $\Delta_+$, and $\Delta_{s}=\Delta_{s^{2}}\bigcup s^2(\Delta_{s^{1}})$, $\Delta_{s^{-1}}=\Delta_{s^{1}}\bigcup s^1(\Delta_{s^{2}})$ (disjoint unions). Here $s^1,s^2$ are the involutions entering decomposition (\ref{inv}), $s^1=s_{\gamma_1}\ldots s_{\gamma_n}$, $s^2=s_{\gamma_{n+1}}\ldots s_{\gamma_{l'}}$, the roots in each of the sets $\gamma_1, \ldots, \gamma_n$ and ${\gamma_{n+1}},\ldots, {\gamma_{l'}}$ are positive and mutually orthogonal.

Moreover, there is a normal ordering of the root system $\Delta_+$ such that
$$
\gamma_1< \gamma_2<\gamma_3<\ldots< \gamma_n<\gamma_{n+1}<\ldots< \gamma_{l'}
$$
with respect to this normal ordering.

The length of the ordered segment $\Delta_{\m_+}\subset \Delta$,
\begin{eqnarray}
\Delta_{\m_+}=\{\beta\in \Delta_+: \gamma_1\leq \beta\leq \gamma_{l'}\}
\end{eqnarray}
is equal to
\begin{equation}\label{dimm}
D-(\frac{l(s)-l'}{2}+D_0),
\end{equation}
where $D$ is the number of roots in $\Delta_+$, $l(s)$ is the length of $s$ and $D_0$ is the number of positive roots fixed by the action of $s$.
\end{proposition}

The system of positive roots $\Delta_+$ equipped with the normal ordering as in the previous proposition is called the normally ordered system of positive roots associated to the (conjugacy class of) the Weyl group element $s\in W$.

Let
\begin{equation*}
\overline{\Delta}_0=\{\alpha\in \Delta|s(\alpha)=\alpha\},
\end{equation*}
and $\Gamma$ the set of simple roots in $\Delta_+$.
We shall need the parabolic subalgebra $\p$ of $\g$ and the parabolic subgroup $P$ associated to the subset $\Gamma_0=\Gamma\bigcap \overline{\Delta}_{0}$ of simple roots. Let $\n$ and $\l$ be the nilradical and the Levi factor of $\p$, $N$ and $L$ the unipotent radical and the Levi factor of $P$, respectively. Note that we have natural inclusions of Lie algebras $\p\supset\b_+\supset\n$, where $\b_+$ is the Borel subalgebra of $\g$ corresponding to the system $\Gamma$ of simple roots, and $\overline{\Delta}_{0}$ is the root system of the reductive Lie algebra $\l$. We also denote by $\opn$ the nilpotent subalgebra opposite to $\n$.
The linear subspace of $\g$ generated by the root vectors $X_{\alpha}$ ($X_{-\alpha}$), $\alpha\in \Delta_{\m_+}$ is in fact a Lie subalgebra ${\m_+}\subset \g$ (${\m_-}\subset \g$).

Denote by $U_h^{s}({\frak n}_\pm) $ the subalgebra in $U_h^{s}({\frak g})$ generated by
$e_i ~(f_i) ,i=1, \ldots, l$.
Let $U_h^{s}({\frak h})$ be the subalgebra in $U_h^{s}({\frak g})$ generated by $H_i,~i=1,\ldots ,l$.

We shall construct analogues of root vectors for $U_h^{s}({\frak g})$.
It is convenient to introduce an operator $K\in {\rm End}~{\frak h}$ defined by
\begin{equation}\label{Kdef}
KH_i=\sum_{j=1}^l{n_{ij} \over d_i}Y_j.
\end{equation}

\begin{proposition}{\bf (\cite{S10}, Proposition 4.2, Theorem 6.1)}\label{rootss}
Let $s\in W$ be an element of the Weyl group $W$ of the pair $(\g,\h)$, $\Delta$ the root system of the pair $(\g,\h)$. Let $U_h^{s}({\frak g})$ be the realization of the quantum group $U_h({\frak g})$ associated to $s$.

For any solution of equation (\ref{eqpi}) and any normal ordering of the root system $\Delta_+$
the elements $e_{\beta}=\psi_{\{ n\}}^{-1}(X_{\beta}^+e^{hK\beta^\vee})$ and
$f_{\beta}=\psi_{\{ n\}}^{-1}(e^{-hK\beta^\vee}X_{\beta}^-),~\beta \in \Delta_+$
lie in the subalgebras $U_h^{s}({\frak n}_+)$ and $U_h^{s}({\frak n}_-)$, respectively.

Let $\Delta_+$ be the system of positive roots associated to $s$.
Let $e_\beta\in U_h^{s}({\n_+})$, $\beta \in \Delta_+$ be the root vectors constructed with the help of the normal ordering of $\Delta_+$ associated to $s$.

Then elements $e_\beta\in U_h^{s}({\n_+})$, $\beta \in \Delta_{\m_+}$ generate a subalgebra $U_h^{s}({\frak m}_+)\subset U_h^{s}({\frak g})$ such that $U_h^{s}({\frak m}_+)/hU_h^{s}({\frak m}_+)\simeq U({\m_+})$, where ${\m_+}$ is the Lie subalgebra of $\g$ generated by the root vectors $X_\alpha$, $\alpha\in \Delta_{\m_+}$.
\end{proposition}

The realizations $U_h^{s}({\frak g})$ of the quantum group $U_h({\frak g})$
are connected with quantizations of some nonstandard bialgebra structures on $\frak g$. At the quantum level
changing bialgebra structure corresponds to the so--called Drinfeld twist (see \cite{S10}, Section 4).

Equip $U_h^{s}({\frak g})$ with the comultiplication $\Delta_{s}$ given by
$$
\begin{array}{c}
\Delta_{s}(H_i)=H_i\otimes 1+1\otimes H_i,\\
\\
\Delta_{s}(e_i)=e_i\otimes e^{hd_i({2 \over 1-s}P_{{\h'}}+P_{{\h'}^\perp})H_i}+1\otimes e_i,~~
\Delta_{s}(f_i)=f_i\otimes e^{-hd_i{1+s \over 1-s}P_{{\h'}}H_i}+e^{-hd_iH_i}\otimes f_i,
\end{array}
$$
where $P_{{\h'}^\perp}$ is the orthogonal projection operator onto the orthogonal complement ${{\h'}^\perp}$ to $\h'$ in $\h$ with respect to the Killing form, and the antipode $S_s(x)$ given by
$$
S_s(e_i)=-e_ie^{-hd_i({2 \over 1-s}P_{{\h'}}+P_{{\h'}^\perp})H_i},~S_s(f_i)=-e^{hd_iH_i}f_ie^{hd_i{1+s \over 1-s}P_{{\h'}}H_i},~S_s(H_i)=-H_i.
$$

Then $U_h^{s}({\frak g})$ becomes a quasitriangular topological Hopf algebra with the universal R--matrix
${\mathcal R}^{s}$,
\begin{equation}\label{rmatrspi}
\begin{array}{l}
{\mathcal R}^{s}=exp\left[ h(\sum_{i=1}^l(Y_i\otimes H_i)+
\sum_{i=1}^l {1+s \over 1-s }P_{{\h'}}H_i\otimes Y_i) \right]\times \\
\prod_{\beta}
exp_{q_{\beta}}[(1-q_{\beta}^{-2})e_{\beta} \otimes
e^{h{1+s \over 1-s}P_{{\h'}} \beta^\vee}f_{\beta}],
\end{array}
\end{equation}
where $P_{{\h'}}$ is the orthogonal projection operator onto $\h'$ in $\h$ with respect to the Killing form.

The element ${\mathcal R}^{s}$ may be also represented in the form
\begin{equation}\label{rmatrspi'}
\begin{array}{l}
{\mathcal R}^{s}=
\prod_{\beta}
exp_{q_{\beta}}[(1-q_{\beta}^{-2})e_{\beta}e^{-h({1+s \over 1-s}P_{{\h'}}+1)\beta^\vee}\otimes e^{h\beta^\vee}f_\beta]\times \\
exp\left[ h(\sum_{i=1}^l(Y_i\otimes H_i)+\sum_{i=1}^l {1+s \over 1-s }P_{{\h'}}H_i\otimes Y_i)\right] .
\end{array}
\end{equation}

Note that the Hopf algebra $U_h^{s}({\frak g})$ is a quantization of the bialgebra structure on $\frak g$
defined by the cocycle
\begin{equation}\label{cocycles}
\delta (x)=({\rm ad}_x\otimes 1+1\otimes {\rm ad}_x)2r^{s}_+,~~ r^{s}_+\in {\frak g}\otimes {\frak g},
\end{equation}
where $r^{s}_+=r_+ + \frac 12 \sum_{i=1}^l {1+s \over 1-s }P_{{\h'}}H_i\otimes Y_i$, and $r_+$ is given by (\ref{rcl}).

Using formula (\ref{comult}) and Proposition 4.3 in \cite{S10} one can also easily find that
\begin{equation}\label{comults}
\Delta_s(e_{\beta_k})=(\widetilde{\mathcal R}_{<\beta_k}^s)^{-1}(e_{\beta_k}\otimes e^{h{1+s \over 1-s}P_{{\h'}}\beta^\vee}+1\otimes e_{\beta_k})\widetilde{\mathcal R}^s_{<\beta_k},
\end{equation}
where
$$
\widetilde{\mathcal R}^s_{<\beta_k}=\widetilde{\mathcal R}^s_{\beta_{k-1}}\ldots \widetilde{\mathcal R}^s_{\beta_1},~\widetilde{\mathcal R}^s_{\beta_r}=exp_{q_{\beta_r}}[(1-q_{\beta_r}^{-2})e_{\beta_r}\otimes e^{h{1+s \over 1-s}P_{{\h'}}\beta^\vee}f_{\beta_r}],
$$
and
$$
(\widetilde{\mathcal R}_{<\beta_k}^s)^{-1}=(\widetilde{\mathcal R}_{\beta_{1}}^s)^{-1}\ldots (\widetilde{\mathcal R}^s_{\beta_{k-1}})^{-1},~(\widetilde{\mathcal R}^s_{\beta_r})^{-1}=exp_{q_{\beta_r^{-1}}}[(1-q_{\beta_r}^{2})e_{\beta_r}\otimes e^{h{1+s \over 1-s}P_{{\h'}}\beta^\vee}f_{\beta_r}].
$$

We shall actually need not the algebras $U_h^{s}({\frak g})$ themselves but some their specializations defined over certain rings and over the field of complex numbers. They are similar to the rational form, the restricted integral form and to its specialization for the standard quantum group $U_h(\g)$. The motivations of the definitions given below will be clear in Section \ref{qplproups}. The results below are slight modifications of similar statements for $U_h(\g)$, and we refer to \cite{ChP}, Ch. 9 for the proofs.

We start with the observation that by the results of Section 7 in \cite{S10} the numbers
\begin{equation}\label{rat}
p_{ij}=\left( {1+s \over 1-s }P_{{\h'}}Y_i,Y_j\right)+(Y_i,Y_j)
\end{equation}
are rational, $p_{ij}\in \mathbb{Q}$.
Denote by $d$ the smallest integer number divisible by all the denominators of the rational numbers $p_{ij}/2$, $i,j=1,\ldots, l$.

Let $U_q^{s}({\frak g})$ be the $\mathbb{C}(q^{\frac{1}{2d}})$-subalgebra of $U_h^{s}({\frak g})$ generated by the elements $e_i , f_i , t_i^{\pm 1}=\exp(\pm\frac{h}{2d}H_i),~i=1, \ldots, l$.

The defining relations for the algebra $U_q^{s}({\frak g})$ are
\begin{equation}\label{sqgr1}
\begin{array}{l}
t_it_j=t_jt_i,~~t_it_i^{-1}=t_i^{-1}t_i=1,~~ t_ie_jt_i^{-1}=q^{\frac{a_{ij}}{2d}}e_j, ~~t_if_jt_i^{-1}=q^{-\frac{a_{ij}}{2d}}f_j,\\
\\
e_i f_j -q^{ c_{ij}} f_j e_i = \delta _{i,j}{K_i -K_i^{-1} \over q_i -q_i^{-1}} , c_{ij}=\left( {1+s \over 1-s }P_{{\h'}^*}\alpha_i , \alpha_j \right),\\
 \\
K_i=t_i^{2d d_i}, \\
 \\
\sum_{r=0}^{1-a_{ij}}(-1)^r q^{r c_{ij}}
\left[ \begin{array}{c} 1-a_{ij} \\ r \end{array} \right]_{q_i}
(e_i )^{1-a_{ij}-r}e_j (e_i)^r =0 ,~ i \neq j , \\
 \\
\sum_{r=0}^{1-a_{ij}}(-1)^r q^{r c_{ij}}
\left[ \begin{array}{c} 1-a_{ij} \\ r \end{array} \right]_{q_i}
(f_i )^{1-a_{ij}-r}f_j (f_i)^r =0 ,~ i \neq j .
\end{array}
\end{equation}
Note that by the choice of $d$ we have $q^{c_{ij}}\in \mathbb{C}[q^{\frac{1}{2d}},q^{-\frac{1}{2d}}]$.

The second form of $U_h^{s}({\frak g})$ is a subalgebra $U_\mathcal{A}^{s}(\g)$ in $U_q^s(\g)$ over the ring $\mathcal{A}=\mathbb{C}[q^{\frac{1}{2d}},q^{-\frac{1}{2d}},\frac{1}{[2]_{q_i}},\ldots ,\frac{1}{[r]_{q_i}}]$, where $i=1, \ldots ,l$, $r$ is the maximal number $k_i$ that appears in the right--hand sides of formulas (\ref{qcom}) for various $\alpha$ and $\beta$. $U_\mathcal{A}^{s}(\g)$ is the subalgebra in $U_q^s(\g)$ generated over $\mathcal{A}$ by the elements $t_i^{\pm 1},~{K_i -K_i^{-1} \over q_i -q_i^{-1}},~e_i,~f_i,~i=1,\ldots ,l$.

The most important for us is the specialization $U_\varepsilon^s(\g)$ of $U_\mathcal{A}^{s}(\g)$, $U_\varepsilon^s(\g)=U_\mathcal{A}^{s}(\g)/(q^{\frac{1}{2d}}-\varepsilon^{\frac{1}{2d}})U_\mathcal{A}^{s}(\g)$, $\varepsilon\in \mathbb{C}^*$, $[r]_{\varepsilon_i}!\neq 0$, $i=1, \ldots ,l$. Note that $[r]_{1}!\neq 0$, and hence one can define the specialization $U_1^s(\g)$.

$U_q^s(\g)$, $U_\mathcal{A}^{s}(\g)$ and $U_\varepsilon^s(\g)$ are Hopf algebras with the comultiplication induced from $U_h^s(\g)$. If in addition $\varepsilon^{2d_i}\neq 1$, $i=1,\ldots, l$, then $U_\varepsilon^s(\g)$ is generated over $\mathbb{C}$ by $t_i^{\pm 1},~e_i,~f_i,~i=1,\ldots ,l$ subject to relations (\ref{sqgr1}) where $q=\varepsilon$.
The algebra $U_\mathcal{A}^{s}(\g)$ has a similar description.
The elements $t_i$ are central in the algebra $U_1^s(\g)$, and the quotient of $U_1^s(\g)$ by the two--sided ideal generated by $t_i-1$ is isomorphic to $U(\g)$. Note that none of the subalgebras of $U_h^s(\g)$ introduced above is quasitriangular.

As usual, one can define highest weight, Verma and finite--dimensional modules for all forms and specializations of the quantum group $U_h^{s}({\frak g})$ introduced above (see \cite{S10}, Section 7).

For the solution $n_{ij}=\frac{1}{2d_j}c_{ij}$ to equations (\ref{eqpi}) the root vectors $e_{\beta},f_\beta$ belong to all the above introduced subalgebras of $U_h({\frak g})$, and one can define analogues of root vectors for them in a similar way. From now on we shall assume that the solution to equations (\ref{eqpi}) is fixed as above, $n_{ij}=\frac{1}{2d_j}c_{ij}$.

Denote by $U_q^s({\frak n}_+),U_q^s({\frak n}_-)$ and $U_q^s({\frak h})$ the subalgebras of $U_q^s({\frak g})$
generated by the
$e_i$, $f_i$ and by the $t_i$, respectively. Then the elements $e^{\bf r}=e_{\beta_1}^{r_1}\ldots e_{\beta_D}^{r_D},~~f^{\bf t}=f_{\beta_D}^{t_D}\ldots f_{\beta_1}^{t_1}$
and $t^{\bf s}=t_1^{s_1}\ldots t_l^{s_l}$, for ${\bf r}=(r_1,\ldots r_D),~{\bf t}=(t_1,\ldots t_D)\in {\Bbb N}^D$,
${\bf s}=(s_1,\ldots s_l)\in {\Bbb Z}^l$, form bases of $U_q^s({\frak n}_+),U_q^s({\frak n}_-)$ and $U_q^s({\frak h})$,
respectively, and the products $e^{\bf r}t^{\bf s}f^{\bf t}$ form a basis of
$U_q^s({\frak g})$ (see \cite{S10}, Section 7).

We shall also use quantum analogues of Borel subalgebras $U_q^s({\frak b}_+),U_q^s({\frak b}_-)$, $U_q^s({\frak b}_\pm)$ is the subalgebra in $U_q^s({\frak g})$ generated by $U_q^s({\frak n}_\pm)$ and by $U_q^s({\frak h})$, $U_q^s({\frak b}_\pm)=U_q^s({\frak n}_+)U_q^s({\frak h})$.
By specializing the above constructed basis for $q=\varepsilon$ we obtain a similar basis and similar subalgebras for $U_\varepsilon^s(\g)$.

Let $U_\mathcal{A}^{s}({\frak n}_+),U_\mathcal{A}^{s}({\frak n}_-)$ be the subalgebras of $U_\mathcal{A}^{s}({\frak g})$
generated by the
$e_i$ and by the $f_i$, $i=1,\ldots,l$, respectively.
Using the root vectors $e_{\beta}$ and $f_{\beta}$ we can construct a basis of
$U_\mathcal{A}^{s}({\frak g})$.
Namely, the elements $e^{\bf r}$, $f^{\bf t}$ for ${\bf r},~{\bf t}\in {\Bbb N}^N$
form bases of $U_\mathcal{A}^{s}({\frak n}_+),U_\mathcal{A}^{s}({\frak n}_-)$,
respectively.

The elements
$$
\left[ \begin{array}{l}
K_i;c \\
r
\end{array} \right]_{q_i}=\prod_{s=1}^r \frac{K_i q_i^{c+1-s}-K_i^{-1}q_i^{s-1-c}}{q_i^s-q_i^{-s}}~,~i=1,\ldots,l,~c\in \mathbb{Z},~r\in \mathbb{N}
$$
belong to $U_\mathcal{A}^{s}(\g)$. Denote by $U_\mathcal{A}^{s}(\h)$ the subalgebra of $U_\mathcal{A}^{s}(\g)$ generated by those elements and by $t_i^{\pm 1}$, $i=1,\ldots,l$.
Then multiplication defines an isomorphism of $\mathcal{A}$ modules:
$$
U_\mathcal{A}^{s}({\frak n}_-)\otimes U_\mathcal{A}^{s}({\frak h}) \otimes U_\mathcal{A}^{s}({\frak n}_+)\rightarrow U_\mathcal{A}^{s}({\frak g}).
$$

We shall also use the subalgebras $U_\mathcal{A}^s({\frak b}_\pm)\subset U_\mathcal{A}^{s}(\g)$ generated by $U_\mathcal{A}^s({\frak n}_\pm)$ and by $U_\mathcal{A}^s({\frak h})$.
A basis for $U_\mathcal{A}^{s}(\h)$ is a little bit more difficult to describe. We do not need its explicit description (see \cite{ChP}, Proposition 9.3.3 for details).

Finally we discuss an obvious analogue of the subalgebra $U_h^{s}({\frak m}_+)\subset U_h^{s}({\frak g})$ for $U_\mathcal{A}^{s}(\g)$.

Let $U_\mathcal{A}^{s}({\frak m}_+)\subset U_\mathcal{A}^{s}({\frak g})$ be the subalgebra generated by elements $e_\beta\in U_\mathcal{A}^{s}({\n_+})$, $\beta \in \Delta_{\m_+}$, where $\Delta_{\m_+}\subset \Delta$ is the ordered segment $\Delta_{\m_+}$.

By the results of \cite{S10}, Section 7 the elements
$e^{\bf r}=e_{\beta_1}^{r_1}\ldots e_{\beta_D}^{r_D}$, $r_i\in \mathbb{N}$, $i=1,\ldots, D$, and $r_i$ can be strictly positive only if $\beta_i\in \Delta_{\m_+}$, form a
basis of $U_\mathcal{A}^{s}({\frak m}_+)$.
Obviously $U_\mathcal{A}^{s}({\frak m}_+)/(q^{\frac{1}{2d}}-1)U_\mathcal{A}^{s}({\frak m}_+)\simeq U({\m_+})$, where ${\m_+}$ is the Lie subalgebra of $\g$ generated by the root vectors $X_\alpha$, $\alpha\in \Delta_{\m_+}$.
By specializing $q$ to a particular value $q=\varepsilon$ one can obtain a subalgebra $U_\varepsilon^{s}({\frak m}_+)\subset U_\varepsilon^{s}(\g)$ with similar properties.

The algebras $U_q^s(\g)$,$U_{\mathcal{A}}^s(\g)$ and $U_\varepsilon^s(\g)$ can be equipped with remarkable filtrations such that the associated graded algebras are almost commutative (see \cite{DK}). For ${\bf r},~{\bf t}\in {\Bbb N}^D$ define the height of the element
 $u_{{\bf r},{\bf t},t}=e^{\bf r}tf^{\bf t}$, $t\in U_q^s({\frak h})$ as follows ${\rm ht}~(u_{{\bf r},{\bf t},t})=\sum_{i=1}^D(t_i+r_i){\rm ht}~\beta_i\in \mathbb{N}$, where ${\rm ht}~\beta_i$ is the height of the root $\beta_i$. Introduce also the degree of $u_{{\bf r},{\bf t},t}$ by
 $$
 d(u_{{\bf r},{\bf t},t})=(r_1,\ldots,r_D,t_D,\ldots,t_1,{\rm ht}~(u_{{\bf r},{\bf t},t}))\in \mathbb{N}^{2D+1}.
 $$
 Equip $\mathbb{N}^{2D+1}$ with the total lexicographic order and denote by $(U_q^s(\g))_k$ the span of elements $u_{{\bf r},{\bf t},t}$ with $d(u_{{\bf r},{\bf t},t})\leq k$ in $U_q^s(\g)$. Then Proposition 1.7 in \cite{DK} implies that $(U_q^s(\g))_k$ is a filtration of $U_q^s(\g)$ such that the associated graded algebra is the associative algebra over $\mathbb{C}(q)$ with generators $e_\alpha,~f_\alpha$, $\alpha\in \Delta_+$, $t_i^{\pm 1}$, $i=1,\ldots l$ subject to the relations
 \begin{equation}\label{scomm}
 \begin{array}{l}
 t_it_j=t_jt_i,~~t_it_i^{-1}=t_i^{-1}t_i=1,~~ t_ie_\alpha t_i^{-1}=q^{\frac{H_i(\alpha)}{2d}}e_\alpha, ~~t_if_\alpha t_i^{-1}=q^{-\frac{H_i(\alpha)}{2d}}f_j,\\
\\
e_\alpha f_\beta =q^{({1+s \over 1-s}P_{{\h'}^*}\alpha,\beta)} f_\beta e_\alpha,\\
 \\
 e_{\alpha}e_{\beta} = q^{(\alpha,\beta)+({1+s \over 1-s}P_{{\h'}^*}\alpha,\beta)}e_{\beta}e_{\alpha},~ \alpha<\beta, \\
 \\
 f_{\alpha}f_{\beta} = q^{(\alpha,\beta)+({1+s \over 1-s}P_{{\h'}^*}\alpha,\beta)}f_{\beta}f_{\alpha},~ \alpha<\beta.
 \end{array}
 \end{equation}
 Such algebras are called semi--commutative. A similar result holds for the algebras $U_\varepsilon^s(\g)$ and $U_\mathcal{A}^{s}(\g)$.


\section{Quantized algebras of regular functions on Poisson--Lie groups and q-W algebras}\label{qplproups}

\setcounter{equation}{0}
\setcounter{theorem}{0}

First we recall some notions concerned with Poisson--Lie groups (see \cite{ChP}, \cite{Dm},
\cite{fact}, \cite{dual}). These facts will be used for the study of q-W algebras.

Let $G$ be a finite--dimensional Lie group equipped with a Poisson bracket,
$\frak g$ its Lie algebra. $G$ is called
a Poisson--Lie group if the multiplication $G\times G \rightarrow G$ is a
Poisson map.
A Poisson bracket satisfying this axiom is degenerate and, in particular, is
identically zero
at the unit element of the group. Linearizing this bracket at the unit element
defines the
structure of a Lie algebra in the space $T^*_eG\simeq {\frak g}^*$.
The pair (${\frak g},{\frak g}^{*})$ is called the tangent bialgebra of $G$.

Lie brackets in $\frak{g}$ and $\frak{g}^{*}$ satisfy the following
compatibility condition:

{\em Let }$\delta: {\frak g}\rightarrow {\frak g}\wedge {\frak g}$ {\em be
the dual  of the commutator map } $[,]_{*}: {\frak g}^{*}\wedge
{\frak g}^{*}\rightarrow {\frak g}^{*}$. {\em Then } $\delta$ {\em is a
1-cocycle on} $  {\frak g}$ {\em (with respect to the adjoint action
of } $\frak g$ {\em on} ${\frak g}\wedge{\frak g}$).

Let $c_{ij}^{k}, f^{ab}_{c}$ be the structure constants of
${\frak g}, {\frak g}^{*}$ with respect to the dual bases $\{e_{i}\},
\{e^{i}\}$ in ${\frak g},{\frak g}^{*}$. The compatibility condition
means that

$$
c_{ab}^{s} f^{ik}_{s} ~-~ c_{as}^{i} f^{sk}_{b} ~+~ c_{as}^{k}
f^{si}_{b} ~-~ c_{bs}^{k} f^{si}_{a} ~+~ c_{bs}^{i} f^{sk}_{a} ~~=
~~0.
$$
This condition is symmetric with respect to exchange of $c$ and
$f$. Thus if $({\frak g},{\frak g}^{*})$ is a Lie bialgebra, then
$({\frak g}^{*}, {\frak g})$ is also a Lie bialgebra.

The following proposition shows that the category of finite--dimensional Lie
bialgebras is isomorphic to
the category of finite--dimensional connected simply connected Poisson--Lie
groups.
\begin{proposition}{\bf (\cite{ChP}, Theorem 1.3.2)}
If $G$ is a connected simply connected finite--dimensional Lie group, every
bialgebra structure on $\frak g$
is the tangent bialgebra of a unique Poisson structure on $G$ which makes $G$
into a Poisson--Lie group.
\end{proposition}

Let $G$ be a finite--dimensional Poisson--Lie group, $({\frak g},{\frak g}^{*})$
the tangent bialgebra of $G$.
The connected simply connected finite--dimensional
Poisson--Lie group corresponding to the Lie bialgebra $({\frak g}^{*}, {\frak
g})$ is called the dual
Poisson--Lie group and denoted by $G^*$.

$({\frak g},{\frak g}^{*})$ is called a factorizable Lie
bialgebra if the following conditions are satisfied (see \cite{Dm}, \cite{fact}):
\begin{enumerate}
\item
${\frak g}${\em \ is equipped with a non--degenerate invariant
scalar product} $\left( \cdot ,\cdot \right)$.

We shall always identify ${\frak g}^{*}$ and ${\frak g}$ by means of this
scalar product.

\item  {\em The dual Lie bracket on }${\frak g}^{*}\simeq {\frak g}${\em \
is given by}
\begin{equation}
\left[ X,Y\right] _{*}=\frac 12\left( \left[ rX,Y\right] +\left[ X,rY\right]
\right) ,X,Y\in {\frak g},  \label{rbr}
\end{equation}
{\em where }$r\in {\rm End}\ {\frak g}${\em \ is a skew symmetric linear
operator
(classical r-matrix).}

\item  $r${\em \ satisfies} {\em the} {\em modified classical Yang-Baxter
identity:}
\begin{equation}
\left[ rX,rY\right] -r\left( \left[ rX,Y\right] +\left[ X,rY\right] \right)
=-\left[ X,Y\right] ,\;X,Y\in {\frak g}{\bf .}  \label{cybe}
\end{equation}
\end{enumerate}

Define operators $r_\pm \in {\rm End}\ {\frak g}$ by
\[
r_{\pm }=\frac 12\left( r\pm id\right) .
\]
We shall need some properties of the operators $r_{\pm }$.
Denote by ${\frak b}_\pm$ and ${\frak n}_\mp$ the image and the kernel of the
operator
$r_\pm $:
\begin{equation}\label{bnpm}
{\frak b}_\pm = Im~r_\pm,~~{\frak n}_\mp = Ker~r_\pm.
\end{equation}

The classical Yang--Baxter equation implies that $r_{\pm }$ , regarded as a
mapping from
${\frak g}^{*}$ into ${\frak g}$, is a Lie algebra homomorphism.
Moreover, $r_{+}^{*}=-r_{-},$\ and $r_{+}-r_{-}=id.$

Put ${\frak {d}}={\frak g\oplus {\g}}$ (direct sum of two
copies). The mapping
\begin{eqnarray}\label{imbd}
{\frak {g}}^{*}\rightarrow {\frak {d}}~~~:X\mapsto (X_{+},~X_{-}),~~~X_{\pm
}~=~r_{\pm }X
\end{eqnarray}
is a Lie algebra embedding. Thus we may identify ${\frak g^{*}}$ with a Lie
subalgebra in ${\frak {d}}$.

Naturally, embedding (\ref{imbd}) extends to a homomorphism
$$
G^*\rightarrow G\times G,~~L\mapsto (L_+,L_-).
$$
We shall identify $G^*$ with the corresponding subgroup in $G\times G$.

There exists a unique right local Poisson group action
$$
G^*\times G\rightarrow G^*,~~((L_+,L_-),g)\mapsto g\circ (L_+,L_-),
$$
such that if $q:G^* \rightarrow G$ is the map defined by
$$
q(L_+,L_-)=L_-L_+^{-1}
$$
then
$$
q(g\circ (L_+,L_-))=g^{-1}L_-L_+^{-1}g.
$$
This action is called the dressing action of $G$ on $G^*$.

Let $\frak g$ be a finite--dimensional complex simple Lie algebra, $\h\subset \g$ its Cartan subalgebra. Let $s\in W$ be an element of the Weyl group $W$ of the pair $(\g,\h)$ and $\Delta_+$ the system of positive roots associated to $s$.
Observe that
cocycle (\ref{cocycles}) equips
$\frak g$ with
the structure of a factorizable Lie bialgebra, where the scalar product is given by the Killing form.
Using the identification
${\rm End}~{\frak g}\cong {\frak g}\otimes {\frak g}$ the corresponding
r--matrix may be represented as
$$
r^{s}=P_+-P_-+{1+s \over 1-s}P_{{\h'}},
$$
where $P_+,P_-$ and $P_{{\h'}}$ are the projection operators onto ${\frak n}_+,{\frak
n}_-$ and ${\frak h}'$ in
the direct sum
$$
{\frak g}={\frak n}_+ +{\frak h}'+{\h'}^\perp + {\frak n}_-,
$$
where ${\h'}^\perp$ is the orthogonal complement to $\h'$ in $\h$ with respect to the Killing form.

Let $G$ be the connected simply connected simple Poisson--Lie group with the
tangent Lie bialgebra $({\frak g},{\frak g}^*)$,
$G^*$ the dual group.

Observe that $G$ is an algebraic group (see \S 104,
Theorem 12 in \cite{Z}).

Note also that
$$
r^{s}_+=P_+ + {1 \over 1-s}P_{{\h'}}+\frac{1}{2}P_{{\h'}^\perp},~~r^{s}_-=-P_- + {s \over 1-s}P_{{\h'}}-\frac{1}{2}P_{{\h'}^\perp},
$$
and hence the subspaces ${\frak b}_\pm$ and ${\frak n}_\pm$ defined by
(\ref{bnpm}) coincide with
the Borel subalgebras in $\frak g$ and their nilradicals, respectively.
Therefore every element $(L_+,L_-)\in G^*$ may be uniquely written as
\begin{equation}\label{fact}
(L_+,L_-)=(h_+,h_-)(n_+,n_-),
\end{equation}
where $n_\pm \in N_\pm$, $h_+=exp(({1 \over 1-s}P_{{\h'}}+\frac{1}{2}P_{{\h'}^\perp})x),~h_-=exp(({s \over 1-s}P_{{\h'}}-\frac{1}{2}P_{{\h'}^\perp})x),~x\in
{\frak h}$.
In particular, $G^*$ is a solvable algebraic subgroup in $G\times G$.

For every algebraic variety $V$ we denote by ${\mathbb{C}}[V]$ the algebra of
regular functions on $V$.
Our main object will be the algebra of regular functions on $G^*$, ${\mathbb{C}}[G^*]$.
This algebra may be explicitly described as follows.
Let $\pi_V$ be a finite--dimensional representation of $G$. Then matrix
elements of $\pi_V(L_\pm)$ are well--defined functions on $G^*$, and ${\mathbb{C}}[G^*]$ is the subspace
in $C^\infty(G^*)$ generated by matrix elements of $\pi_V(L_\pm)$, where $V$
runs through all finite--dimensional
representations of $G$.
The elements $L^{\pm,V}=\pi_V(L_\pm)$ may be viewed as elements of the space
${\mathbb{C}}[G^*]\otimes {\rm End}V$.
\begin{proposition}{\bf (\cite{dual}, Section 2)}\label{pbff}
${\mathbb{C}}[G^*]$ is a Poisson--Hopf subalgebra in the Poisson algebra $C^\infty(G^*)$, the comultiplication and the antipode being induced by the multiplication and by taking inverse in $G^*$, respectively.
\end{proposition}

Now we construct a quantization of the Poisson--Hopf algebra ${\mathbb{C}}[G^*]$.
For technical reasons we shall need an extension of the algebra $U_\mathcal{A}^{s}({\frak g})$ to an algebra $U_{\mathcal{A}'}^{s}({\frak g})=U_\mathcal{A}^{s}({\frak g})\otimes_\mathcal{A}\mathcal{A}'$, where
$\mathcal{A}'=\mathbb{C}[q^{\frac{1}{2d}}, q^{-\frac{1}{2d}}, \frac{1}{[2]_{q_i}},\ldots ,\frac{1}{[r]_{q_i}}, \frac{1-q^{\frac{1}{2d}}}{1-q_i^{-2}}]_{i=1,\ldots, l}$. Note that the ratios $\frac{1-q^{\frac{1}{2d}}}{1-q_i^{-2}}$ have no singularities when $q=1$, and we can define a localization,
$\mathcal{A}'/(1-q^{\frac{1}{2d}})\mathcal{A}'=\mathbb{C}$ as well as similar localizations for other generic values of $\varepsilon$, $\mathcal{A}'/(\varepsilon^{\frac{1}{2d}}-q^{\frac{1}{2d}})\mathcal{A}'=\mathbb{C}$ and the corresponding localizations of algebras over $\mathcal{A}'$. $U_{\mathcal{A}'}^{s}({\frak g})$ is naturally a Hopf algebra with the comultiplication and the antipode induced from $U_{\mathcal{A}}^{s}({\frak g})$.

First, by the results of Section 10 in \cite{S10} for any finite--dimensional $U_\mathcal{A}^{s}({\frak g})$--module $V$ the invertible elements
${^q{L^{\pm,V}}}$ given by
$$
{^q{L^{+,V}}}=(id\otimes \pi_V){{\mathcal R}_{21}^{s}}^{-1}=(id\otimes
\pi_VS^{s}){\mathcal R}_{21}^{s}
,~~ {^q{L^{-,V}}}=(id\otimes \pi_V){\mathcal R}^{s}.
$$
are well--defined elements of $U_\mathcal{A}^{s}({\frak g})\otimes {\rm End}V$.
If we fix a basis in $V$, ${^q{L^{\pm,V}}}$ may be regarded as matrices
with matrix
elements $({^q{L^{\pm,V}}})_{ij}$ being elements of $U_\mathcal{A}^{s}({\frak g})$.

We denote by ${\mathbb{C}}_{\mathcal{A}'}[G^*]$ the Hopf subalgebra in $U_{\mathcal{A}'}^{s}({\frak
g})$ generated by matrix elements
of $({^q{L^{\pm,V}}})^{\pm 1}$, where $V$ runs through all finite--dimensional
representations of $U_\mathcal{A}^{s}({\frak
g})$.

The quotient algebra ${\mathbb{C}}_{\mathcal{A}'}[G^*]/(q^{\frac{1}{2d}}-1){\mathbb{C}}_{\mathcal{A}'}[G^*]$ is commutative (see e.g. \cite{S10}, Section 10), and one can equip it with a Poisson structure given by
\begin{equation}\label{quasipb}
\{x_1,x_2\}=\frac{1}{2d}{[a_1,a_2] \over q^{\frac{1}{2d}}-1}~(\mbox{mod }(q^{\frac{1}{2d}}-1)),
\end{equation}
where $a_1,a_2\in {\mathbb{C}}_{\mathcal{A}'}[G^*]$ reduce to
$x_1,x_2\in {\mathbb{C}}_{\mathcal{A}'}[G^*]/(q^{\frac{1}{2d}}-1){\mathbb{C}}_{\mathcal{A}'}[G^*]~(\mbox{mod }(q^{\frac{1}{2d}}-1))$.
Obviously, the the comultiplication and the antipode on ${\mathbb{C}}_{\mathcal{A}'}[G^*]$ induce a comultiplication and an antipode on ${\mathbb{C}}_{\mathcal{A}'}[G^*]/(q^{\frac{1}{2d}}-1){\mathbb{C}}_{\mathcal{A}'}[G^*]$ compatible with the introduced Poisson structure, and the quotient ${\mathbb{C}}_{\mathcal{A}'}[G^*]/(q^{\frac{1}{2d}}-1){\mathbb{C}}_{\mathcal{A}'}[G^*]$ becomes a Poisson--Hopf algebra.

\begin{proposition}{\bf (\cite{S10}, Proposition 10.2)}\label{quantreg}
The Poisson--Hopf algebra ${\mathbb{C}}_{\mathcal{A}'}[G^*]/(q^{\frac{1}{2d}}-1){\mathbb{C}}_{\mathcal{A}'}[G^*]$ is isomorphic to ${\mathbb{C}}[G^*]$ as a Poisson--Hopf
algebra.
\end{proposition}

We shall call the map $p:{\mathbb{C}}_{\mathcal{A}'}[G^*] \rightarrow {\mathbb{C}}_{\mathcal{A}'}[G^*]/(q^{\frac{1}{2d}}-1){\mathbb{C}}_{\mathcal{A}'}[G^*]={\mathbb{C}}[G^*]$ the
quasiclassical limit.

From the definition of the elements ${^q{L^{\pm,V}}}$ it follows that ${\mathbb{C}}_{\mathcal{A}'}[G^*]$ is the subalgebra in $U_{\mathcal{A}'}^{s}({\frak g})$ generated by the elements $\prod_{j=1}^lt_j^{\pm 2dp_{ij}},~\prod_{j=1}^lt_j^{\pm 2dp_{ji}},~i=1,\ldots l,~\tilde e_{\beta}=(1-q_\beta^{-2})e_{\beta},~\tilde f_{\beta}=(1-q_\beta^{-2})e^{h\beta^\vee}f_{\beta},~\beta\in \Delta_+$.

Now using the Hopf algebra ${\mathbb{C}}_{\mathcal{A}'}[G^*]$ we shall define quantum
versions of W--algebras.
From the definition of the elements ${^q{L^{\pm,V}}}$ it follows that the matrix elements of ${^q{L^{\pm,V}}}^{\pm 1}$ form Hopf subalgebras ${\mathbb{C}}_{\mathcal{A}'}[B_\pm]\subset {\mathbb{C}}_{\mathcal{A}'}[G^*]$, and that
${\mathbb{C}}_{\mathcal{A}'}[G^*]$ contains the subalgebra ${\mathbb{C}}_{\mathcal{A}'}[N_-]$
generated by elements
$\tilde e_{\beta}=(1-q_\beta^{-2})e_{\beta}$, $\beta\in \Delta_+$.

Suppose that the positive root system $\Delta_+$ and its ordering are associated to $s$.
Denote by ${\mathbb{C}}_{\mathcal{A}'}[M_-]$ the subalgebra in ${\mathbb{C}}_{\mathcal{A}'}[N_-]$
generated by elements
$\tilde e_{\beta}$, $\beta\in {\Delta_{\m_+}}$.

Note that one van consider $\n_-$ and $\m_\pm$ as Lie subalgebras in $\g^*$ via imbeddings
$$
\n_-\rightarrow \g^*\subset \g\oplus \g,~x\mapsto (0,x),
$$
$$
\m_+\rightarrow \g^*\subset \g\oplus \g,~x\mapsto (x,0),
$$
$$
\m_-\rightarrow \g^*\subset \g\oplus \g,~x\mapsto (0,x),
$$
By construction ${\mathbb{C}}_{\mathcal{A}'}[N_-]$ is a quantization of the algebra of regular functions on the algebraic subgroup $N_-\subset G^*$ corresponding to the Lie subalgebra $\n_- \subset \g^*$, and ${\mathbb{C}}_{\mathcal{A}'}[M_-]$ is a quantization of the algebra of regular functions on the algebraic subgroup $M_-\subset G^*$ corresponding to the Lie subalgebra $\m_- \subset \g^*$ in the sense that $p({\mathbb{C}}_{\mathcal{A}'}[N_-])={\mathbb{C}}[N_-]$ and $p({\mathbb{C}}_{\mathcal{A}'}[M_-])={\mathbb{C}}[M_-]$. We also denote by $M_+$ the algebraic subgroup $M_+\subset G^*$ corresponding to the Lie subalgebra $\m_+ \subset \g^*$.

The following proposition gives the most important property of the subalgebra ${\mathbb{C}}_{\mathcal{A}'}[M_-]$ which plays the key role in the definition of q-W--algebras.
\begin{proposition}{\bf (\cite{S10}, Section 10)}\label{Qdefr}
The defining relations in the subalgebra ${\mathbb{C}}_{\mathcal{A}'}[M_-]$ for the generators $\tilde{e}_\beta=(1-q_\beta^{-2})e_\beta$, $\beta\in \Delta_{\m_+}$ are of the form
\begin{equation}\label{cmrel}
\tilde{e}_{\alpha}\tilde{e}_{\beta} - q^{(\alpha,\beta)+({1+s \over 1-s}P_{{\h'}^*}\alpha,\beta)}\tilde{e}_{\beta}\tilde{e}_{\alpha}= \sum_{\alpha<\delta_1<\ldots<\delta_n<\beta}C''(k_1,\ldots,k_n)
{\tilde{e}_{\delta_1}}^{k_1}{\tilde{e}_{\delta_2}}^{k_2}\ldots {\tilde{e}_{\delta_n}}^{k_n},~\alpha<\beta,
\end{equation}
where $C''(k_1,\ldots,k_n)\in \mathcal{A}'$ has a zero of order 1, as a function of $q$, at point $q=1$, and for any $k_i\in \mathcal{A}'$, $i=1,\ldots ,l'$ the map $\chi_q^{s}:{\mathbb{C}}_{\mathcal{A}'}[M_-]\rightarrow \mathcal{A}'$,
\begin{equation}\label{charq}
\chi_q^s(\tilde e_\beta)=\left\{ \begin{array}{ll} 0 & \beta \not \in \{\gamma_1, \ldots, \gamma_{l'}\} \\ k_i & \beta=\gamma_i
\end{array}
\right  .,
\end{equation}
is a character of ${\mathbb{C}}_{\mathcal{A}'}[M_-]$ vanishing on the r.h.s. and on the l.h.s. of relations (\ref{cmrel}).
\end{proposition}

Denote by $\mathbb{C}_{\chi_q^{s}}$ the rank one representation of the algebra ${\mathbb{C}}_{\mathcal{A}'}[M_-]$ defined by the character $\chi_q^{s}$.

For any finite--dimensional $U_\mathcal{A}^{s}({\frak g})$--module $V$ let ${^q{L^{V}}}={^q{L^{-,V}}}{^q{L^{+,V}}}^{-1}=(id\otimes \pi_V){\mathcal R}^{s}{{\mathcal R}_{21}^{s}}$.
Let ${\mathbb{C}}_{\mathcal{A}'}[{G}_*]$ be the $\mathcal{A}'$--subalgebra in ${\mathbb{C}}_{\mathcal{A}'}[G^*]$ generated by the matrix entries of ${^q{{L}^{V}}}$, where $V$ runs over all finite--dimensional representations of $U_\mathcal{A}^{s}({\frak g})$.

Define the right adjoint action of $U_{\mathcal{A}'}^s(\g)$ on $U_{\mathcal{A}'}^s(\g)$ by the formula
\begin{equation}\label{ad}
{\rm Ad}x(w)=S_s^{-1}(x_2)wx_1,
\end{equation}
where we use the abbreviated notation for the coproduct $\Delta_s(x)=x_1\otimes x_2$, $x\in U_{\mathcal{A}'}^s(\g)$, $w\in U_{\mathcal{A}'}^s(\g)$.

Note that by Lemma 2.2 in \cite{JL}
\begin{equation}\label{adm}
{\rm Ad}x(wz)={\rm Ad}x_2(w){\rm Ad}x_1(z).
\end{equation}

Observe also that by definition the adjoint action introduced above is dual to a restriction of the dressing coaction of the quantization of the algebra of regular functions on the Poisson-Lie group $G$ on the space ${\mathbb{C}}_{\mathcal{A}'}[G^*]$. Therefore the subspace ${\mathbb{C}}_{\mathcal{A}'}[G^*]\subset U_{\mathcal{A}'}^s(\g)$ is stable under the adjoint action. The subalgebra ${\mathbb{C}}_{\mathcal{A}'}[{G}_*]\subset {\mathbb{C}}_{\mathcal{A}'}[G^*]$ is also stable under the dressing coaction (see \cite{dual}, Section 3), and hence ${\mathbb{C}}_{\mathcal{A}'}[{G}_*]$ is stable under the adjoint action.

\begin{proposition}\label{locfin}
Let $\varepsilon\in \mathbb{C}$ be generic such that $[r]_{\varepsilon_i}!\neq 0$, $\varepsilon\neq 0$, $i=1,\ldots ,l$. Define the complex associative algebra
${\mathbb{C}}_\varepsilon[G_*]={\mathbb{C}}_{\mathcal{A}'}[G_*]/(q^{\frac{1}{2d}}-\varepsilon^{\frac{1}{2d}})
{\mathbb{C}}_{\mathcal{A}'}[G_*]$. Then
the algebra ${\mathbb{C}}_\varepsilon[{G}_*]$ can be identified with  the ${\rm Ad}$ locally finite part $U_{\varepsilon}^s(\g)^{fin}$ of $U_{\varepsilon}^s(\g)$,
$$
U_{\varepsilon}^s(\g)^{fin}=\{x\in U_{\varepsilon}^s(\g):{\rm dim}({\rm Ad}U_{\varepsilon}^s(\g)(x))<+\infty \},
$$
where the adjoint action of the algebra $U_{\varepsilon}^s(\g)$ on itself is defined by formula (\ref{ad}).
\end{proposition}

\begin{proof}
Indeed, let $V_i$, $i=1,\ldots, l$ be the fundamental representations of $U_{\mathcal{A}}^s(\g)$ with highest weights $Y_i$, $i=1,\ldots, l$. From formula (\ref{rmatrspi}) and from the definition of ${^q{L^{V}}}=(id\otimes \pi_V){\mathcal R}^{s}{{\mathcal R}_{21}^{s}}$ it follows that the matrix element $(id\otimes v_i^*){\mathcal R}^{s}{{\mathcal R}_{21}^{s}}(id\otimes v_i)$ of ${^q{L^{V_i}}}$ corresponding to the highest weight vector $v_i$ of $V_i$ and to the lowest weight vector $v_i^*\in V_i^*$ of the dual representation $V_i^*$, normalized in such a way that $v_i^*(v_i)=1$, coincides with $L_i^2$. This implies that $L_i^2$ are elements of the algebra ${\mathbb{C}}_\varepsilon[G_*]\subset U_{\varepsilon}^s(\g)$ as well. Denote by $\mathfrak{H}\subset {\mathbb{C}}_{\varepsilon}[{G}_*]\subset U_{\varepsilon}^s(\g)$ the subalgebra generated by the elements $L_i^2\in {\mathbb{C}}_{\varepsilon}[{G}_*]$, $i=1,\ldots, l$. By Theorem 7.1.6 and Lemma 7.1.16 in \cite{Jos} $U_{\varepsilon}^s(\g)^{fin}={\rm Ad}U_{\varepsilon}^s(\g)\mathfrak{H}$. Since ${\mathbb{C}}_{\varepsilon}[{G}_*]$ is stable under the adjoint action we have an inclusion, $U_{\varepsilon}^s(\g)^{fin}\subset {\mathbb{C}}_{\varepsilon}[{G}_*]$. On the other hand from formula (3.26) in \cite{dual} it follows that the dressing coaction on ${\mathbb{C}}_{\varepsilon}[{G}_*]$ is locally cofinite, and hence the adjoint action of $U_{\varepsilon}^s(\g)$ on ${\mathbb{C}}_{\varepsilon}[{G}_*]$ is locally finite. Hence ${\mathbb{C}}_{\varepsilon}[{G}_*]\subset U_{\varepsilon}^s(\g)^{fin}$, and ${\mathbb{C}}_{\varepsilon}[{G}_*]=U_{\varepsilon}^s(\g)^{fin}$.
\end{proof}

Denote by $I_q$ the left ideal in ${\mathbb{C}}_{\mathcal{A}'}[G^*]$ generated by the kernel of $\chi_q^s$, and by $\rho_{\chi^{s}_q}$ the canonical projection ${\mathbb{C}}_{\mathcal{A}'}[G^*]\rightarrow {\mathbb{C}}_{\mathcal{A}'}[G^*]/I_q$. Let $Q_{\mathcal{A}'}$ be the image of ${\mathbb{C}}_{\mathcal{A}'}[G_*]$ under the projection $\rho_{\chi^{s}_q}$.

Assume that the roots $\gamma_1, \ldots , \gamma_n$ are simple or that  the set $\gamma_1, \ldots , \gamma_n$ is empty.
Then from formula (\ref{comults}) and from the definition of the normal ordering of $\Delta_+$ associated to $s$ it follows that $\Delta_s(U_{\mathcal{A}'}^s(\m_+))\subset U_{\mathcal{A}'}^s(\m_+)\otimes U_{\mathcal{A}'}^s(\b_+)$, where $U_{\mathcal{A}'}^s(\m_+)=U_{\mathcal{A}}^s(\m_+)\otimes_{\mathcal{A}} \mathcal{A'}$, $U_{\mathcal{A}'}^s(\b_+)=U_{\mathcal{A}}^s(\b_+)\otimes_{\mathcal{A}} \mathcal{A'}$.

Now observe that from Proposition \ref{Qdefr} it follows that the r.h.s. in formula (\ref{cmrel}) belongs to the subspace $(1-q_{\alpha}^{-2}){\rm Ker }\chi_q^{s}$ and hence dividing (\ref{cmrel}) by $(1-q_{\alpha}^{-2})$ we obtain
\begin{equation}\label{epr4}
{e}_{\alpha}\tilde{e}_{\beta} - q^{(\alpha,\beta)+({1+s \over 1-s}P_{{\h'}^*}\alpha,\beta)}\tilde{e}_{\beta}{e}_{\alpha}= \sum_{\alpha<\delta_1<\ldots<\delta_n<\beta}C'''(k_1,\ldots,k_n)
{\tilde{e}_{\delta_1}}^{k_1}{\tilde{e}_{\delta_2}}^{k_2}\ldots {\tilde{e}_{\delta_n}}^{k_n},
\end{equation}
where $C'''(k_1,\ldots,k_n)=C''(k_1,\ldots,k_n)/(1-q_{\alpha}^{-2})\in \mathcal{A}'$.

Therefore we have an inclusion $[U_{\mathcal{A}'}^s(\m_+),{\rm Ker } \chi_q^{s}]\subset {\rm Ker} \chi_q^{s}$. Using this inclusion, formula (\ref{ad}), the fact that $\Delta_s(U_{\mathcal{A}'}^s(\m_+))\subset U_{\mathcal{A}'}^s(\m_+)\otimes U_{\mathcal{A}'}^s(\b_+)$ (see formula (\ref{comults})) we deduce  that the adjoint action of $U_{\mathcal{A}'}^s(\m_+)$ on ${\mathbb{C}}_{\mathcal{A}'}[G_*]$ induces an adjoint action on $Q_{\mathcal{A}'}$ which we also call the adjoint action and denote it by $\rm Ad$.

Let $\mathbb{C}_{\mathcal{A}'}$ be the trivial representation of $U_{\mathcal{A}'}^s(\m_+)$ given by the counit. Consider the space $W_q^s(G)$ of $\rm Ad$--invariants in $Q_{\mathcal{A}'}$,
\begin{equation}\label{QW}
W_q^s(G)={\rm Hom}_{U_{\mathcal{A}'}^s(\m_+)}(\mathbb{C}_{\mathcal{A}'},Q_{\mathcal{A}'}).
\end{equation}

\begin{proposition}\label{9.2}
$W_q^s(G)$ is isomorphic to the subspace of all $v+I_q\in Q_{\mathcal{A}'}$ such that $mv\in I_q$ (or $[m,v]\in I_q$) in ${\mathbb{C}}_{\mathcal{A}'}[G^*]$ for any $m\in I_q$, where $v\in {\mathbb{C}}_{\mathcal{A}'}[G^*]$ is any representative of $v+I_q\in Q_{\mathcal{A}'}$.

Multiplication in ${\mathbb{C}}_{\mathcal{A}'}[G^*]$ induces a multiplication on the space $W_q^s(G)$.
\end{proposition}

\begin{proof}
From formulas (\ref{comults}) and (\ref{epr4}) it follows that for $\beta_k\in \Delta_{\m_+}$
\begin{equation}\label{cm1}
\Delta_s(e_{\beta_k})=e_{\beta_k}\otimes e^{h({2 \over 1-s}P_{{\h'}}+P_{{\h'}^\perp})\beta_k^\vee}+1\otimes e_{\beta_k}+x,~x\in I_{<\beta_k}\otimes U_{\mathcal{A}'}^s(\b_+),
\end{equation}
where $I_{<\beta_k}$ is the intersection of the ideal $I_q$ and of the subalgebra in ${\mathbb{C}}_{\mathcal{A}'}[M_-]$ generated by $\tilde{e}_{\beta_r}$, $0<\beta_r<\beta_k$.

Recall that $S^{-1}_s$ is the antipode for the comultiplication $\Delta_s^{opp}$ and hence from the definition of the antipode, the fact that $S^{-1}_s(U_{\mathcal{A}'}^s(\b_+))\subset U_{\mathcal{A}'}^s(\b_+)$ and inclusion (\ref{cm1}) we must have
$$
S^{-1}_s(e_{\beta_k})+e^{-h({2 \over 1-s}P_{{\h'}}+P_{{\h'}^\perp})\beta_k^\vee}e_{\beta_k}\in U_{\mathcal{A}'}^s(\b_+)I_{<\beta_k}.
$$
Therefore
$$
S^{-1}_s(e_{\beta_k})=-e^{-h({2 \over 1-s}P_{{\h'}}+P_{{\h'}^\perp})\beta_k^\vee}e_{\beta_k}+y, y\in U_{\mathcal{A}'}^s(\b_+)I_{<\beta_k}
$$

Now the last formula, formula (\ref{cm1}) and the definition of the adjoint action of $U_{\mathcal{A}'}^s(\m_+)$ on $Q_\mathcal{A}'$ imply that for any representative $v\in {\mathbb{C}}_{\mathcal{A}'}[G^*]$ of any element $v+I_q\in Q_{\mathcal{A}'}$ we have the following identity in ${\mathbb{C}}_{\mathcal{A}'}[G^*]$
\begin{equation}\label{qqq}
{\rm Ad}e_{\beta_k}v=-\frac{1}{(1-q_{\beta_k}^{-2})}e^{-h({2 \over 1-s}P_{{\h'}}+P_{{\h'}^\perp})\beta_k^\vee}(\tilde{e}_{\beta_k}-\chi_q^s(\tilde{e}_{\beta_k}))v+yv+z,~y\in U_{\mathcal{A}'}^s(\b_+)I_{<\beta_k},~z\in I_q.
\end{equation}

From the last identity it obviously follows that if $mv\in I_q$ in ${\mathbb{C}}_{\mathcal{A}'}[G^*]$ for any $m\in I_q$ then $v+I_q$ is invariant with respect to the adjoint action.

Now let $v+I_q\in Q_{\mathcal{A}'}$, be an element which is invariant with respect to the adjoint action,
$$
{\rm Ad}x(v)=\varepsilon(x)v+z',~x\in U_{\mathcal{A}'}^s(\m_+),~z'\in I_q.
$$

Since $\beta_1\in \Delta_{\m_+}$ is the first positive root in the normal ordering associated to $s$ we have $I_{<\beta_1}=0$, and (\ref{qqq}) implies that
$$
z'={\rm Ad}e_{\beta_1}v=-\frac{1}{(1-q_{\beta_1}^{-2})}e^{-h({2 \over 1-s}P_{{\h'}}+P_{{\h'}^\perp})\beta_1^\vee}(\tilde{e}_{\beta_1}-\chi_q^s(\tilde{e}_{\beta_1}))v+z, ~z\in I_q.
$$
We obtain from the last identity that
$$
-\frac{1}{(1-q_{\beta_1}^{-2})}e^{-h({2 \over 1-s}P_{{\h'}}+P_{{\h'}^\perp})\beta_1^\vee}(\tilde{e}_{\beta_1}-\chi_q^s(\tilde{e}_{\beta_1})){v}\in I_q
$$
which is obviously possible only in case if $(\tilde{e}_{\beta_1}-\chi_q^s(\tilde{e}_{\beta_1})){v}\in I_q$, i.e. when
$$
(\tilde{e}_{\beta_1}-\chi_q^s(\tilde{e}_{\beta_1})){v}\in I_q.
$$

Now we proceed by induction. Assume that
$$
(\tilde{e}_{\beta_r}-\chi_q^s(\tilde{e}_{\beta_r})){v}\in I_q
$$
for $0<\beta_r<\beta_k$. From (\ref{qqq}) we have
$$
{\rm Ad}e_{\beta_k}v=-\frac{1}{(1-q_{\beta_k}^{-2})}e^{-h({2 \over 1-s}P_{{\h'}}+P_{{\h'}^\perp})\beta_k^\vee}(\tilde{e}_{\beta_k}-\chi_q^s(\tilde{e}_{\beta_k}))v+yv,~y\in U_{\mathcal{A}'}^s(\b_+)I_{<\beta_k}\in I_q.
$$
By the induction hypothesis $I_{<\beta_k}v \in I_q$, and hence
$$
{\rm Ad}e_{\beta_k}v=-\frac{1}{(1-q_{\beta_k}^{-2})}e^{-h({2 \over 1-s}P_{{\h'}}+P_{{\h'}^\perp})\beta_k^\vee}(\tilde{e}_{\beta_k}-\chi_q^s(\tilde{e}_{\beta_k}))v\in I_q.
$$
Finally an argument similar to that applied in case $k=1$ shows that $(\tilde{e}_{\beta_k}-\chi_q^s(\tilde{e}_{\beta_k}))v \in I_q$. This establishes the induction step and proves that
\begin{equation}\label{qqqqqq}
(\tilde{e}_{\beta_r}-\chi_q^s(\tilde{e}_{\beta_r})){v}\in I_q
\end{equation}
for any $\beta \in \Delta_{\m_+}$. Since as a left ideal $I_q$ is generated by the elements $\tilde{e}_{\beta_r}-\chi_q^s(\tilde{e}_{\beta_r})$, $\beta \in \Delta_{\m_+}$ (\ref{qqqqqq}) proves that $mv\in I_q$ in ${\mathbb{C}}_{\mathcal{A}'}[G^*]$ for any $m\in I_q$.

Now if $v_1,v_2\in {\mathbb{C}}_{\mathcal{A}'}[G^*]$ are any representatives of elements $v_1+I_q,v_2+I_q\in W_q^s(G)$ the formula
$$
(v_1+I_q)(v_2+I_q)=v_1v_2+I_q
$$
defines a multiplication in $W_q^s(G)$.
\end{proof}

We call the space $W_q^s(G)$ equipped with the multiplication defined in the previous proposition the q-W algebra associated to (the conjugacy class of) the Weyl group element $s\in W$.

Now consider the Lie algebra $\mathfrak{L}_{\mathcal{A}'}$ associated to the associative algebra ${\mathbb{C}}_{\mathcal{A}'}[M_-]$, i.e. $\mathfrak{L}_{\mathcal{A}'}$ is the Lie algebra which is isomorphic to ${\mathbb{C}}_{\mathcal{A}'}[M_-]$ as a linear space, and the Lie bracket in $\mathfrak{L}_{\mathcal{A}'}$ is given by the usual commutator of elements in ${\mathbb{C}}_{\mathcal{A}'}[M_-]$.

Define an action of the Lie algebra $\mathfrak{L}_{\mathcal{A}'}$ on the space ${\mathbb{C}}_{\mathcal{A}'}[G^*]/I_q$:
\begin{equation}\label{qmainactcl}
m\cdot (x+I_q) =\rho_{\chi^{s}_q}([m,x] ).
\end{equation}
where $x\in {\mathbb{C}}_{\mathcal{A}'}[G^*]$ is any representative of $x+I_q\in {\mathbb{C}}_{\mathcal{A}'}[G^*]/I_q$ and $m\in {\mathbb{C}}_{\mathcal{A}'}[M_-]$.
The algebra $W_q^s(G)$ can be regarded as the intersection of the space of invariants with respect to action (\ref{qmainactcl}) with the subspace $Q_{\mathcal{A}'}\subset {\mathbb{C}}_{\mathcal{A}'}[G^*]/I_q$.

Note also that since ${\chi_q^s}$ is a character of ${\mathbb{C}}_{\mathcal{A}'}[M_-]$ the ideal $I_q$ is stable under that action of ${\mathbb{C}}_{\mathcal{A}'}[M_-]$ on ${\mathbb{C}}_{\mathcal{A}'}[G^*]$ by commutators.

Denote by $\mathbb{C}_{\chi_q^{s}}$ the rank one representation of the algebra ${\mathbb{C}}_{\mathcal{A}'}[M_-]$ defined by the character $\chi_q^{s}$. Using the description of the algebra $W_q^s(G)$ in terms of action (\ref{qmainactcl}) and the isomorphism ${\mathbb{C}}_{\mathcal{A}'}[G^*]/I_q=
{\mathbb{C}}_{\mathcal{A}'}[G^*]\otimes_{{\mathbb{C}}_{\mathcal{A}'}[M_-]}\mathbb{C}_{\chi_q^{s}}$ one can also define
the algebra $W_q^s(G)$ as the intersection
$$
W_q^s(G)={\rm Hom}_{{\mathbb{C}}_{\mathcal{A}'}[M_-]}(\mathbb{C}_{\chi_q^{s}},
{\mathbb{C}}_{\mathcal{A}'}[G^*]\otimes_{{\mathbb{C}}_{\mathcal{A}'}[M_-]}\mathbb{C}_{\chi_q^{s}})\cap Q_{\mathcal{A}'}.
$$
Using Frobenius reciprocity we also have
$$
{\rm Hom}_{{\mathbb{C}}_{\mathcal{A}'}[M_-]}(\mathbb{C}_{\chi_q^{s}},
{\mathbb{C}}_{\mathcal{A}'}[G^*]\otimes_{{\mathbb{C}}_{\mathcal{A}'}[M_-]}\mathbb{C}_{\chi_q^{s}})={\rm End}_{{\mathbb{C}}_{\mathcal{A}'}[G^*]}({\mathbb{C}}_{\mathcal{A}'}[G^*]\otimes_{{\mathbb{C}}_{\mathcal{A}'}[M_-]}\mathbb{C}_{\chi_q^{s}}).
$$
Hence the algebra $W_q^s(G)$ acts on the space ${\mathbb{C}}_{\mathcal{A}'}[G^*]\otimes_{{\mathbb{C}}_{\mathcal{A}'}[M_-]}\mathbb{C}_{\chi_q^{s}}$ from the right by operators commuting with the natural left ${\mathbb{C}}_{\mathcal{A}'}[G^*]$--action on ${\mathbb{C}}_{\mathcal{A}'}[G^*]\otimes_{{\mathbb{C}}_{\mathcal{A}'}[M_-]}\mathbb{C}_{\chi_q^{s}}$. By the definition of $W_q^s(G)$ this action preserves $Q_{\mathcal{A}'}$ and by the above presented arguments it commutes with the natural left ${\mathbb{C}}_{\mathcal{A}'}[G_*]$--action on $Q_{\mathcal{A}'}$.

Thus $Q_{\mathcal{A}'}$ is a ${\mathbb{C}}_{\mathcal{A}'}[G_*]$--$W_q^s(G)$ bimodule equipped also with the adjoint action of $U_{\mathcal{A}'}^s(\m_+)$.  By (\ref{adm}) the adjoint action satisfies
\begin{equation}\label{adm4}
{\rm Ad}x(yv)={\rm Ad}x_2(y){\rm Ad}x_1(v),~x\in U_{\mathcal{A}'}^s(\m_+),~y\in {\mathbb{C}}_{\mathcal{A}'}[G_*],v\in Q_{\mathcal{A}'},
\end{equation}
and $\Delta_s(x)=x_1\otimes x_2$.

Denote by $v_0$ the image of the element $1 \in {\mathbb{C}}_{\mathcal{A}'}[G_*]$ in the quotient $Q_{\mathcal{A}'}$ under the canonical projection ${\mathbb{C}}_{\mathcal{A}'}[G_*]\rightarrow Q_{\mathcal{A}'}$. Obviously $v_0$ is the generating vector for $Q_{\mathcal{A}'}$ as a module over ${\mathbb{C}}_{\mathcal{A}'}[G_*]$.
Using formula (\ref{adm4}) and recalling that $Q_{\mathcal{A}'}$ is a ${\mathbb{C}}_{\mathcal{A}'}[G_*]$--$W_q^s(G)$ bimodule, for $x\in U_{\mathcal{A}'}^s(\m_+), y\in {\mathbb{C}}_{\mathcal{A}'}[G_*]$, and for a representative $w\in {\mathbb{C}}_{\varepsilon}[G_*]$ of an element $w+I_q\in W_q^s(G)$ we have
$$
{\rm Ad}x(wyv_0)={\rm Ad}x(ywv_0)={\rm Ad}x_2(y){\rm Ad}x_1(wv_0)=
$$
$$
={\rm Ad}x_2(y)\varepsilon(x_1)wv_0={\rm Ad}x(y)wv_0=w{\rm Ad}x(yv_0).
$$
Since $Q_{\mathcal{A}'}$ is generated by the vector $v_0$ over ${\mathbb{C}}_{\mathcal{A}'}[G_*]$ the last relation implies that the action of $W_q^s(G)$ on $Q_{\mathcal{A}'}$ commutes with the adjoint action.

We can summarize the results of the above discussion in the following proposition.
\begin{proposition}\label{Qpr}
The space $Q_{\mathcal{A}'}$ is naturally equipped with the structure of a left ${\mathbb{C}}_{\mathcal{A}'}[G_*]$--module, a right $U_{\mathcal{A}'}^s(\m_+)$--module via the adjoint action and a right $W_q^s(G)$--module in such a way that the left ${\mathbb{C}}_{\mathcal{A}'}[G_*]$--action and the right $U_{\mathcal{A}'}^s(\m_+)$--action commute with the right $W_q^s(G)$--action and compatibility condition (\ref{adm4}) is satisfied.
\end{proposition}

In conclusion we remark that by specializing $q$ to a particular value $\varepsilon\in \mathbb{C}$ such that $[r]_{\varepsilon_i}!\neq 0$, $\varepsilon\neq 0$, $i=1,\ldots ,l$, one can define a complex associative algebra
${\mathbb{C}}_\varepsilon[G_*]={\mathbb{C}}_{\mathcal{A}'}[G_*]/(q^{\frac{1}{2d}}-\varepsilon^{\frac{1}{2d}})
{\mathbb{C}}_{\mathcal{A}'}[G_*]$, its subalgebra ${\mathbb{C}}_\varepsilon[M_-]$ with a nontrivial character $\chi_\varepsilon^{s}$ and the corresponding W--algebra
\begin{equation}\label{eW}
W_\varepsilon^s(G)={\rm Hom}_{U_{\varepsilon}^s(\m_+)}(\mathbb{C}_{\varepsilon},Q_\varepsilon),
\end{equation}
where $\mathbb{C}_{\varepsilon}$ is the trivial representation of the algebra $U_{\varepsilon}^s(\m_+)$ induced by the counit, $Q_\varepsilon=Q_{\mathcal{A}'}/Q_{\mathcal{A}'}(q^{\frac{1}{2d}}-\varepsilon^{\frac{1}{2d}})$.

Obviously, for generic $\varepsilon$ we have $W_\varepsilon^s(G)=W_q^s(G)/(q^{\frac{1}{2d}}-\varepsilon^{\frac{1}{2d}})W_q^s(G)$.


\section{Poisson reduction and q-W algebras}\label{wpsred}

\setcounter{equation}{0}
\setcounter{theorem}{0}

In this section we shall analyze the
quasiclassical limit of
the algebra $W_q^s(G)$. Using results of Section 9 in \cite{S10}
we realize this limit algebra as the algebra of functions on a reduced
Poisson manifold. We assume again that the roots $\gamma_1, \ldots , \gamma_n$ are simple or that  the set $\gamma_1, \ldots , \gamma_n$ is empty.

Denote by $\chi^{s}$ the character of the Poisson subalgebra ${\mathbb{C}}[M_-]$ such that
$\chi^{s}(p(x))=\chi_q^{s}(x)~(\mbox{mod }~(q^{\frac{1}{2d}}-1))$ for every $x\in {\mathbb{C}}_{\mathcal{A}'}[M_-]$.

Note that under the projection $p:{\mathbb{C}}_{\mathcal{A}'}[G^*]\rightarrow {\mathbb{C}}_{\mathcal{A}'}[G^*]/(1-q^{\frac{1}{2d}}){\mathbb{C}}_{\mathcal{A}'}[G^*]$ and the canonical projection $U_{\mathcal{A}'}^s(\m_+)\rightarrow U_{\mathcal{A}'}^s(\m_+)/U_{\mathcal{A}'}^s(\m_+)(1-q^{\frac{1}{2d}})=U(\m_+)$ the right adjoint action of $U_{\mathcal{A}'}^s(\m_+)$ on ${\mathbb{C}}_{\mathcal{A}'}[G^*]$ induces the right infinitesimal dressing action of $U(\m_+)$ on ${\mathbb{C}}[G^*]$, and the image of the algebra ${\mathbb{C}}_{\mathcal{A}'}[G_*]$ under the projection $p$ is a certain subalgebra of ${\mathbb{C}}[G^*]$ that we denote by ${\mathbb{C}}[G_*]$. By definition we have ${\mathbb{C}}[G_*]\simeq {\mathbb{C}}[G]$ as algebras.
Let $I=p(I_q)$ be the ideal in ${\mathbb{C}}[G^*]$ generated by the kernel of $\chi^{s}$.
Then by the discussion after formula (\ref{qmainactcl}) the Poisson algebra $W^s(G)=W^s_q(G)/(q^{\frac{1}{2d}}-1)W^s_q(G)$ is the subspace of all $x+I\in Q_{1}$, $Q_{1}=Q_{\mathcal{A}'}/(1-q^{\frac{1}{2d}})Q_{\mathcal{A}'}\subset {\mathbb{C}}[G^*]/I$, such that $\{m,x\}\in I$ for any $m\in {\mathbb{C}}[M_-]$, and the Poisson bracket in $W^s(G)$ takes the form $\{(x+I),(y+I)\}=\{x,y\}+I$, $x+I,y+I\in W^s(G)$. We shall also write $W^s(G)=({\mathbb{C}}[G^*]/I)^{{\mathbb{C}}[M_-]}\cap Q_{1}$, where
the action of the Poisson algebra ${{\mathbb{C}}[M_-]}$ on the space ${\mathbb{C}}[G^*]/I$ is defined as follows
\begin{equation}\label{mainactcl}
x\cdot (v+I) =\rho_{\chi^{s}}(\{x,v\} ),
\end{equation}
$v\in {\mathbb{C}}[G^*]$ is any representative of $v+I\in {\mathbb{C}}[G^*]/I$ and $x\in {\mathbb{C}}[M_-]$.

We shall describe the space of invariants $({\mathbb{C}}[G^*]/I)^{{\mathbb{C}}[M_-]}$  with respect to this action by analyzing ``dual geometric objects''. First observe that algebra $({\mathbb{C}}[G^*]/I)^{{\mathbb{C}}[M_-]}$ is a particular
example of the
reduced Poisson algebra introduced in Lemma 8.1 in \cite{S10}.

Indeed, recall that according to (\ref{fact}) any element $(L_+,L_-)\in G^*$ may be uniquely written as
\begin{equation}\label{fact1}
(L_+,L_-)=(h_+,h_-)(n_+,n_-),
\end{equation}
where $n_\pm \in N_\pm$, $h_+=exp(({1 \over 1-s}P_{{\h'}}+\frac{1}{2}P_{{\h'}^\perp})x),~h_-=exp(({s \over 1-s}P_{{\h'}}-\frac{1}{2}P_{{\h'}^\perp})x),~x\in
{\frak h}$.

Formula (\ref{fact}) and decomposition of $N_-$ into products of one--dimensional subgroups corresponding to roots also imply that every element $L_-$ may be
represented in the form
\begin{equation}\label{lm}
\begin{array}{l}
L_- = exp\left[ \sum_{i=1}^lb_i({s \over 1-s}P_{{\h'}}-\frac{1}{2}P_{{\h'}^\perp})H_i\right]\times \\
\prod_{\beta}
exp[b_{\beta}X_{-\beta}],~b_i,b_\beta\in {\Bbb C},
\end{array}
\end{equation}
where the product over roots is taken in the same order as in the normal ordering associated to $s$.

Now define a map $\mu_{M_+}:G^* \rightarrow M_-$ by
\begin{equation}\label{mun}
\mu_{M_+}(L_+,L_-)=m_-,
\end{equation}
where for $L_-$ given by (\ref{lm}) $m_-$ is defined as follows
$$
m_-=\prod_{\beta\in \Delta_{\m_+}}
exp[b_{\beta}X_{-\beta}],
$$
and the product over roots is taken in the same order as in the normally ordered segment $\Delta_{\m_+}$.

By definition $\mu_{M_+}$ is a morphism of algebraic
varieties.
We also note that by definition ${\mathbb{C}}[M_-]=\{ \varphi\in {\mathbb{C}}[G^*]:\varphi=
\varphi(m_-)\}$. Therefore ${\mathbb{C}}[M_-]$ is generated by the pullbacks of
regular functions on $M_-$ with respect to the map $\mu_{M_+}$.
Since ${\mathbb{C}}[M_-]$ is a Poisson subalgebra in ${\mathbb{C}}[G^*]$, and  regular
functions
on $M_-$ are dense in $C^\infty(M_-)$ on every compact subset, we can equip the
manifold $M_-$ with
the Poisson structure in such a way that $\mu_{M_+}$ becomes a Poisson mapping.

Let $u$ be the element defined by
\begin{equation}\label{defu}
u=\prod_{i=1}^{l'}exp[t_{i} X_{-\gamma_i}]~ \in M_-,t_{i}=k_{i}~({\rm mod}~(q^{\frac{1}{2d}}-1)),
\end{equation}
where the product over roots is taken in the same order as in the normally ordered segment $\Delta_{\m_+}$.

Denote by $p:{\mathbb{C}}_{\mathcal{A}'}[G^*] \rightarrow {\mathbb{C}}_{\mathcal{A}'}[G^*]/(q^{\frac{1}{2d}}-1){\mathbb{C}}_{\mathcal{A}'}[G^*]={\mathbb{C}}[G^*]$ the
canonical projection. By Proposition \ref{quantreg} the elements $
L^{\pm,{V}}=(p\otimes p_V)({^q{L^{\pm,V}}})$ belong to the space $
{\mathbb{C}}[G^*]\otimes {\rm End}\overline{V}$, where $p_V:V\rightarrow \overline{V}=V/(q^{\frac{1}{2d}}-1)V$ is the projection of finite--dimensional $U_\mathcal{A}^{s}({\frak g})$--module $V$ onto the corresponding $\g$--module $\overline{V}$, and the map
$$
 {\mathbb{C}}_{\mathcal{A}'}[G^*]/(q^{\frac{1}{2d}}-1){\mathbb{C}}_{\mathcal{A}'}[G^*] \rightarrow {\mathbb{C}}[G^*],~
L^{\pm,V}\mapsto {L^{\pm,\overline{V}}}
$$
is an isomorphism.
In particular, from (\ref{rmatrspi}) it follows that
\begin{equation}\label{lv}
\begin{array}{l}
L^{-,\overline{V}}=(p\otimes id)exp\left[ \sum_{i=1}^lhH_i\otimes \pi_{\overline{V}}((-{2s \over 1-s}P_{{\h'}}+P_{{\h'}^\perp})Y_i)\right]\times \\
\prod_{\beta}
exp[p((1-q_\beta^{-2})e_{\beta}) \otimes
\pi_{\overline{V}}(X_{-\beta})].
\end{array}
\end{equation}

From (\ref{lv}) and the definition of $\chi^s$ we obtain that $\chi^s(\varphi)=\varphi (u)$
for every $\varphi \in {\mathbb{C}}[M_-]$. $\chi^s$ naturally extends to a character
of the
Poisson algebra $C^\infty(M_-)$.

Now applying Lemma 8.1 in \cite{S10} we can
define a
reduced Poisson algebra $C^\infty(\mu_{M_+}^{-1}(u))^{C^\infty(M_-)}$ as follows (see also
Remark 8.4 in \cite{S10}).
Denote by $I_u$ the ideal in $C^\infty(G^*)$ generated by elements
$\mu_{M_+}^*\psi,~\psi \in C^\infty(M_-),
~\psi(u)=0$. Let $P_u:C^\infty(G^*)\rightarrow
C^\infty(G^*)/I_u=C^\infty(\mu_{M_+}^{-1}(u))$ be the
canonical projection. Define the action of $C^\infty(M_-)$ on
$C^\infty(\mu_{M_+}^{-1}(u))$ by
\begin{equation}\label{actred}
\psi\cdot \varphi=P_u(\{ \mu_{M_+}^*\psi, \tilde \varphi\}),
\end{equation}
where $\psi \in C^\infty(M_-),~\varphi \in C^\infty(\mu_{M_+}^{-1}(u))$ and
$\tilde \varphi \in C^\infty(G^*)$
is a representative of $\varphi$ such that $P_u\tilde \varphi=\varphi$.
The reduced Poisson algebra $C^\infty(\mu_{M_+}^{-1}(u))^{C^\infty(M_-)}$ is the algebra of $C^\infty(M_-)$--invariants in $C^\infty(\mu_{M_+}^{-1}(u))$ with respect to action (\ref{actred}). The reduced Poisson algebra is naturally equipped with a Poisson structure induced from $C^\infty(G^*)$.

\begin{lemma}\label{redreg}
$\mu_{M_+}^{-1}(u)$ is a subvariety in $G^*$. Let $\overline{q(\mu_{M_+}^{-1}(u))}$  be the closure of $q(\mu_{M_+}^{-1}(u))$ in $G$  with respect to Zariski topology. Then the algebra
$W^s(G)$ is isomorphic to the algebra of regular functions on $\overline{q(\mu_{M_+}^{-1}(u))}$
 pullbacks of which under the map $q$ are invariant with respect to the action
(\ref{actred}) of
$C^\infty(M_-)$ on $C^\infty(\mu_{M_+}^{-1}(u))$, i.e.
$$
W^s(G)={\mathbb{C}}[\overline{q(\mu_{M_+}^{-1}(u))}]\cap
C^\infty(\mu_{M_+}^{-1}(u))^{C^\infty(M_-)},
$$
where ${\mathbb{C}}[\overline{q(\mu_{M_+}^{-1}(u))}]$ is regarded as a subalgebra in $C^\infty(\mu_{M_+}^{-1}(u))$ using the map $q^*:C^\infty(q(\mu_{M_+}^{-1}(u)))\rightarrow C^\infty(\mu_{M_+}^{-1}(u))$ and the imbedding ${\mathbb{C}}[\overline{q(\mu_{M_+}^{-1}(u))}]\subset C^\infty(q(\mu_{M_+}^{-1}(u)))$.
\end{lemma}
\begin{proof}
By definition $\mu_{M_+}^{-1}(u)$ is a subvariety in $G^*$. Next observe that
$I=
{\mathbb{C}}[G^*]\cap I_u$. Therefore by the definition of the algebra ${\mathbb{C}}[G^*]$ and of the map $\mu_{M_+}$ the quotient ${\mathbb{C}}[G^*]/I$ is
identified with the algebra of regular functions on $\mu_{M_+}^{-1}(u)$.

Since ${\mathbb{C}}[M_-]$ is dense in $C^\infty(M_-)$ on every compact subset in
$M_-$ we have:
$$
C^\infty(\mu_{M_+}^{-1}(u))^{C^\infty(M_-)}\cong
C^\infty(\mu_{M_+}^{-1}(u))^{{\mathbb{C}}[M_-]}.
$$

Finally observe that action (\ref{actred}) coincides with action
(\ref{mainactcl}) when restricted to
regular functions, and that the image of the map $q:G^*\rightarrow G$ is open in $G$; its closure coincides with $G$. Therefore by definition $Q_{1}={\mathbb{C}}[\overline{q(\mu_{M_+}^{-1}(u))}]$. Since $Q_1\subset {\mathbb{C}}[G^*]/I$ we have
$W^s(G)=({\mathbb{C}}[G^*]/I)^{{\mathbb{C}}[M_-]}\cap Q_{1}=C^\infty(\mu_{M_+}^{-1}(u))^{{\mathbb{C}}[M_-]}\cap Q_{1}=C^\infty(\mu_{M_+}^{-1}(u))^{C^\infty(M_-)}\cap {\mathbb{C}}[\overline{q(\mu_{M_+}^{-1}(u))}]$.
\end{proof}

We shall realize the
algebra $C^\infty(\mu_{M_+}^{-1}(u))^{C^\infty(M_-)}$ as the algebra of
functions on a reduced
Poisson manifold. In this construction we use
the dressing
action of the Poisson--Lie group $G$ on $G^*$ (see e.g. Proposition 8.2 in \cite{S10}).

Consider the restriction of the dressing action $G^*\times G \rightarrow G^*$ to
the subgroup $M_+\subset G$.
Let $G^*/M_+$ be the quotient of $G^*$ with respect to the dressing action of
$M_+$,
$\pi:G^* \rightarrow G^*/M_+$
the canonical projection. Note that the space $G^*/M_+$ is not a smooth
manifold. However, we will see that the subspace $\pi(\mu_{M_+}^{-1}(u))\subset
G^*/M_+$ is
a smooth manifold. By the results of Section 9 in \cite{S10}  $\mu_{M_+}^{-1}(u)$ is locally stable under the (locally defined) dressing action of $M_+$, and the algebra
$C^\infty(\pi(\mu_{M_+}^{-1}(u)))$
is isomorphic to $C^\infty(\mu_{M_+}^{-1}(u))^{C^\infty(M_-)}$.

Observe that using the map $q:G^*\rightarrow G$, $q(L_+,L_-)=L_-L_+^{-1}$ one can reduce the study of the dressing action to the study of the action of $G$ on itself by conjugations.
This simplifies many geometric problems. Consider the restriction
of this action to the subgroup $M_+$. Denote by $\pi_q:G\rightarrow G/M_+$
the canonical projection onto the quotient with respect to this action.

Next, following \cite{S10}, we explicitly describe the reduced space $\pi_q(\overline{q(\mu_{M_+}^{-1}(u))})$ and the algebra $W^s(G)$ assuming again that the roots $\gamma_1, \ldots , \gamma_n$ are simple or that the set $\gamma_1, \ldots , \gamma_n$ is empty.

First we describe the image of the ``level surface'' $\mu_{M_+}^{-1}(u)$ under the map $q$.
Let $X_\alpha(t)=\exp(tX_\alpha)\in G$, $t\in \mathbb{C}$ be the one--parametric subgroup in the algebraic group $G$ corresponding to root $\alpha\in \Delta$. Recall that for any $\alpha \in \Delta_+$ and any $t\neq 0$ the element $s_\alpha(t)=X_{-\alpha}(t)X_{\alpha}(-t^{-1})X_{-\alpha}(t)\in G$ is a representative for the reflection $s_\alpha$ corresponding to the root $\alpha$. Denote by $s\in G$ the following representative of the Weyl group element $s\in W$,
\begin{equation}\label{defrep}
s=s_{\gamma_1}(t_1)\ldots s_{\gamma_{l'}}(t_{l'}),
\end{equation}
where the numbers $t_{i}$ are defined in (\ref{defu}), and we assume that $t_i\neq 0$ for any $i$.

We shall also use the following representatives for $s^1$ and $s^2$
$$
s^1=s_{\gamma_1}(t_1)\ldots s_{\gamma_{n}}(t_{n}),~s^2=s_{\gamma_{n+1}}(t_{n+1})\ldots s_{\gamma_{l'}}(t_{l'}).
$$

Let $Z$ be the subgroup of $G$ generated by the semisimple part of the Levi subgroup $L$ and by the centralizer of $s$ in $H$. Denote by $N$ the subgroup of $G$ corresponding to the Lie subalgebra $\n$ and by $\overline{N}$ the opposite unipotent subgroup in $G$ with the Lie algebra $\overline{\n}$. By definition we have that $N_+\subset ZN$.

\begin{proposition}{\bf (\cite{S11}, Proposition 7.2)}\label{constrt}
Let $q:G^*\rightarrow G$ be the map defined by
$$
q(L_+,L_-)=L_-L_+^{-1}.
$$
Suppose that the numbers $t_{i}$ defined in (\ref{defu}) are not equal to zero for all $i$. Then $q(\mu_{M_+}^{-1}(u))$ is a subvariety in $NsZN$ and the closure $\overline{q(\mu_{M_+}^{-1}(u))}$ of $q(\mu_{M_+}^{-1}(u))$ with respect to Zariski topology is also contained in $NsZN$.
\end{proposition}

\begin{proposition}\label{crosssect}
{\bf (\cite{S6}, Propositions 2.1 and 2.2)}
Let $N_s=\{ v \in N|svs^{-1}\in \overline{N} \}$.
Then the conjugation map
\begin{equation}\label{cross}
N\times sZN_s\rightarrow NsZN
\end{equation}
is an isomorphism of varieties. Moreover, the variety $sZN_s$ is a transversal slice to the set of conjugacy classes in $G$.
\end{proposition}

\begin{theorem}\label{var}
Suppose that the numbers $t_{i}$ defined in (\ref{defu}) are not equal to zero for all $i$. Then $\overline{q(\mu_{M_+}^{-1}(u))}$ is invariant under conjugations by elements of $M_+$, the conjugation action of $M_+$ on $\overline{q(\mu_{M_+}^{-1}(u))}$ is free, the quotient $\pi_q(\overline{q(\mu_{M_+}^{-1}(u))})$ is a smooth variety isomorphic to $sZN_s$, $\overline{q(\mu_{M_+}^{-1}(u))}\simeq M_+ \times \pi_q(\overline{q(\mu_{M_+}^{-1}(u))})\simeq M_+ \times sZN_s$, and the algebra ${\mathbb{C}}[\overline{q(\mu_{M_+}^{-1}(u))}]$ is isomorphic to ${\mathbb{C}}[M_+] \otimes {\mathbb{C}}[sZN_s]$.

Assume also that the roots $\gamma_1, \ldots , \gamma_n$ are simple or the set $\gamma_1, \ldots , \gamma_n$ is empty.
Then the Poisson algebra $W^s(G)$ is isomorphic to the Poisson algebra of regular
functions on  \\ $\pi_q(\overline{q(\mu_{M_+}^{-1}(u))})$, $W^s(G)=\mathbb{C}[\pi_q(\overline{q(\mu_{M_+}^{-1}(u))})]=\mathbb{C}[sZN_s]$. Thus the algebra $W_q^s(G)$ is a noncommutative deformation of the algebra of regular functions on the transversal slice $sZN_s$.
\end{theorem}
\begin{proof}
As we observed above  $\mu_{M_+}^{-1}(u)$ is locally stable under the (locally defined) dressing action of $M_+$, and hence
$q(\mu_{M_+}^{-1}(u))\subset NsZN$ is (locally) stable under the action of $M_+\subset N$ on $NsZN$ by conjugations. Since the conjugation action of $N$ on $NsZN$ is free the (locally defined) conjugation action of $M_+$ on $q(\mu_{M_+}^{-1}(u))$ is (locally) free as well.

Now observe that by Proposition \ref{constrt} $\overline{q(\mu_{M_+}^{-1}(u))}\subset NsZN$.  Since the conjugation action of $N$ on $NsZN$ is free and regular, and  $\overline{q(\mu_{M_+}^{-1}(u))}$ is closed, the induced action of $M_+$ on $\overline{q(\mu_{M_+}^{-1}(u))}$ is globally defined and is free as well.  Therefore the quotient $\pi_q(\overline{q(\mu_{M_+}^{-1}(u))})$ is a smooth variety.

By Theorem \ref{crosssect} $\pi_q(\overline{q(\mu_{M_+}^{-1}(u))})$ can be regarded as a closed subvariety of the closed variety $sZN_s$.  From formula (\ref{dimm}) for the cardinality $\sharp \Delta_{\m_+}$ of the set $\Delta_{\m_+}$ and from the definition of $\overline{q(\mu_{M_+}^{-1}(u))}$ we deduce that the dimension of the quotient $\pi_q(\overline{q(\mu_{M_+}^{-1}(u))})$ is equal to the dimension of the variety $sZN_s$,
\begin{eqnarray*}
{\rm dim}~\pi_q(\overline{q(\mu_{M_+}^{-1}(u))})={\rm dim}~G-2{\rm dim}~M_+=2D+l-2\sharp \Delta_{\m_+} =2D+l- \\ -2(D-\frac{l(s)-l'}{2}-D_0)=l(s)+2D_0+l-l'={\rm dim}~N_s+{\rm dim}~Z={\rm dim}~sZN_s.
\end{eqnarray*}
Therefore $\pi_q(\overline{q(\mu_{M_+}^{-1}(u))})\simeq sZN_s$, and ${\mathbb{C}}[\overline{q(\mu_{M_+}^{-1}(u))}]\cong {\mathbb{C}}[M_+] \otimes {\mathbb{C}}[sZN_s]$ because the action of $M_+$ on $\overline{q(\mu_{M_+}^{-1}(u))}$ is free, and hence $\overline{q(\mu_{M_+}^{-1}(u))}\simeq M_+ \times \pi_q(\overline{q(\mu_{M_+}^{-1}(u))})\simeq M_+ \times sZN_s$.

Now observe that by Remark 9.5 in \cite{S10} the map
$$
C^\infty(\pi(\mu_{M_+}^{-1}(u)))\rightarrow
C^\infty(\mu_{M_+}^{-1}(u))^{C^\infty(M_-)},~~\psi \mapsto \pi^*\psi
$$
is an isomorphism.
By construction the map $\pi_q: \overline{q(\mu_{M_+}^{-1}(u))}\rightarrow
\pi_q\overline{q(\mu_{M_+}^{-1}(u))}$
is a morphism of varieties.
Therefore the map
$$
\mathbb{C}[\pi_q\overline{q(\mu_{M_+}^{-1}(u))}]\rightarrow {\mathbb{C}}[\overline{q(\mu_{M_+}^{-1}(u))}]\cap
C^\infty(\mu_{M_+}^{-1}(u))^{C^\infty(M_-)},~~\psi \mapsto \pi_q^*\psi
$$
is an isomorphism, where ${\mathbb{C}}[\overline{q(\mu_{M_+}^{-1}(u))}]$ is regarded as a subalgebra in $C^\infty(\mu_{M_+}^{-1}(u))$ using the map $q^*:C^\infty(q(\mu_{M_+}^{-1}(u)))\rightarrow C^\infty(\mu_{M_+}^{-1}(u))$ and the imbedding ${\mathbb{C}}[\overline{q(\mu_{M_+}^{-1}(u))}]\subset C^\infty(q(\mu_{M_+}^{-1}(u)))$.

Finally observe that by Lemma \ref{redreg} the algebra ${\mathbb{C}}[\overline{q(\mu_{M_+}^{-1}(u))}]\cap
C^\infty(\mu_{M_+}^{-1}(u))^{C^\infty(M_-)}$ is isomorphic to $W^s(G)$, and hence $W^s(G)\cong \mathbb{C}[\pi_q\overline{q(\mu_{M_+}^{-1}(u))}]\simeq {\mathbb{C}}[sZN_s]$.
This completes the proof.

\end{proof}

A similar theorem can be proved in case when the roots  $\gamma_{n+1}, \ldots , \gamma_{l'}$ are simple or the set $\gamma_{n+1}, \ldots , \gamma_{l'}$ is empty. In that case instead of the map $q:G^*\rightarrow G$ one should use another map $q':G^*\rightarrow G$, $q'(L_+,L_-)=L_-^{-1}L_+$ which has the same properties as $q$, see \cite{dual}, Section 2.

Theorem \ref{var} implies that the algebra $W^s(G)$ coincides with the deformed Poisson W--algebra introduced in \cite{S6}.

In conclusion we discuss a simple property of the algebra $W_\varepsilon^s(G)$ which allows to construct noncommutative deformations of coordinate rings of singularities arising in the fibers of the conjugation quotient map $\delta_G: G \rightarrow H/W$ generated by the inclusion $\mathbb{C}[H]^W\simeq
\mathbb{C}[G]^G\hookrightarrow \mathbb{C}[G]$, where $H$ is the maximal torus of $G$ corresponding to the Cartan subalgebra
$\h$ and $W$ is the Weyl group of the pair $(G, H)$.

Observe that each central element $z\in Z({\mathbb{C}}_\varepsilon[G_*])$ obviously gives rise to an element $\rho_{\chi_\varepsilon^s}(z)\in Q_\varepsilon$, and since $z$ is central
\begin{eqnarray*}
\rho_{\chi_\varepsilon^s}(z)\in {\rm Hom}_{{\mathbb{C}}_\varepsilon[M_-]}(\mathbb{C}_{\chi_\varepsilon^{s}},{\mathbb{C}}_\varepsilon[G_*]\otimes_{{\mathbb{C}}_\varepsilon[M_-]}
\mathbb{C}_{\chi_\varepsilon^{s}})\bigcap Q_\varepsilon=W_\varepsilon^s(G).
\end{eqnarray*}

The proof of the following proposition is similar to that of Theorem $\rm A_h$ in \cite{S7}.

\begin{proposition}\label{ZWimb}
Let $\varepsilon\in \mathbb{C}$ be generic. Then
the restriction of the linear map $\rho_{\chi_\varepsilon^s}:{\mathbb{C}}_\varepsilon[G_*]\rightarrow Q_\varepsilon$ to the center $Z({\mathbb{C}}_\varepsilon[G_*])$ of ${\mathbb{C}}_\varepsilon[G_*]$ gives rise to an injective homomorphism of algebras,
$$
\rho_{\chi_\varepsilon^s}:Z({\mathbb{C}}_\varepsilon[G_*])\rightarrow W_\varepsilon^s(G).
$$
\end{proposition}

Now if $\eta:Z({\mathbb{C}}_\varepsilon[G_*])\rightarrow \mathbb{C}$ is a character then from Theorem \ref{var} and the results of Section 6 in \cite{S6} it follows that the algebra $W_\varepsilon^s(G)/W_\varepsilon^s(G){\rm ker}~\eta$ can be regarded as a noncommutative deformation of the algebra of regular functions defined on a fiber of the conjugation quotient map $\delta_G: sZN_s \rightarrow H/W$. In particular, for singular fibers we obtain noncommutative deformations of the coordinate rings of the corresponding singularities.


\section{Skryabin equivalence for equivariant modules over quantum groups}\label{skryabin}

\setcounter{equation}{0}
\setcounter{theorem}{0}

In this section we establish a remarkable equivalence between the category of $W_\varepsilon^s(G)$--modules and a certain category of ${\mathbb{C}}_{\varepsilon}[G_*]$ modules. This equivalence is a quantum group counterpart of Skryabin equivalence established in the Appendix to \cite{Pr}.

Let $J={\rm Ker}~\varepsilon|_{U_{\mathcal{A}'}^s(\m_+)}$ be the augmentation ideal of $U_{\mathcal{A}'}^s(\m_+)$ related to the counit $\varepsilon$ of $U_{\mathcal{A}'}^s(\g)$, and
$\mathbb{C}_{\mathcal{A}'}$ the trivial representation of $U_{\mathcal{A}'}^s(\m_+)$ given by the counit.
Let $V$ be a finitely generated ${\mathbb{C}}_{\mathcal{A}'}[G_*]$--module which satisfies the following conditions:
\begin{enumerate}
\item
$V$ is free as an ${\mathcal{A}'}$--module.

\item
$V$ is a right $U_{\mathcal{A}'}^s(\m_+)$--module with respect to an action ${\rm Ad}$ such that the action of the augmentation ideal $J$ on $V$ is locally nilpotent.

\item
The following compatibility condition holds for the two actions
\begin{equation}\label{compat}
{\rm Ad}x(yv)={\rm Ad}x_2(y){\rm Ad}x_1(v),~x\in U_{\mathcal{A}'}^s(\m_+),~y\in {\mathbb{C}}_{\mathcal{A}'}[G_*],~v\in V,
\end{equation}
where $\Delta_s(x)=x_1\otimes x_2$, ${\rm Ad}x(y)$ is the adjoint action of $x\in U_{\mathcal{A}'}^s(\m_+)$ on $y\in {\mathbb{C}}_{\mathcal{A}'}[G_*]$.

An element $v\in V$ is called a Whittaker vector if
${\rm Ad}xv=\varepsilon(x)v$ for any $x\in U_{\mathcal{A}'}^s(\m_+)$.
The space
\begin{equation}\label{winv}
{\rm Hom}_{U_{\mathcal{A}'}^s(\m_+)}(\mathbb{C}_{\mathcal{A}'},V)={\rm Wh}(V).
\end{equation}
is called the space of Whittaker vectors of $V$.

Consider the induced $U_{\mathcal{A}'}^s(\g)$--module $W=U_{\mathcal{A}'}^s(\g)\otimes_{{\mathbb{C}}_{\mathcal{A}'}[G_*]}V$. Using the adjoint action of $U_{\mathcal{A}'}^s(\g)$ on itself one can naturally extend the adjoint action of $U_{\mathcal{A}'}^s(\m_+)$ from $V$ to $W$ in such a way that compatibility condition $(\ref{compat})$ is satisfied for the natural action of $U_{\mathcal{A}'}^s(\g)$ and the adjoint action $\rm Ad$ of $U_{\mathcal{A}'}^s(\m_+)$ on $W$. As we observed in Section \ref{qplproups} (see formula (\ref{cm1})) $\Delta_s^{opp}(U_{\mathcal{A}'}^s(\m_+))\subset U_{\mathcal{A}'}^s(\b_+)\otimes U_{\mathcal{A}'}^s(\m_+)$. From this using the fact that the elements $\widetilde{e}_\beta=(1-q_\beta^{-2})e_\beta$, $e_\beta\in U_{\mathcal{A}'}^s(\m_+)$, are generators of ${\mathbb{C}}_{\mathcal{A}'}[M_-]$ one immediately deduces that $\Delta_s^{opp}({\mathbb{C}}_{\mathcal{A}'}[M_-])\subset U_{\mathcal{A}'}^s(\b_+)\otimes {\mathbb{C}}_{\mathcal{A}'}[M_-]$. In fact $\Delta_s^{opp}({\mathbb{C}}_{\mathcal{A}'}[M_-])\subset {\mathbb{C}}_{\mathcal{A}'}[B_-]\otimes {\mathbb{C}}_{\mathcal{A}'}[M_-]$ since ${\mathbb{C}}_{\mathcal{A}'}[M_-]\subset {\mathbb{C}}_{\mathcal{A}'}[B_-]$ which is a Hopf algebra.

We shall require that
\item
For any $x\in {\mathbb{C}}_{\mathcal{A}'}[M_-]$ the natural action of the element $(S^{-1}\otimes\chi_q^s)\Delta^{opp}(x)\in {\mathbb{C}}_{\mathcal{A}'}[G^*]$ on $W$ coincides with the adjoint action ${\rm Ad}x$ of $x$ on $W$.

As in the second part of the proof of Proposition \ref{9.2} one can see that the last condition implies that
for any $z\in {\mathbb{C}}_{\mathcal{A}'}[G_*]\cap I_q$ and $v\in {\rm Wh}(V)$ $zv=0$.
\end{enumerate}
Denote by ${{\mathbb{C}}_{\mathcal{A}'}[G_*]-{\rm mod}_{U_{\mathcal{A}'}^s(\m_+)}^{\chi^{s}_q}}_{loc}$ the category of finitely generated ${\mathbb{C}}_{\mathcal{A}'}[G_*]$--module which satisfy conditions 1--4. Morphisms in the category ${{\mathbb{C}}_{\mathcal{A}'}[G_*]-{\rm mod}_{U_{\mathcal{A}'}^s(\m_+)}^{\chi^{s}_q}}_{loc}$ are ${\mathbb{C}}_{\mathcal{A}'}[G_*]$- and $U_{\mathcal{A}'}^s(\m_+)$--module homomorphisms.
We call ${{\mathbb{C}}_{\mathcal{A}'}[G_*]-{\rm mod}_{U_{\mathcal{A}'}^s(\m_+)}^{\chi^{s}_q}}_{loc}$ the category of $(U_{\mathcal{A}'}^s(\m_+),\chi^{s}_q)$--equivariant modules over ${\mathbb{C}}_{\mathcal{A}'}[G_*]$.

Note that the algebra $W_q^s(G)$ naturally acts in the space of Whittaker vectors for any object $V$ of the category ${{\mathbb{C}}_{\mathcal{A}'}[G_*]-{\rm mod}_{U_{\mathcal{A}'}^s(\m_+)}^{\chi^{s}_q}}_{loc}$. Indeed, if $w,w'\in {\mathbb{C}}_{\mathcal{A}'}[G_*]$ are two representatives of an element from $W_q^s(G)$ then $w-w'\in {\mathbb{C}}_{\mathcal{A}'}[G_*]\cap I_q$, and hence for any $v\in {\rm Wh}(V)$ $wv=w'v$. Moreover, by the definition of the algebra $W_q^s(G)$ and by condition (\ref{compat}) we have
$$
{\rm Ad}x(wv)={\rm Ad}x_2(w){\rm Ad}x_1(v)={\rm Ad}x_2(w)\varepsilon(x_1)v={\rm Ad}x(w)v=\varepsilon(x)wv.
$$
Therefore $wv$ is a Whittaker vector independent of the choice of the representative $w$.

\begin{proposition}\label{Qprop}
For any  finitely generated $W_q^s(G)$--module $E$ which is free as an  ${\mathcal{A}'}$--module the space $Q_{\mathcal{A}'}\otimes_{W_q^s(G)}E$ is an object in ${{\mathbb{C}}_{\mathcal{A}'}[G_*]-{\rm mod}_{U_{\mathcal{A}'}^s(\m_+)}^{\chi^{s}_q}}_{loc}$, and
$$
{\rm Wh}(Q_{\mathcal{A}'}\otimes_{W_q^s(G)}E)={\rm Hom}_{U_{\mathcal{A}'}^s(\m_+)}(\mathbb{C}_{\mathcal{A}'},Q_{\mathcal{A}'}\otimes_{W_q^s(G)}E)=E.
$$
\end{proposition}

\begin{proof}
First we prove that
$Q_{\mathcal{A}'}$ is an object in ${{\mathbb{C}}_{\mathcal{A}'}[G_*]-{\rm mod}_{U_{\mathcal{A}'}^s(\m_+)}^{\chi^{s}_q}}_{loc}$.
We shall prove that the  adjoint action of the augmentation ideal $J$ of $U_{\mathcal{A}'}^s(\m_+)$ on $Q_{\mathcal{A}'}$ is locally nilpotent. All the other properties of objects of the category ${{\mathbb{C}}_{\mathcal{A}'}[G_*]-{\rm mod}_{U_{\mathcal{A}'}^s(\m_+)}^{\chi^{s}_q}}_{loc}$ for $Q_{\mathcal{A}'}$ were already established in Proposition \ref{Qpr}.

Indeed, let ${\rm hom}_{\mathcal{A}'}(U_{\mathcal{A}'}^s(\m_+)), W_q^s(G))$ be the subspace in ${\rm Hom}_{\mathcal{A}'}(U_{\mathcal{A}'}(\m_+),W_q^s(G))$ which consists of the linear maps vanishing on some power of the augmentation ideal $J={\rm Ker }~\varepsilon$ of $U_{\mathcal{A}'}(\m_+)$, ${\rm hom}_{\mathcal{A}'}(U_{\mathcal{A}'}(\m_+),W_q^s(G))=\{f\in {\rm Hom}_{\mathcal{A}'}(U_{\mathcal{A}'}(\m_+),W_q^s(G)):f(J^n)=0~{\rm for~some}~n>0\}$.
Fix any linear map $\rho: Q_{\mathcal{A}'}\rightarrow W_q^s(G)\subset Q_{\mathcal{A}'}$ the restriction of which to $W_q^s(G)$ is the identity map, and let for any $v\in Q_{\mathcal{A}'}$ $\sigma(v):U_{\mathcal{A}'}^s(\m_+)\rightarrow W_q^s(G)$ be the $\mathcal{A}'$--linear homomorphism given by $\sigma(v)(x)=\rho({\rm Ad}x(v))$.
Since the adjoint action of $U_{\mathcal{A}'}(\m_+)$ on ${\mathbb{C}}_{\mathcal{A}'}[G_*]$ is locally finite the induced adjoint action of $U_{\mathcal{A}'}(\m_+)$ on $Q_{\mathcal{A}'}$ is locally finite as well (see the arguments in the end of the proof of Proposition \ref{locfin}). Therefore for any $v\in Q_{\mathcal{A}'}$ the space ${\rm Ad}U_{\mathcal{A}'}(\m_+)(v)$ has finite rank over $\mathcal{A}'$. This implies that in fact $\sigma(v)\in {\rm hom}_{\mathcal{A}'}(U_{\mathcal{A}'}^s(\m_+), W_q^s(G))$, and we have a map $\sigma:Q_{\mathcal{A}'}\rightarrow {\rm hom}_{\mathcal{A}'}(U_{\mathcal{A}'}^s(\m_+), W_q^s(G))$.

By definition $\sigma$ is a homomorphism of right $U_{\mathcal{A}'}^s(\m_+)$--modules, where the right action of $U_{\mathcal{A}'}^s(\m_+)$ on ${\rm hom}_{\mathcal{A}'}(U_{\mathcal{A}'}^s(\m_+), W_q^s(G))$ is induced by multiplication in $U_{\mathcal{A}'}^s(\m_+)$ from the left.

We claim that $\sigma$ is injective. Indeed, consider the specialization $\sigma_1$ of the homomorphism $\sigma$ at $q=1$.
The specialization of the algebra $U_{\mathcal{A}'}(\m_+)$ at $q=1$ is isomorphic to $U(\m_+)$, and the specialization $Q_{1}=Q_{\mathcal{A}'}/ (q^{\frac{1}{2d}}-1)Q_{\mathcal{A}'}$ of the $U_{\mathcal{A}'}(\m_+)$--module $Q_{\mathcal{A}'}$ at $q=1$ is isomorphic to ${\mathbb{C}}[\overline{q(\mu_{M_+}^{-1}(u))}]$. By Theorem \ref{var} ${\mathbb{C}}[\overline{q(\mu_{M_+}^{-1}(u))}]\cong {\mathbb{C}}[M_+]\otimes W^s(G)$. From Proposition 11.2 in \cite{S10} we obtain that the induced action of $U(\m_+)$ on the corresponding variety $\overline{q(\mu_{M_+}^{-1}(u))}$ is induced by the conjugation action of $M_+$ and now using Proposition \ref{crosssect} one immediately deduces that the induced action of $U(\m_+)$ on  ${\mathbb{C}}[M_+]\otimes W^s(G)$ is generated by the action of $U(\m_+)$ on ${\mathbb{C}}[M_+]$ by left invariant differential operators.

Using the exponential map $\exp: \m_+\rightarrow M_+$ we can also identify ${\mathbb{C}}[M_+]\otimes W^s(G)$ with the right $U(\m_+)$--module ${\rm hom}_{\mathbb{C}}(U(\m_+),W^s(G))=\{f\in {\rm Hom}_{\mathbb{C}}(U(\m_+),W^s(G)):f(J_1^n)=0~{\rm for~some}~n>0\}$, where $J_1$ is the augmentation ideal of $U(\m_+)$ generated by $\m_+$, and the right action of $U(\m_+)$ on ${\rm hom}_{\mathbb{C}}(U(\m_+), W^s(G))$ is induced by multiplication in $U(\m_+)$ from the left.

On the other hand the specialization of ${\rm hom}_{\mathcal{A}'}(U_{\mathcal{A}'}^s(\m_+), W_q^s(G))$ at $q=1$ is also isomorphic to ${\rm hom}_{\mathbb{C}}(U(\m_+), W^s(G))$, and hence under the above identifications the specialization of $\sigma_1$ of map $\sigma$ at $q=1$ becomes  the identity map.

Now let $W$ be the kernel of $\sigma$, and $W_1\subset Q_{1}$ its image under the canonical projection $Q_{\mathcal{A}'}\rightarrow Q_{1}=Q_{\mathcal{A}'}/ (q^{\frac{1}{2d}}-1)Q_{\mathcal{A}'}$.
$W_1$ must be contained in the kernel of $\sigma_1$. Since this kernel is trivial $W_1$ must be trivial as well, and hence $W=(q^{\frac{1}{2d}}-1)W'$, $W'\subset Q_{\mathcal{A}'}$. Since $Q_{\mathcal{A}'}$ is $\mathcal{A}'$--free and $\mathcal{A}'$ has no zero divisors we also have $W'\subset W$. Iterating this process we deduce that any element $w\in W$ can be represented in the form $w=(q^{\frac{1}{2d}}-1)^{B}w', w'\in W$ with arbitrary large $B\in \mathbb{N}$ which is possible only in case when $W=0$. Therefore $\sigma$ is injective.

Thus $Q_{\mathcal{A}'}$ is a submodule of ${\rm hom}_{\mathcal{A}'}(U_{\mathcal{A}'}^s(\m_+), W_q^s(G))$ the action of $J$ on which is locally nilpotent. Therefore the action of $J$ on $Q_{\mathcal{A}'}$ is locally nilpotent as well.

We conclude that for any  finitely generated $W_q^s(G)$--module $E$ which is free as an  ${\mathcal{A}'}$--module the space $Q_{\mathcal{A}'}\otimes_{W_q^s(G)}E$ can be equipped with the adjoint action induced by the adjoint action on $Q_{\mathcal{A}'}$ in such a way that the compatibility condition (\ref{compat}) is satisfied. Since the adjoint action of the augmentation ideal $J$ on $Q_{\mathcal{A}'}$ is locally nilpotent the induced adjoint action of the augmentation ideal $J$ on $Q_{\mathcal{A}'}\otimes_{W_q^s(G)}E$ is locally nilpotent as well.

The fact that $Q_{\mathcal{A}'}$ is an object of the category ${{\mathbb{C}}_{\mathcal{A}'}[G_*]-{\rm mod}_{U_{\mathcal{A}'}^s(\m_+)}^{\chi^{s}_q}}_{loc}$ implies now that $Q_{\mathcal{A}'}\otimes_{W_q^s(G)}E$ is an object of the category ${{\mathbb{C}}_{\mathcal{A}'}[G_*]-{\rm mod}_{U_{\mathcal{A}'}^s(\m_+)}^{\chi^{s}_q}}_{loc}$ as well.
Moreover, by the definition of the algebra $W_q^s(G)$
\begin{equation}\label{adw}
{\rm Hom}_{U_{\mathcal{A}'}^s(\m_+)}(\mathbb{C}_{\mathcal{A}'},Q_{\mathcal{A}'}\otimes_{W_q^s(G)}E)={\rm Wh}(Q_{\mathcal{A}'}\otimes_{W_q^s(G)}E)={W_q^s(G)}\otimes_{W_q^s(G)}E=E.
\end{equation}

This completes the proof of the fact that $Q_{\mathcal{A}'}$ and $Q_{\mathcal{A}'}\otimes_{W_q^s(G)}E$ are objects of the category ${{\mathbb{C}}_{\mathcal{A}'}[G_*]-{\rm mod}_{U_{\mathcal{A}'}^s(\m_+)}^{\chi^{s}_q}}_{loc}$.

\end{proof}

Obviously we also have that for any object $V$ of the category ${{\mathbb{C}}_{\mathcal{A}'}[G_*]-{\rm mod}_{U_{\mathcal{A}'}^s(\m_+)}^{\chi^{s}_q}}_{loc}$ the canonical map $Q_{\mathcal{A}'}\otimes_{W_q^s(G)}{\rm Wh}(V)\rightarrow V$ is a morphism in the category ${{\mathbb{C}}_{\mathcal{A}'}[G_*]-{\rm mod}_{U_{\mathcal{A}'}^s(\m_+)}^{\chi^{s}_q}}_{loc}$.

We also denote by ${{\mathbb{C}}_{\varepsilon}[G_*]-{\rm mod}_{U_{\varepsilon}^s(\m_+)}^{\chi^{s}_\varepsilon}}_{loc}$ the category of ${\mathbb{C}}_{\varepsilon}[G_*]$--modules which are specializations of modules from ${{\mathbb{C}}_{\mathcal{A}'}[G_*]-{\rm mod}_{U_{\mathcal{A}'}^s(\m_+)}^{\chi^{s}_q}}_{loc}$ at $q=\varepsilon\in \mathbb{C}$. The spaces of Whittaker vectors for modules from ${{\mathbb{C}}_{\varepsilon}[G_*]-{\rm mod}_{U_{\varepsilon}^s(\m_+)}^{\chi^{s}_\varepsilon}}_{loc}$, the adjoint action and the canonical map $Q_{\varepsilon}\otimes_{W_\varepsilon^s(G)}{\rm Wh}(V)\rightarrow V$, $V\in {{\mathbb{C}}_{\varepsilon}[G_*]-{\rm mod}_{U_{\varepsilon}^s(\m_+)}^{\chi^{s}_\varepsilon}}_{loc}$ are defined similarly to the case of modules from ${{\mathbb{C}}_{\mathcal{A}'}[G_*]-{\rm mod}_{U_{\mathcal{A}'}^s(\m_+)}^{\chi^{s}_q}}_{loc}$.

We have the following obvious $\varepsilon$--specialization of Proposition \ref{Qprop}.
\begin{proposition}\label{Qprop1}
Let $\varepsilon\in \mathbb{C}$ be generic. Then
for any  finitely generated $W_\varepsilon^s(G)$--module $E$ the space $Q_{\varepsilon}\otimes_{W_\varepsilon^s(G)}E$ is an object in ${{\mathbb{C}}_{\varepsilon}[G_*]-{\rm mod}_{U_{\varepsilon}^s(\m_+)}^{\chi^{s}_\varepsilon}}_{loc}$, and
$$
{\rm Wh}(Q_{\varepsilon}\otimes_{W_\varepsilon^s(G)}E)={\rm Hom}_{U_{\varepsilon}^s(\m_+)}(\mathbb{C}_{\varepsilon},Q_{\varepsilon}\otimes_{W_\varepsilon^s(G)}E)=E,
$$
where $\mathbb{C}_{\varepsilon}$ is the trivial representation of $U_{\varepsilon}^s(\m_+)$ given by the counit.
\end{proposition}

The following proposition is crucial for the proof of the main statement of this paper.
\begin{proposition}
Assume that the roots $\gamma_1,\ldots, \gamma_n$ (or $\gamma_{n+1},\ldots, \gamma_{l'}$) are simple or one of the sets $\gamma_1,\ldots, \gamma_n$ or $\gamma_{n+1},\ldots, \gamma_{l'}$ is empty. Suppose also that the numbers $t_{i}$ defined in (\ref{defu}) are not equal to zero for all $i$. Then for generic $\varepsilon\in \mathbb{C}$
$Q_\varepsilon$ is isomorphic to ${\rm hom}_{\mathbb{C}}(U_\varepsilon^s(\m_+), \mathbb{C}) \otimes W^s_\varepsilon(G)$ as a $U_\varepsilon^s(\m_+)$--$W^s_\varepsilon(G)$--bimodule, where ${\rm hom}_{\mathbb{C}}(U_{\varepsilon}^s(\m_+), \mathbb{C})$ is the subspace in ${\rm Hom}_{\mathbb{C}}(U_{\varepsilon}^s(\m_+),\mathbb{C})$ which consists of the linear maps vanishing on some power of the augmentation ideal $J={\rm Ker }~\varepsilon$ (here $\varepsilon$ is the counit of $U_{\varepsilon}^s(\g)$) of $U_{\varepsilon}^s(\m_+)$, ${\rm hom}_{\mathbb{C}}(U_{\varepsilon}^s(\m_+),\mathbb{C})=\{f\in {\rm Hom}_{\mathbb{C}}(U_{\varepsilon}^s(\m_+),\mathbb{C}):f(J^n)=0~{\rm for~some}~n>0\}$.
\end{proposition}

\begin{proof}
First we show that the specialization $\sigma_{\varepsilon}:Q_{\varepsilon}\rightarrow {\rm hom}_{\mathbb{C}}(U_{\varepsilon}^s(\m_+), W_\varepsilon^s(G))$ at $q=\varepsilon$ of the $U_{\mathcal{A}'}^s(\m_+)$--module homomorphism $\sigma:Q_{\mathcal{A}'}\rightarrow {\rm hom}_{\mathcal{A}'}(U_{\mathcal{A}'}^s(\m_+), W_q^s(G))$ constructed in the proof of Proposition \ref{Qprop} is an isomorphism of right $U_{\varepsilon}^s(\m_+)$--modules.

First we prove that $\sigma_\varepsilon$ is injective. The proof will be based on the following lemma that will be also used later.

\begin{lemma}\label{inj}
Let $\phi:X\rightarrow Y$ be a homomorphism of $U_{\varepsilon}^s(\m_+)$--modules. Denote by ${\rm Wh}(X)$ the subspace of Whittaker vectors of $X$, i.e. the subspace of $X$ which consists of elements $v$ such that $xv=\varepsilon(x)v$, $x\in U_{\varepsilon}^s(\m_+)$. Assume that the action of the augmentation ideal of $U_{\varepsilon}^s(\m_+)$ on $X$ is locally nilpotent and that the restriction of $\phi$ to the subspace of Whittaker vectors of $X$ is injective. Then $\phi$ is injective.
\end{lemma}

\begin{proof}
Let $Z \subset X$ be the kernel of $\phi$. Assume that $Z$ is not trivial. Observe that $Z$ is invariant with respect to the action induced by the action of $U_{\varepsilon}^s(\m_+)$ on $X$, and that the augmentation ideal of $U_{\varepsilon}^s(\m_+)$ acts on $X$ by locally nilpotent transformations. Therefore by Engel theorem $Z$ must contain a nonzero $U_{\varepsilon}^s(\m_+)$--invariant vector which is a Whittaker vector $v\in X$. But since the restriction of $\phi$ to the subspace of Whittaker vectors of $X$ is injective $\phi(v)\neq 0$. Thus we arrive at a contradiction, and hence $\phi$ is injective.
\end{proof}

Now we prove that $\sigma_\varepsilon$ is injective.
Observe that by Proposition \ref{Qprop1} the augmentation ideal of $U_{\varepsilon}^s(\m_+)$ acts on $Q_{\varepsilon}$ by locally nilpotent transformations. Let $v\in W_\varepsilon^s(G)$ be a nonzero Whittaker vector of $Q_{\varepsilon}$. By the definition of map $\sigma_\varepsilon$ we have $\sigma_\varepsilon(v)(1)=\rho_\varepsilon(v)=v$, where $\rho_\varepsilon: Q_{\varepsilon}\rightarrow W_\varepsilon^s(G)\subset Q_{\varepsilon}$ is the linear map used in the definition of the map $\sigma_\varepsilon$ the restriction of which to $W_q^s(G)$ is the identity map. Therefore $\sigma_\varepsilon(v)\neq 0$. Now by Lemma \ref{inj} applied to $\sigma_{\varepsilon}:Q_{\varepsilon}\rightarrow {\rm hom}_{\mathbb{C}}(U_{\varepsilon}^s(\m_+), W_\varepsilon^s(G))$ the homomorphism $\sigma_\varepsilon$ is injective.

Now we prove that $\sigma_\varepsilon$ is surjective.
In order to do that we shall calculate the cohomology space of the right $U_{\varepsilon}^s(\m_+)$--module $Q_{\varepsilon}$ with respect to the adjoint action of $U_{\varepsilon}^s(\m_+)$,
\begin{equation}\label{c0}
{\rm Ext}^{\bullet}_{U_{\varepsilon}^s(\m_+)}(\mathbb{C}_{\varepsilon}, Q_{\varepsilon}).
\end{equation}

We shall show that
\begin{equation}\label{c1}
{\rm Ext}^{n}_{U_{\varepsilon}^s(\m_+)}(\mathbb{C}_{\varepsilon}, Q_{\varepsilon})=0,~n>0.
\end{equation}
Note that we already know that by definition
\begin{equation}\label{c2}
{\rm Ext}^{0}_{U_{\varepsilon}^s(\m_+)}(\mathbb{C}_{\varepsilon}, Q_{\varepsilon})={\rm Hom}_{U_{\varepsilon}^s(\m_+)}(\mathbb{C}_{\varepsilon}, Q_{\varepsilon})=W_\varepsilon^s(G).
\end{equation}

We shall calculate the $\rm Ext$ functors in formula (\ref{c1}) using a deformation argument which is based on upper semicontinuity of cohomology functor with respect to base ring localizations discovered by Grothendieck (see for instance \cite{semc}, Theorem 1.2 for the formulation of this principle suitable for our purposes). Let $X^\gr$ be a complex of finitely generated free modules over a ring $\bf k$, $X^\gr_p$ the corresponding complex over the residue field ${k}(p)$ of the localization of $\bf k$ at a prime ideal $p$. Then for each $i$ the function $p \mapsto {\rm dim}_{\mathbb{C}}H^i(X^\gr_p)$ is upper semicontinuous on ${\rm Spec}(\bf k)$. In particular, if $H^i(X^\gr_{p_0})=0$ for some $p_0$ then for generic $p$ we have $H^i(X^\gr_p)=0$.

As $\bf k$ we shall take $\mathcal{A}'$. Note that one can define a localization,
$\mathcal{A}'/(1-q^{\frac{1}{2d}})\mathcal{A}'=\mathbb{C}$ as well as similar localizations for other generic values of $\varepsilon$, $\mathcal{A}'/(\varepsilon^{\frac{1}{2d}}-q^{\frac{1}{2d}})\mathcal{A}'=\mathbb{C}$.

An appropriate complex $X^\gr$ is a little bit more complicated to define.
Let $\mathbb{C}_{\mathcal{A}'}$ be the trivial representation of $U_{\mathcal{A}'}^s(\m_+)$ given by the counit.
We shall construct a complex $X^\gr_{\mathcal{A}'}$ for calculating the functor ${\rm Ext}^\gr_{U_{\mathcal{A}'}(\m_+)}(\mathbb{C}_{\mathcal{A}'},Q_{\mathcal{A}'})$ the specialization of which for any generic $\varepsilon$ is a complex for calculating the functor ${\rm Ext}^{\gr}_{U_{\varepsilon}^s(\m_+)}(\mathbb{C}_{\varepsilon},Q_{\varepsilon})$, and the specialization of $X^\gr_{\mathcal{A}'}$ at $q=1$ is a complex for calculating $U(\m_+)$--cohomology with values $\mathbb{C}[M_+]\otimes W^s(G)$, where the action of $U(\m_+)$ on $\mathbb{C}[M_+]\otimes W^s(G)$ is induced by the natural action of $U(\m_+)$ on $\mathbb{C}[M_+]$ by left invariant differential operators. These cohomology is just the de Rham cohomology of $M_+$, and hence is trivial in nonzero degrees. Moreover, the complex $X^\gr_{\mathcal{A}'}$ will be filtered by finitely generated free modules. Therefore Grothendieck upper semicontinuity of cohomology together with the property of the specialization of our complex at $q=1$ imply vanishing property (\ref{c1}).

To construct the complex $X^\gr_{\mathcal{A}'}$ we first recall the definition of the standard bar resolution of an associative algebra $A$ over a ring $\bf k$ regarded as an $A-A$--bimodule (see \cite{carteil}, Ch. 9, \S 6),
\begin{equation}\label{bar}
\begin{array}{l}
{\rm Bar}^n(A)=\underbrace{A\otimes_{\bf k}\ldots \otimes_{\bf k} A}_{n+2 ~\mbox{\tiny  times}},~~n\geq 0,  \\
 \\
 d(a_0\otimes \ldots \otimes a_{n+1})= \\
\\
  \sum_{s=0}^{n}(-1)^s a_0\otimes\ldots\otimes
  a_sa_{s+1}\otimes\ldots\otimes a_{n+1}
\end{array}
\end{equation}
where
$a_0,\ldots ,a_{n+1}\in A$.

Now observe that if one introduces degrees of elements of $U_{\mathcal{A}'}^s(\n_+)$ by putting ${\rm deg}e_i=1$, $i=1,\ldots ,l$ the algebra $U_{\mathcal{A}'}^s(\n_+)$ becomes naturally $\mathbb{N}$--graded  by subspaces $U_{\mathcal{A}'}^s(\n_+)^k$ which are free over $\mathcal{A}'$ and have finite rank over $\mathcal{A}'$. Let $U_{\mathcal{A}'}^s(\m_+)^k$ be the induced grading of $U_{\mathcal{A}'}^s(\m_+)$ and denote by  $U_{\mathcal{A}'}^s(\m_+)_k$ the induced filtration of $U_{\mathcal{A}'}^s(\m_+)$ by subspaces of finite rank over $\mathcal{A}'$.

Now one can define a filtration of the $U_{\mathcal{A}'}^s(\m_+)$--module $Q_{\mathcal{A}'}$ by free $\mathcal{A}'$--modules of finite rank over $\mathcal{A}'$. In order to do that we recall that $Q_{\mathcal{A}'}$ is a submodule of ${\rm hom}_{\mathcal{A}'}(U_{\mathcal{A}'}^s(\m_+),W_q^s(G))$ as we observed in the proof of Proposition \ref{Qprop}. We also observe that from the definition of the space ${\rm hom}_{\mathcal{A}'}(U_{\mathcal{A}'}^s(\m_+),W_q^s(G))$ it follows that
\begin{equation}\label{homt}
{\rm hom}_{\mathcal{A}'}(U_{\mathcal{A}'}^s(\m_+),W_q^s(G))={\rm hom}_{\mathcal{A}'}(U_{\mathcal{A}'}^s(\m_+),\mathcal{A}')\otimes_{\mathcal{A}'} W_q^s(G),
\end{equation}
where
$$
{\rm hom}_{\mathcal{A}'}(U_{\mathcal{A}'}^s(\m_+),\mathcal{A}')=\oplus_{k\leq 0}{\rm hom}_{\mathcal{A}'}(U_{\mathcal{A}'}^s(\m_+)^{-k},\mathcal{A}'),
$$
Observe that ${\rm hom}_{\mathcal{A}'}(U_{\mathcal{A}'}^s(\m_+),\mathcal{A}')$ is naturally a $\mathbb{Z}_-$--graded module over the $\mathbb{N}$--graded algebra $U_{\mathcal{A}'}^s(\m_+)$. Denote by ${\rm hom}_{\mathcal{A}'}(U_{\mathcal{A}'}^s(\m_+),\mathcal{A}')_k=\oplus_{p\geq k}{\rm hom}_{\mathcal{A}'}(U_{\mathcal{A}'}^s(\m_+)^{-p},\mathcal{A}')$ the corresponding filtration of ${\rm hom}_{\mathcal{A}'}(U_{\mathcal{A}'}^s(\m_+),\mathcal{A}')$ by subspaces which are free over $\mathcal{A}'$ and have finite rank over $\mathcal{A}'$. By construction the action of $U_{\mathcal{A}'}^s(\m_+)$ on ${\rm hom}_{\mathcal{A}'}(U_{\mathcal{A}'}^s(\m_+),\mathcal{A}')$ preserves the filtration of ${\rm hom}_{\mathcal{A}'}(U_{\mathcal{A}'}^s(\m_+),\mathcal{A}')$. Combining the filtration on ${\rm hom}_{\mathcal{A}'}(U_{\mathcal{A}'}^s(\m_+),\mathcal{A}')$ with an arbitrary filtration $W_q^s(G)_k$, $k\in \mathbb{N}$ of $W_q^s(G)$ by free $\mathcal{A}'$--submodules of finite rank we obtain a filtration of
$$
{\rm hom}_{\mathcal{A}'}(U_{\mathcal{A}'}^s(\m_+),W_q^s(G))={\rm hom}_{\mathcal{A}'}(U_{\mathcal{A}'}^s(\m_+),\mathcal{A}')\otimes_{\mathcal{A}'} W_q^s(G),
$$
$$
{\rm hom}_{\mathcal{A}'}(U_{\mathcal{A}'}^s(\m_+),W_q^s(G))_k=\bigcup_{q-p\leq k}{\rm hom}_{\mathcal{A}'}(U_{\mathcal{A}'}^s(\m_+),\mathcal{A}')_p\otimes_{\mathcal{A}'} W_q^s(G)_q,~k\in \mathbb{N}.
$$
The induced filtration of the submodule $Q_{\mathcal{A}'}$ has components which are free $\mathcal{A}'$--modules of finite rank. By construction the action of $U_{\mathcal{A}'}^s(\m_+)$ on $Q_{\mathcal{A}'}$ preserves the components $(Q_{\mathcal{A}'})_k$ of that filtration.

The filtration $U_{\mathcal{A}'}^s(\m_+)_k$ induces a filtration ${\rm Bar}^n(U_{\mathcal{A}'}^s(\m_+))_k$ of the complex ${\rm Bar}^n(U_{\mathcal{A}'}^s(\m_+))$ by subcomplexes with finite rank graded components.

Consider the subcomplex
$$
{\rm hom}_{U_{\mathcal{A}'}^s(\m_+)}({\rm Bar}^\gr(U_{\mathcal{A}'}^s(\m_+)),Q_{\mathcal{A}'})=
$$
$$
=\bigcup_k{\rm Hom}_{U_{\mathcal{A}'}^s(\m_+)}({\rm Bar}^\gr(U_{\mathcal{A}'}^s(\m_+))_k,Q_{\mathcal{A}'})
$$
of the complex ${\rm Hom}_{U_{\mathcal{A}'}^s(\m_+)}({\rm Bar}^\gr(U_{\mathcal{A}'}^s(\m_+),Q_{\mathcal{A}'})$. Since ${\rm Bar}^\gr(U_{\mathcal{A}'}^s(\m_+))$ is homotopic to $U_{\mathcal{A}'}^s(\m_+)$ as a filtered $U_{\mathcal{A}'}^s(\m_+)-U_{\mathcal{A}'}^s(\m_+)$ bimodule the cohomology of
$${\rm hom}_{U_{\mathcal{A}'}^s(\m_+)}({\rm Bar}^\gr(U_{\mathcal{A}'}^s(\m_+)),Q_{\mathcal{A}'})$$ coincide with $Q_{\mathcal{A}'}$. We claim that the homological degree graded components of $${\rm hom}_{U_{\mathcal{A}'}^s(\m_+)}({\rm Bar}^\gr(U_{\mathcal{A}'}^s(\m_+)),Q_{\mathcal{A}'})$$ are injective $U_{\mathcal{A}'}^s(\m_+)$--modules, and hence ${\rm hom}_{U_{\mathcal{A}'}^s(\m_+)}({\rm Bar}^\gr(U_{\mathcal{A}'}^s(\m_+)),Q_{\mathcal{A}'})$ is an injective resolution of $Q_{\mathcal{A}'}$.

Indeed, by construction each of the components ${\rm hom}_{U_{\mathcal{A}'}^s(\m_+)}({\rm Bar}^n(U_{\mathcal{A}'}^s(\m_+)),Q_{\mathcal{A}'})$ is isomorphic to ${\rm hom}_{\mathcal{A}'}(U_{\mathcal{A}'}^s(\m_+),W)=\bigcup_k{\rm Hom}_{\mathcal{A}'}((U_{\mathcal{A}'}^s(\m_+))_k,W)$ for some free $\mathcal{A}'$--module $W$, and the right action of $U_{\mathcal{A}'}^s(\m_+)$ on ${\rm hom}_{\mathcal{A}'}(U_{\mathcal{A}'}^s(\m_+),W)$ is induced by multiplication on $U_{\mathcal{A}'}^s(\m_+)$ from the left. Clearly, ${\rm hom}_{\mathcal{A}'}(U_{\mathcal{A}'}^s(\m_+),W)$ is the subspace of ${\rm Hom}_{\mathcal{A}'}(U_{\mathcal{A}'}^s(\m_+),W)$ which consist of the linear maps vanishing on some power of the augmentation ideal $J={\rm Ker }\varepsilon$ of $U_{\mathcal{A}'}^s(\m_+)$, ${\rm hom}_{\mathcal{A}'}(U_{\mathcal{A}'}^s(\m_+),W)=\{f\in {\rm Hom}_{\mathcal{A}'}(U_{\mathcal{A}'}^s(\m_+),W):f(J^p)=0~{\rm for~some}~p>0\}$.

\begin{lemma}\label{inj1}
Let $J={\rm Ker }\varepsilon$ be the augmentation ideal of $U_{\mathcal{A}'}^s(\m_+)$, ${\rm hom}_{\mathcal{A}'}(U_{\mathcal{A}'}^s(\m_+),W)=\{f\in {\rm Hom}_{\mathcal{A}'}(U_{\mathcal{A}'}^s(\m_+),W):f(J^p)=0~{\rm for~some}~p>0\}$, where $W$ is a free $\mathcal{A}'$--module. Equip ${\rm hom}_{\mathcal{A}'}(U_{\mathcal{A}'}^s(\m_+),W)$ with the right action of $U_{\mathcal{A}'}^s(\m_+)$ induced by multiplication on $U_{\mathcal{A}'}^s(\m_+)$ from the left. Then the $U_{\mathcal{A}'}^s(\m_+)$--module ${\rm hom}_{\mathcal{A}'}(U_{\mathcal{A}'}^s(\m_+),W)$ is injective.
\end{lemma}

\begin{proof}
First observe that the algebra $U_{\mathcal{A}'}^s(\m_+)$ is Notherian and ideal $J$ satisfies the so-called weak Artin--Rees property, i.e. for every finitely generated left $U_{\mathcal{A}'}^s(\m_+)$--module $M$ and its submodule $N$ there exists an integer $n>0$ such that $J^nM\bigcap N\subset JN$.
Indeed, observe that the algebra $U_{\mathcal{A}'}^s(\m_+)$ can be equipped with a filtration similar to that introduced in Section \ref{wqreal} on the algebra $U_q^s(\g)$ in such a way that the associated graded algebra is finitely generated and semi--commutative (see (\ref{scomm})).
The fact that $U_{\mathcal{A}'}^s(\m_+)$ is Notherian follows from the existence of the filtration on it for which the associated graded algebra is  semi--commutative and from Theorem 4 in  Ch. 5, \S 3 in \cite{Jac} (compare also with Theorem 4.8 in \cite{Mc}). The ideal $J$ satisfies the weak Artin--Rees property because the subring $U_{\mathcal{A}'}^s(\m_+) +Jt+J^2t^2+\ldots \subset U_{\mathcal{A}'}^s(\m_+)[t]$, where $t$ is a central indeterminate, is Notherian (see \cite{Pas}, Ch. 11, \S 2, Lemma 2.1). The last fact follows from the existence of a filtration on $U_{\mathcal{A}'}^s(\m_+) +Jt+J^2t^2+\ldots$ induced by the filtration on $U_{\mathcal{A}'}^s(\m_+)$ for which the associated graded algebra is semi--commutative and again from Theorem 4 in  Ch. 5, \S 3 in \cite{Jac}.

Finally, the module ${\rm Hom}_{\mathcal{A}'}(U_{\mathcal{A}'}^s(\m_+),W)$ is obviously injective. By Lemma 3.2 in Ch. 3, \cite{Har} the module ${\rm hom}_{\mathcal{A}'}(U_{\mathcal{A}'}^s(\m_+),W)=\{f\in {\rm Hom}_{\mathcal{A}'}(U_{\mathcal{A}'}^s(\m_+),W):f(J^p)=0~{\rm for~some}~p>0\}$ is also injective since the ideal $J$ satisfies the weak Artin--Rees property.
\end{proof}

 Lemma \ref{inj1} implies that the complex ${\rm hom}_{U_{\mathcal{A}'}^s(\m_+)}({\rm Bar}^\gr(U_{\mathcal{A}'}^s(\m_+)),Q_{\mathcal{A}'})$ is  an injective resolution of $Q_{\mathcal{A}'}$ as a right $U_{\mathcal{A}'}^s(\m_+)$--module.

Now consider the complex $$X^\gr_{\mathcal{A}'}={\rm Hom}_{U_{\mathcal{A}'}^s(\m_+)}(\mathbb{C}_{\mathcal{A}'},{\rm hom}_{U_{\mathcal{A}'}^s(\m_+)}({\rm Bar}^\gr(U_{\mathcal{A}'}^s(\m_+)),Q_{\mathcal{A}'}))$$ for calculating the functor $${\rm Ext}^\gr_{U_{\mathcal{A}'}^s(\m_+)}(\mathbb{C}_{\mathcal{A}'},Q_{\mathcal{A}'})=H^\gr(X^\gr_{\mathcal{A}'}).$$

Observe that the specialization of the $U_{\mathcal{A}'}^s(\m_+)$--module $\mathbb{C}_{\mathcal{A}'}$ at $\varepsilon$ is isomorphic to $\mathbb{C}_{\varepsilon}$, and the specialization of the $U_{\mathcal{A}'}^s(\m_+)$--module $Q_{\mathcal{A}'}$ at $\varepsilon$ is isomorphic to $Q_{\varepsilon}$. Therefore the specialization of the complex $X^\gr_{\mathcal{A}'}$ at $\varepsilon$ is isomorphic to $$X^\gr_{\varepsilon}={\rm Hom}_{U_{\varepsilon}^s(\m_+)}(\mathbb{C}_{\varepsilon},{\rm hom}_{U_{\varepsilon}^s(\m_+)}({\rm Bar}^\gr(U_{\varepsilon}^s(\m_+)),
Q_{\varepsilon})),$$ where the complex $${\rm hom}_{U_{\varepsilon}^s(\m_+)}({\rm Bar}^\gr(U_{\varepsilon}^s(\m_+)),
Q_{\varepsilon})$$ is defined similarly to ${\rm hom}_{U_{\mathcal{A}'}^s(\m_+)}({\rm Bar}^\gr(U_{\mathcal{A}'}^s(\m_+)),Q_{\mathcal{A}'})$ using the $\varepsilon$--specialization of the filtration $(U_{\mathcal{A}'}^s(\m_+))_k$. Applying the same arguments as in case of the complex $X^\gr_{\mathcal{A}'}$ one can show that $X^\gr_{\varepsilon}$ is a complex for calculating the functor ${\rm Ext}^{\gr}_{U_{\varepsilon}^s(\m_+)}(\mathbb{C}_{\varepsilon},
Q_{\varepsilon})=
H^\gr(X^\gr_{\varepsilon})$.

The specialization of the algebra $U_{\mathcal{A}'}^s(\m_+)$ at $q=1$ is isomorphic to $U(\m_+)$, the specialization of the $U_{\mathcal{A}'}^s(\m_+)$--module $\mathbb{C}_{\mathcal{A}'}$ at $q=1$ is isomorphic to the trivial representation $\mathbb{C}_0$ of $U(\m_+)$, and the specialization of the $U_{\mathcal{A}'}^s(\m_+)$--module $Q_{\mathcal{A}'}$ at $q=1$ is isomorphic to ${\mathbb{C}}[\overline{q(\mu_{M_+}^{-1}(u))}]$. By Theorem \ref{var} $${\mathbb{C}}[\overline{q(\mu_{M_+}^{-1}(u))}]\cong {\mathbb{C}}[M_+]\otimes W^s(G).$$ From the proof of Proposition 11.2 in \cite{S10} we obtain that the induced action of $U(\m_+)$ on the corresponding variety $\overline{q(\mu_{M_+}^{-1}(u))}$ is obtained from the conjugation action of $M_+$ and now using proposition \ref{crosssect} one immediately deduces that the induced action of $U(\m_+)$ on  ${\mathbb{C}}[M_+]\otimes W^s(G)$ is generated by the action of $U(\m_+)$ on ${\mathbb{C}}[M_+]$ by left invariant differential operators. Therefore the specialization of the complex $X^\gr_{\mathcal{A}'}$ at $q=1$ is isomorphic to $$X^\gr_{1}={\rm Hom}_{U(\m_+)}(\mathbb{C}_0,{\rm hom}_{U(\m_+)}({\rm Bar}^\gr(U(\m_+)),
{\mathbb{C}}[M_+]\otimes W^s(G))),$$ where the complex $${\rm hom}_{U(\m_+)}({\rm Bar}^\gr(U(\m_+)),
{\mathbb{C}}[M_+]\otimes W^s(G))$$ is defined similarly to $${\rm hom}_{U_{\mathcal{A}'}^s(\m_+)}({\rm Bar}^\gr(U_{\mathcal{A}'}^s(\m_+)),Q_{\mathcal{A}'})$$ using the $q=1$--specialization of the filtration $(U_{\mathcal{A}'}^s(\m_+))_k$. Applying the same arguments as in case of the complex $X^\gr_{\mathcal{A}'}$ one can show that $X^\gr_{1}$ is a complex for calculating the functor $${\rm Ext}^{\gr}_{U(\m_+)}(\mathbb{C}_0,
{\mathbb{C}}[M_+]\otimes W^s(G))=
H^\gr(X^\gr_{1}).$$

We also obviously have ${\rm Ext}^{\gr}_{U(\m_+)}(\mathbb{C}_0,
{\mathbb{C}}[M_+]\otimes W^s(G))={\rm Ext}^{\gr}_{U(\m_+)}(\mathbb{C}_0,
{\mathbb{C}}[M_+]\otimes W^s(G))=H_{dR}^\gr(M_+)\otimes W^s(G)$, where $H_{dR}^\gr(M_+)$ is the de Rham cohomology of the unipotent group $M_+$. Since $H_{dR}^n(M_+)=0$ for $n>0$ we deduce that $H^n(X^\gr_{1})=0$ for $n>0$.

Finally observe that the complex $X^\gr_{\mathcal{A}'}$ and its specializations introduced above can be equipped with compatible filtrations by finitely generated free subcomplexes. These filtrations are induced by the filtrations $(Q_{\mathcal{A}'})_k$ and $(U_{\mathcal{A}'}^s(\m_+))_k$ and by their specializations at $q=\varepsilon$ and $q=1$. The Grothendieck cohomology semicontinuity property holds for these subcomplexes, and hence for the complex $X^\gr_{\mathcal{A}'}$ as well. Therefore from the vanishing property $H^n(X^\gr_{1})=0$ for $n>0$ we deduce that for generic $\varepsilon$ ${\rm Ext}^{n}_{U_{\varepsilon}^s(\m_+)}(\mathbb{C}_{\varepsilon},
Q_{\varepsilon})=
H^n(X^\gr_{\varepsilon})=0$ for $n>0$.

Now we prove that $\sigma_\varepsilon$ is surjective. We start with the following lemma.

\begin{lemma}\label{surj}
Let $\phi:X\rightarrow Y$ be an injective homomorphism of $U_{\varepsilon}^s(\m_+)$--modules. Assume that $\phi$ induces an isomorphism of the spaces of Whittaker vectors of $X$  and of $Y$, and that ${\rm Ext}^{1}_{U_{\varepsilon}^s(\m_+)}(\mathbb{C}_{\varepsilon},
X)=0$, where $\mathbb{C}_{\varepsilon}$ is the trivial representation of $U_{\varepsilon}^s(\m_+)$. Suppose also that the action of the augmentation ideal $J$ of $U_{\varepsilon}^s(\m_+)$ on the cokernel of $\phi$ is locally nilpotent. Then $\phi$ is surjective.
\end{lemma}

\begin{proof}
Consider the exact sequence
$$
0\rightarrow X \rightarrow Y \rightarrow W' \rightarrow 0,
$$
where $W'$ is the cokernel of $\phi$, and the corresponding long exact sequence of cohomology,
$$
0\rightarrow {\rm Ext}^{0}_{U_{\varepsilon}^s(\m_+)}(\mathbb{C}_{\varepsilon},
X)\rightarrow {\rm Ext}^{0}_{U_{\varepsilon}^s(\m_+)}(\mathbb{C}_{\varepsilon},
Y)\rightarrow {\rm Ext}^{0}_{U_{\varepsilon}^s(\m_+)}(\mathbb{C}_{\varepsilon},
W')\rightarrow
$$
$$
\rightarrow {\rm Ext}^{1}_{U_{\varepsilon}^s(\m_+)}(\mathbb{C}_{\varepsilon},
X)\rightarrow \ldots .
$$

Since $\phi$ induces an isomorphism of the spaces of Whittaker vectors of $X$  and of $Y$, and \\ ${\rm Ext}^{1}_{U_{\varepsilon}^s(\m_+)}(\mathbb{C}_{\varepsilon},
X)=0$, the initial part of the long exact cohomology sequence takes the form
$$
0\rightarrow {\rm Wh}(X)\rightarrow {\rm Wh}(Y)\rightarrow {\rm Wh}(W')\rightarrow 0,
$$
where the second map in the last sequence is an isomorphism. Using the last exact sequence we deduce that ${\rm Wh}(W')=0$. But the augmentation ideal $J$ acts on $W'$ by locally nilpotent transformations. Therefore, by Engel theorem, if $W'$ is not trivial there should exists a nonzero $U_{\varepsilon}^s(\m_+)$--invariant vector in it. Thus we arrive at a contradiction, and $W'=0$. Therefore $\phi$ is surjective.

\end{proof}

Now recall that by (\ref{c1}) and (\ref{c2}) we already know that
$$
{\rm Wh}(Q_{\varepsilon})=W_\varepsilon^s(G),~~{\rm Ext}^{1}_{U_{\varepsilon}^s(\m_+)}(\mathbb{C}_{\varepsilon}, Q_{\varepsilon})=0,
$$
and by the definition of the module ${\rm hom}_{\mathbb{C}}(U_{\varepsilon}^s(\m_+), W_\varepsilon^s(G))$
$$
{\rm Wh}({\rm hom}_{\mathbb{C}}(U_{\varepsilon}^s(\m_+), W_\varepsilon^s(G)))={\rm Hom}_{U_{\varepsilon}^s(\m_+)}(\mathbb{C}_{\varepsilon}, {\rm hom}_{\mathbb{C}}(U_{\varepsilon}^s(\m_+), W_\varepsilon^s(G)))=W_\varepsilon^s(G).
$$

Observe also that by construction the map $\sigma_{\varepsilon}:Q_{\varepsilon}\rightarrow {\rm hom}_{\mathbb{C}}(U_{\varepsilon}^s(\m_+), W_\varepsilon^s(G))$ induces an isomorphism of the spaces of Whittaker vectors. Since the action of the augmentation ideal $J$ on ${\rm hom}_{\mathbb{C}}(U_{\varepsilon}^s(\m_+), W_\varepsilon^s(G))$ is locally nilpotent its action on the cokernel of $\sigma_{\varepsilon}$ is locally nilpotent as well.
Therefore $\sigma_\varepsilon$ is surjective by Lemma \ref{surj}.

Thus we have proved that $\sigma_{\varepsilon}:Q_{\varepsilon}\rightarrow {\rm hom}_{\mathbb{C}}(U_{\varepsilon}^s(\m_+), W_\varepsilon^s(G))$ is an isomorphism of right $U_{\varepsilon}^s(\m_+)$--modules. Note that by the definitions of the spaces ${\rm hom}_{\mathbb{C}}(U_{\varepsilon}^s(\m_+), W_\varepsilon^s(G))$ and ${\rm hom}_{\mathbb{C}}(U_{\varepsilon}^s(\m_+),\mathbb{C})$ we also have an obvious right $U_{\varepsilon}^s(\m_+)$--module isomorphism ${\rm hom}_{\mathbb{C}}(U_{\varepsilon}^s(\m_+, W_\varepsilon^s(G))={\rm hom}_{\mathbb{C}}(U_{\varepsilon}^s(\m_+),\mathbb{C})\otimes W_\varepsilon^s(G)$.

Now consider the $U_{\varepsilon}^s(\m_+)$--submodule $\sigma_{\varepsilon}^{-1}({\rm hom}_{\mathbb{C}}(U_{\varepsilon}^s(\m_+),\mathbb{C}))$ of $Q_{\varepsilon}$, where ${\rm hom}_{\mathbb{C}}(U_{\varepsilon}^s(\m_+),\mathbb{C})\subset {\rm hom}_{\mathbb{C}}(U_{\varepsilon}^s(\m_+), W_\varepsilon^s(G))$. Obviously $\sigma_{\varepsilon}^{-1}({\rm hom}_{\mathbb{C}}(U_{\varepsilon}^s(\m_+),\mathbb{C}))\simeq {\rm hom}_{\mathbb{C}}(U_{\varepsilon}^s(\m_+),\mathbb{C})$ as a right $U_{\varepsilon}^s(\m_+)$--module.

Let $\phi_\varepsilon:\sigma_{\varepsilon}^{-1}({\rm hom}_{\mathbb{C}}(U_{\varepsilon}^s(\m_+),\mathbb{C}))\otimes W_\varepsilon^s(G)\rightarrow Q_{\varepsilon}$ be the map induced  by the action of $W_\varepsilon^s(G)$ on $Q_{\varepsilon}$. Since this action commutes with the adjoint action of $U_{\varepsilon}^s(\m_+)$ on $Q_{\varepsilon}$ we infer that $\phi_\varepsilon$ is a homomorphism of $U_{\varepsilon}^s(\m_+)$--$W_\varepsilon^s(G)$--bimodules.

We claim that $\phi_\varepsilon$ is injective. This follows straightforwardly from Lemma \ref{inj} because all Whittaker vectors of $\sigma_{\varepsilon}^{-1}({\rm hom}_{\mathbb{C}}(U_{\varepsilon}^s(\m_+),\mathbb{C}))\otimes W_\varepsilon^s(G)$ belong to the subspace $$1\otimes W_\varepsilon^s(G)\subset \sigma_{\varepsilon}^{-1}({\rm hom}_{\mathbb{C}}(U_{\varepsilon}^s(\m_+),\mathbb{C}))\otimes W_\varepsilon^s(G),$$ and the restriction of $\phi_\varepsilon$ to this subspace is injective.

Now we show that $\phi_\varepsilon$ is surjective. By the specializing the result of Lemma \ref{inj1} at $q=\varepsilon$ one can immediately deduce that the right $U_{\varepsilon}^s(\m_+)$--module $\sigma_{\varepsilon}^{-1}({\rm hom}_{\mathbb{C}}(U_{\varepsilon}^s(\m_+),\mathbb{C}))\otimes W_\varepsilon^s(G)\simeq {\rm hom}_{\mathbb{C}}(U_{\varepsilon}^s(\m_+), W_\varepsilon^s(G))$ is injective. In particular, ${\rm Ext}^{1}_{U_{\varepsilon}^s(\m_+)}(\mathbb{C}_{\varepsilon},\sigma_{\varepsilon}^{-1}({\rm hom}_{\mathbb{C}}(U_{\varepsilon}^s(\m_+),\mathbb{C}))\otimes W_\varepsilon^s(G)
)=0$. One checks straightforwardly, similarly the case of the map $\sigma_\varepsilon$, that the other conditions of Lemma \ref{surj} for the map $\phi_\varepsilon$ are satisfied as well. Therefore $\phi_\varepsilon$ is surjective.

This completes the proof of the proposition.
\end{proof}

Now we formulate our main statement.

\begin{theorem}\label{sqeq}
Assume that the roots $\gamma_1,\ldots, \gamma_n$ (or $\gamma_{n+1},\ldots, \gamma_{l'}$) are simple or one of the sets $\gamma_1,\ldots, \gamma_n$ or $\gamma_{n+1},\ldots, \gamma_{l'}$ is empty. Suppose also that the numbers $t_{i}$ defined in (\ref{defu}) are not equal to zero for all $i$. Then for generic $\varepsilon\in \mathbb{C}$ the functor $E\mapsto Q_\varepsilon\otimes_{W_\varepsilon^s(G)}E$, is an equivalence of the category of finitely generated left $W_\varepsilon^s(G)$--modules and the category ${{\mathbb{C}}_{\varepsilon}[G_*]-{\rm mod}_{U_{\varepsilon}^s(\m_+)}^{\chi^{s}_\varepsilon}}_{loc}$. The inverse equivalence is given by the functor $V\mapsto {\rm Wh}(V)$. In particular, the latter functor is exact.

Every module $V\in {{\mathbb{C}}_{\varepsilon}[G_*]-{\rm mod}_{U_{\varepsilon}^s(\m_+)}^{\chi^{s}_\varepsilon}}_{loc}$ is isomorphic to ${\rm hom}_{\mathbb{C}}(U_{\varepsilon}^s(\m_+),\mathbb{C}))\otimes {\rm Wh}(V)$ as a right $U_{\varepsilon}^s(\m_+)$--module. In particular, $V$ is $U_{\varepsilon}^s(\m_+)$--injective, and ${\rm Ext}^\bullet_{U_{\varepsilon}^s(\m_+)}(\mathbb{C}_{\varepsilon},V)={\rm Wh}(V)$.
\end{theorem}

\begin{proof}

Let $E$ be a finitely generated $W_\varepsilon^s(G)$--module. First we observe that by the definition of the algebra $W_\varepsilon^s$ we have
${\rm Wh}(Q_{\mathcal{A}'}\otimes_{W_\varepsilon^s(G)}E)=E$. Therefore to prove the theorem it suffices to check that for any $V\in {{\mathbb{C}}_{\varepsilon}[G_*]-{\rm mod}_{U_{\varepsilon}^s(\m_+)}^{\chi^{s}_\varepsilon}}_{loc}$ the canonical map $f:Q_{\varepsilon}\otimes_{W_\varepsilon^s(G)}{\rm Wh}(V)\rightarrow V$ is an isomorphism.

By the previous Proposition $Q_{\varepsilon}={\rm hom}_{\mathbb{C}}(U_{\varepsilon}^s(\m_+),\mathbb{C})\otimes W_\varepsilon^s(G)$ as a $U_{\varepsilon}^s(\m_+)$--$W_\varepsilon^s(G)$--bimodule. Therefore
\begin{equation}\label{eql}
Q_{\varepsilon}\otimes_{W_\varepsilon^s(G)}{\rm Wh}(V)={\rm hom}_{\mathbb{C}}(U_{\varepsilon}^s(\m_+),\mathbb{C})\otimes {\rm Wh}(V)
\end{equation}
as a right $U_{\varepsilon}^s(\m_+)$--module.

Now the fact that $f$ is an isomorphism can be established by repeating verbatim the arguments used in the proof of a similar statement for the map $\phi_\varepsilon$ in the previous Proposition. In particular $f$ is injective by Lemma \ref{inj}, $Q_{\varepsilon}\otimes_{W_\varepsilon^s(G)}{\rm Wh}(V)={\rm hom}_{\mathbb{C}}(U_{\varepsilon}^s(\m_+),\mathbb{C})\otimes {\rm Wh}(V)$ is an injective right $U_{\varepsilon}^s(\m_+)$--module by Lemma \ref{inj1}, and $f$ is surjective by Lemma \ref{surj}.

This completes the proof of the theorem.
\end{proof}


\section{Localization of quantum biequivariant $\mathcal{D}$--modules}\label{Dmod}

\setcounter{equation}{0}
\setcounter{theorem}{0}

In this section we present a biequivariant version of the localization theorem for quantum $\mathcal{D}$--modules proved in \cite{Kr,Tan}. A similar result for Beilinson--Bernstein localization of $\mathcal{D}$ modules on the flag variety  was already mentioned in the original paper \cite{BBer} (see also \cite{Kash} for more details).

Let $\varepsilon\in \mathbb{C}$ be transcendental and generic. Denote by $\mathbb{C}_\varepsilon[G]$ the Hopf algebra generated by matrix coefficients of finite--dimensional representations of $U_{\varepsilon}^s(\g)$. There is a natural paring $(\cdot,\cdot):U_{\varepsilon}^s(\g)\otimes \mathbb{C}_\varepsilon[G]\rightarrow \mathbb{C}$.
The algebra $\mathbb{C}_\varepsilon[G]$ is equipped with a $U_{\varepsilon}^s(\g)$--bimodule structure via the left and the right regular action,
\begin{equation}\label{}
u(a)=a_1(u,a_2),~(a)u=(u,a_1)a_2,~u\in U_{\varepsilon}^s(\g),~a\in \mathbb{C}_\varepsilon[G],~\Delta a=a_1\otimes a_2.
\end{equation}

Let $\mathcal{D}_\varepsilon$ be the Heisenberg double of $U_{\varepsilon}^s(\g)$ defined in \cite{dual}. As a vector space $\mathcal{D}_\varepsilon=\mathbb{C}_\varepsilon[G]\otimes U_{\varepsilon}^s(\g)$, and the multiplication on $\mathcal{D}_\varepsilon$ is given by
\begin{equation}\label{}
a\otimes u \cdot b\otimes v =au_1(b)\otimes u_2v,~a,b\in \mathbb{C}_\varepsilon[G],~u,v\in U_{\varepsilon}^s(\g),~\Delta_su=u_1\otimes u_2.
\end{equation}

The Heisenberg double is an analogue of the algebra of differential operators on the group $G$ in case of Hopf algebras.
$\mathcal{D}_\varepsilon$ also has the structure of a $U_{\varepsilon}^s(\g)$--bimodule,
\begin{eqnarray}\label{}
u_L(a\otimes v)=u_{(1)}(a)\otimes u_{(2)}vS_s(u_{(3)}),~u_R(a\otimes v)= (a)u\otimes v, \\ u\in U_{\varepsilon}^s(\g),~a\in \mathbb{C}_\varepsilon[G], (id\otimes \Delta_s)\Delta_s(u)=u_{(1)}\otimes u_{(2)}\otimes u_{(3)}. \nonumber
\end{eqnarray}

Both the left and the right $U_{\varepsilon}^s(\g)$ actions on $\mathcal{D}_\varepsilon$ are derivations with respect to the multiplicative structure in the sense that
\begin{equation}\label{derLR}
u_L(a\otimes u \cdot b\otimes v)={u_1}_L(a\otimes u)\cdot {u_2}_L(b\otimes v),~u_R(a\otimes u \cdot b\otimes v)={u_1}_R(a\otimes u)\cdot {u_2}_R(b\otimes v).
\end{equation}
These actions are analogues of the actions generated by left and right translations on $G$ on the algebra of differential operators.

Let $\lambda$ be a character of $U_{\varepsilon}^s(\h)$. $\lambda$ naturally extends to a one--dimensional $U_{\varepsilon}^s(\b_+)$--module that we denote by $\mathbb{C}_\lambda$.

Note that there is an algebra embedding $U_{\varepsilon}^s(\g)\subset \mathcal{D}_\varepsilon$, $x\mapsto 1\otimes x$.
The image of this embedding is an analogue of the algebra of right invariant vector fields on $G$. As in case of Lie groups right invariant vector fields generate left translations in the sense that
$$
1\otimes y_1 \cdot a\otimes x\cdot 1\otimes S_sy_2=y_L(a\otimes x),~y\in U_{\varepsilon}^s(\g)\subset \mathcal{D}_\varepsilon,~a\otimes x\in \mathcal{D}_\varepsilon,~\Delta_s(y)=y_1\otimes y_2.
$$

Let $\mathbb{C}_\varepsilon[B_+]'$ be the quotient Hopf algebra of $\mathbb{C}_\varepsilon[G]$ by the Hopf algebra ideal generated by elements vanishing on $U_{\varepsilon}^s(\b_+)$. Note that if $V$ is a right $\mathbb{C}_\varepsilon[B_+]'$--comodule then $V$ is also naturally a left $U_{\varepsilon}^s(\b_+)$--module.

A $(U_{\varepsilon}^s(\b_+), \lambda)$--equivariant $\mathcal{D}_\varepsilon$--module is a triple $(M,\alpha,\beta)$, where $M$ is a complex vector space equipped with a left $\mathcal{D}_\varepsilon$--action $\alpha:\mathcal{D}_\varepsilon \times M\rightarrow M$, a right $\mathbb{C}_\varepsilon[B_+]'$--coaction which gives rise to a left $U_{\varepsilon}^s(\b_+)$--action $\beta: U_{\varepsilon}^s(\b_+)\times M\rightarrow M$ such that
\begin{enumerate}

\item
The $U_{\varepsilon}^s(\b_+)$--actions on $M\otimes \mathbb{C}_\lambda$ given by $\beta\otimes \lambda$ and by $\alpha|_{U_{\varepsilon}^s(\b_+)}\otimes {\rm Id}$ coincide;

\item
$
\beta(u)(\alpha(a\otimes v)m)=\alpha({u_1}_L(a\otimes v))\beta(u_2)m,{\rm for ~all}~u\in U_{\varepsilon}^s(\b_+),~a\otimes v\in \mathcal{D}_\varepsilon,~m\in M,
$
\noindent
$
\Delta_su=u_1\otimes u_2.
$
\end{enumerate}

$(U_{\varepsilon}^s(\b_+), \lambda)$--equivariant $\mathcal{D}_\varepsilon$--modules form a category $\mathcal{D}^\lambda_{U_{\varepsilon}^s(\b_+)}$ morphisms in which are linear maps of vector spaces respecting all the above introduced structures on $(U_{\varepsilon}^s(\b_+), \lambda)$--equivariant $\mathcal{D}_\varepsilon$--modules.

Let $\mathcal{D}^\lambda_\varepsilon$ be the maximal quotient of $\mathcal{D}_\varepsilon$ which is an object of $\mathcal{D}^\lambda_{U_{\varepsilon}^s(\b_+)}$. In fact one has
$$
\mathcal{D}^\lambda_\varepsilon\simeq \mathcal{D}_\varepsilon/\mathcal{D}_\varepsilon I,
$$
where $I$ is the left ideal in $\mathcal{D}_\varepsilon$ generated by the elements $1\otimes e_i$, $1\otimes t_i-\lambda(t_i)$, $i=1,\ldots ,l$. We denote by $\bf{1}$ the image of $1\otimes 1\in \mathcal{D}_\varepsilon$ in $\mathcal{D}^\lambda_\varepsilon$.

Now define the global section functor $\Gamma:\mathcal{D}^\lambda_{U_{\varepsilon}^s(\b_+)}\rightarrow {\rm Vect}_{\mathbb{C}}$, where ${\rm Vect}_{\mathbb{C}}$ is the category of vector spaces,
$$
\Gamma(M)={\rm Hom}_{\mathcal{D}^\lambda_{U_{\varepsilon}^s(\b_+)}}(\mathcal{D}^\lambda_\varepsilon,M)={\rm Hom}_{U_{\varepsilon}^s(\b_+)}(\mathbb{C}_\varepsilon,M),
$$
where in the last formula $U_{\varepsilon}^s(\b_+)$ acts on $M$ according to $\beta$--action, and $\mathbb{C}_\varepsilon$ is the trivial representation of $U_{\varepsilon}^s(\b_+)$ given by the counit.

One can also write
\begin{equation}\label{GM}
\Gamma(M)={\rm Hom}_{U_{\varepsilon}^s(\b_+)}(\mathbb{C}_\lambda,M),
\end{equation}
where $U_{\varepsilon}^s(\b_+)$ acts on $M$ according to the $\alpha$--action composed with the embedding $U_{\varepsilon}^s(\b_+)\rightarrow \mathcal{D}_\varepsilon$, $x\mapsto 1\otimes x$.

Naturally $\Gamma(\mathcal{D}^\lambda_{\varepsilon})={\rm End}_{\mathcal{D}^\lambda_{U_{\varepsilon}^s(\b_+)}}(\mathcal{D}^\lambda_\varepsilon)$ is an algebra with multiplication induced from $\mathcal{D}_\varepsilon$. The algebra $\Gamma(\mathcal{D}^\lambda_{\varepsilon})$ naturally acts from the left on spaces $\Gamma(M)$ for $M\in \mathcal{D}^\lambda_{U_{\varepsilon}^s(\b_+)}$.

Recall that there is a locally finite right adjoint action of ${\rm Ad}:U_{\varepsilon}^s(\g)\times U_{\varepsilon}^s(\g)^{fin}\rightarrow U_{\varepsilon}^s(\g)^{fin}$ given by
$$
{\rm Ad}x(w)=S_s^{-1}(x_2)wx_1,
$$
where $\Delta_s(x)=x_1\otimes x_2$, $x\in U_{\varepsilon}^s(\g)$, $w\in U_{\varepsilon}^s(\g)^{fin}$.
Let $\Delta_{\rm Ad}:U_{\varepsilon}^s(\g)^{fin}\rightarrow \mathbb{C}_\varepsilon[G]\otimes U_{\varepsilon}^s(\g)^{fin}$ be the dual $\mathbb{C}_\varepsilon[G]$--coaction on $U_{\varepsilon}^s(\g)^{fin}$. One can consider the tensor product $\mathbb{C}_\varepsilon[G]\otimes U_{\varepsilon}^s(\g)^{fin}$ as a linear subspace of $\mathcal{D}_\varepsilon$. Using this fact $\Delta_{\rm Ad}$ can be regarded as a linear map to $\mathcal{D}_\varepsilon$, $\Delta_{\rm Ad}:U_{\varepsilon}^s(\g)^{fin}\rightarrow \mathcal{D}_\varepsilon$. In fact $\Delta_{\rm Ad}$ is an embedding, the left inverse map $\Delta_{{\rm Ad}S_s}$ is given by
\begin{equation}\label{invD}
\Delta_{{\rm Ad}S_s}(a\otimes x)=a\otimes 1\cdot \Delta_{{\rm Ad}S_s}(x),~a\otimes x\in {\rm Im}\Delta_{\rm Ad}\subset \mathbb{C}_\varepsilon[G]\otimes U_{\varepsilon}^s(\g)^{fin},
\end{equation}
where $\Delta_{{\rm Ad}S_s}:U_{\varepsilon}^s(\g)^{fin}\rightarrow \mathbb{C}_\varepsilon[G]\otimes U_{\varepsilon}^s(\g)^{fin}$ is the map dual to the action of $U_{\varepsilon}^s(\g)$ on $U_{\varepsilon}^s(\g)^{fin}$ given by ${\rm Ad}S_s$, and the image of $\Delta_{{\rm Ad}S_s}$ in (\ref{invD}) belongs to the subspace $1\otimes U_{\varepsilon}^s(\g)^{fin}$ which is naturally identified with $U_{\varepsilon}^s(\g)^{fin}$.

Direct calculation also shows that $\Delta_{\rm Ad}:U_{\varepsilon}^s(\g)^{fin}\rightarrow \mathcal{D}_\varepsilon$ is an algebra antihomomorphism,
$$
\Delta_{\rm Ad}(x)\cdot \Delta_{\rm Ad}(y)=\Delta_{\rm Ad}(yx).
$$

Note that $\Delta_{\rm Ad}$ can be extended to a homomorphism from $U_{\varepsilon}^s(\g)$ to a certain completion $\mathbb{C}_\varepsilon[G]\widehat{\otimes} U_{\varepsilon}^s(\g)$ of $\mathbb{C}_\varepsilon[G]\otimes U_{\varepsilon}^s(\g)$ by infinite series terms of which are elements of $\mathbb{C}_\varepsilon[G]\otimes U_{\varepsilon}^s(\g)$. We denote this extension by the same symbol, $\Delta_{\rm Ad}:U_{\varepsilon}^s(\g)\rightarrow \mathbb{C}_\varepsilon[G]\widehat{\otimes} U_{\varepsilon}^s(\g)$. One can equip the completion $\mathbb{C}_\varepsilon[G]\widehat{\otimes} U_{\varepsilon}^s(\g)$ with a multiplication induced from $\mathcal{D}_\varepsilon$. We denote the obtained algebra by $\widehat{\mathcal{D}}_\varepsilon$.

One checks that the map $\Delta_{{\rm Ad}S_s}$ naturally extends to a left inverse of $\Delta_{\rm Ad}:U_{\varepsilon}^s(\g)\rightarrow \widehat{\mathcal{D}}_\varepsilon$, and hence $\Delta_{\rm Ad}:U_{\varepsilon}^s(\g)\rightarrow \widehat{\mathcal{D}}_\varepsilon$ is an embedding.
The image of this map can be regarded as an analogue of the algebra of left invariant vector fields on $G$. In particular, these analogues generate the right action of $U_{\varepsilon}^s(\g)$ on ${\mathcal{D}}_\varepsilon$,
\begin{equation}\label{Dequiv}
\Delta_{\rm Ad}(y_1)\cdot a\otimes x\cdot \Delta_{\rm Ad}(S_s^{-1}y_2)=y_R(a\otimes x),~y\in U_{\varepsilon}^s(\g),~\Delta_s(y)=y_1\otimes y_2,~a\otimes x\in {\mathcal{D}}_\varepsilon.
\end{equation}

The map $\Delta_{\rm Ad}$ is also equivariant with respect to the right action of $U_{\varepsilon}^s(\g)$ on ${\mathcal{D}}_\varepsilon$ in the sense that
\begin{equation}\label{Requiv}
u_R(\Delta_{\rm Ad}(v))=\Delta_{\rm Ad}({\rm Ad}u(v)),~u\in U_{\varepsilon}^s(\g),~v\in U_{\varepsilon}^s(\g)^{fin}.
\end{equation}

Denote by $J_\lambda$ the annihilator of the Verma module $M_\varepsilon(\lambda)=U_{\varepsilon}^s(\g)\otimes_{U_{\varepsilon}^s(\b_+)}\mathbb{C}_\lambda$ in $U_{\varepsilon}^s(\g)^{fin}$. $J_\lambda$ is generated by the ideal of the center $Z(U_{\varepsilon}^s(\g)^{fin})=Z(U_{\varepsilon}^s(\g))$ corresponding to a character $\chi_{\lambda+\rho}:Z(U_{\varepsilon}^s(\g))\rightarrow \mathbb{C}$, where $\rho=\frac{1}{2}\sum_{\alpha \in \Delta_+}\alpha\in P_+$. Let $U_{\varepsilon}^s(\g)^\lambda$ be the quotient of $U_{\varepsilon}^s(\g)^{fin}$ by $J_\lambda$, $U_{\varepsilon}^s(\g)^\lambda=U_{\varepsilon}^s(\g)^{fin}/J_\lambda$.

We also denote by $I_\lambda$ the annihilator of $M_\varepsilon(\lambda)$ in $U_{\varepsilon}^s(\g)$ and by $U_{\varepsilon}^s(\g)_\lambda$ the quotient $U_{\varepsilon}^s(\g)_\lambda=U_{\varepsilon}^s(\g)/I_\lambda$.

A character $\lambda:U_{\varepsilon}^s(\h)\rightarrow \mathbb{C}$ is called regular dominant if for each $\phi\in P_+$ and all weights $\psi$ of $V_\varepsilon(\phi)$, $\phi \neq \psi$, one has $\chi_{\lambda+\phi}\neq \chi_{\lambda+\psi}$.

\begin{proposition}\label{BB}{\bf (\cite{Kr}, Proposition 4.8, Theorem 4.12)}
The map
\begin{equation}\label{}
U_{\varepsilon}^s(\g)^\lambda \rightarrow \Gamma(\mathcal{D}^\lambda_{\varepsilon})^{opp}={\rm Hom}_{U_{\varepsilon}^s(\b_+)}(\mathbb{C}_\varepsilon,\mathcal{D}^\lambda_{\varepsilon}), x\mapsto \Delta_{\rm Ad}(x){\bf 1}
\end{equation}
is an algebra isomorphism.

If $\lambda$ is regular dominant the global section functor $\Gamma:\mathcal{D}^\lambda_{U_{\varepsilon}^s(\b_+)}\rightarrow {\rm mod}-U_{\varepsilon}^s(\g)^\lambda$ is an equivalence of the category $\mathcal{D}^\lambda_{U_{\varepsilon}^s(\b_+)}$ and of the category ${\rm mod}-U_{\varepsilon}^s(\g)^\lambda$ of right $U_{\varepsilon}^s(\g)^\lambda$--modules. The inverse functor is given by
\begin{equation}\label{}
V\mapsto V\otimes_{U_{\varepsilon}^s(\g)^\lambda}\mathcal{D}^\lambda_{\varepsilon},~V\in {\rm mod}-U_{\varepsilon}^s(\g)^\lambda.
\end{equation}
\end{proposition}

Now we present an equivariant version of the previous proposition. Let $U\subset U_{\varepsilon}^s(\g)$ be a subalgebra equipped with a character $\chi:U\rightarrow \mathbb{C}$. Denote by $\mathbb{C}_\chi$ the corresponding one--dimensional representation of $U$. Assume that $U$ is also a coideal, i.e. $\Delta_s(U)\subset U\otimes U_{\varepsilon}^s(\g)$

A biequivariant $\mathcal{D}_\varepsilon$--module is a $(U_{\varepsilon}^s(\b_+), \lambda)$--equivariant $\mathcal{D}_\varepsilon$--module $M$ which is also equipped with the structure of a left $U$-module $\gamma: U\times M\rightarrow M$ such that
\begin{enumerate}

\item
For any $u\in U$ the action of the operator $\chi(u_1)\alpha(\Delta_{\rm Ad}(S_su_2))$ on $M$ is well defined, and
the $U$--actions on $\mathbb{C}_\chi \otimes M$ given by ${\rm Id} \otimes \gamma $ and by $\chi\otimes \alpha\Delta_{\rm Ad}\circ S_s$ coincide;

\item
$
\gamma(u)(\alpha(a\otimes v)m)=\alpha({S_s(u_2)}_R(a\otimes v))\gamma(u_1)m,{\rm for ~all}~u\in U,~a\otimes v\in \mathcal{D}_\varepsilon,~m\in M,
$
\noindent
$
\Delta_su=u_1\otimes u_2.
$
\end{enumerate}

Biequivariant $\mathcal{D}_\varepsilon$--modules form a category $_{~U}^{~~~\chi}\mathcal{D}^\lambda_{U_{\varepsilon}^s(\b_+)}$ morphisms in which are linear maps of vector spaces respecting all the above introduced structures on biequivariant $\mathcal{D}_\varepsilon$--modules.

A $(U,\chi)$--equivariant $U_{\varepsilon}^s(\g)^\lambda$--module is a right $U_{\varepsilon}^s(\g)^\lambda$--module $V$ equipped with the structure of a left $U$-module $\gamma: U\times V\rightarrow V$
such that
\begin{enumerate}

\item
For any $u\in U$ and $v\in V$ one has $\gamma(u)m=\chi(u_1)S_su_2v$, where a priori $\chi(u_1)S_su_2m$ should be understood as the natural action of the image of the element $\chi(u_1)S_su_2\in U_{\varepsilon}^s(\g)$ in $U_{\varepsilon}^s(\g)_\lambda$ on the induced $U_{\varepsilon}^s(\g)_\lambda$--module $V'=V\otimes_{U_{\varepsilon}^s(\g)^\lambda}U_{\varepsilon}^s(\g)_\lambda$;

\item
$
\gamma(u)(xv)={\rm Ad}{S_s(u_2)}(x)(\gamma(u_1)v),{\rm for ~all}~u\in U,~x\in U_{\varepsilon}^s(\g)^\lambda,~v\in V,
$
\noindent
$
\Delta_su=u_1\otimes u_2.
$
\end{enumerate}

$(U,\chi)$--equivariant $U_{\varepsilon}^s(\g)^\lambda$--modules form a category $_{~U}^{~~~\chi}{\rm mod}-U_{\varepsilon}^s(\g)^\lambda$ morphisms in which are linear maps of vector spaces respecting all the above introduced structures on equivariant $U_{\varepsilon}^s(\g)^\lambda$--modules.

Formula (\ref{GM}), condition (2) in the definition of biequivariant $\mathcal{D}_\varepsilon$--modules and the obvious relation
$$
u_R(1\otimes x)=\varepsilon(u)1\otimes x,~u,x\in U_{\varepsilon}^s(\g)
$$
imply that if $M$ is a biequivariant $\mathcal{D}_\varepsilon$--module then $\gamma$ induces a $U$--action on $\Gamma(M)$.
From formula (\ref{Requiv}) it also follows that if $M$ is a biequivariant $\mathcal{D}_\varepsilon$--module then $\Gamma(M)$ is an equivariant $U_{\varepsilon}^s(\g)^\lambda$--module. Conversely, the second  relation in (\ref{derLR}) and (\ref{Dequiv}) imply that if $V$ is an equivariant $U_{\varepsilon}^s(\g)^\lambda$--module with an equivariant structure $\gamma$ then the formula
\begin{equation}\label{gdef}
\gamma(u)(v \otimes (a\otimes x))=\gamma(u_1)(v)\otimes {S_s(u_2)}_R(a\otimes x),~v\in V,~a\otimes x\in \mathcal{D}^\lambda_{\varepsilon},~u\in U
\end{equation}
defines the structure of a biequivariant $\mathcal{D}_\varepsilon$--module on $V\otimes_{U_{\varepsilon}^s(\g)^\lambda}\mathcal{D}^\lambda_{\varepsilon}$.

Thus we have the following proposition.
\begin{proposition}\label{BBeq}
If $\lambda$ is regular dominant the global section functor $\Gamma:\mathcal{D}^\lambda_{U_{\varepsilon}^s(\b_+)}\rightarrow {\rm mod}-U_{\varepsilon}^s(\g)^\lambda$ gives rise to an equivalence of the category $_{~U}^{~~~\chi}\mathcal{D}^\lambda_{U_{\varepsilon}^s(\b_+)}$ of biequivariant $\mathcal{D}_\varepsilon$--modules and of the category $_{~U}^{~~~\chi}{\rm mod}-U_{\varepsilon}^s(\g)^\lambda$ of equivariant right $U_{\varepsilon}^s(\g)^\lambda$--modules. The inverse functor is given by
\begin{equation}\label{}
V\mapsto V\otimes_{U_{\varepsilon}^s(\g)^\lambda}\mathcal{D}^\lambda_{\varepsilon},~V\in _{~U}^{~~~\chi}{\rm mod}-U_{\varepsilon}^s(\g)^\lambda.
\end{equation}
\end{proposition}

Denote by $I_\varepsilon^r$ the right ideal in ${\mathbb{C}}_{\varepsilon}[G^*]$ generated by the kernel of $\chi_\varepsilon^s$ in ${\mathbb{C}}_{\varepsilon}[M_-]$, and by $\rho_{\chi^{s}_\varepsilon}$ the canonical projection ${\mathbb{C}}_{\varepsilon}[G^*]\rightarrow I_\varepsilon^r \backslash {\mathbb{C}}_{\varepsilon}[G^*]$. Let $Q_{\varepsilon}^r$ be the image of ${\mathbb{C}}_{\varepsilon}[G_*]$ under the projection $\rho_{\chi^{s}_\varepsilon}$.

Assume that the roots $\gamma_1, \ldots , \gamma_n$ are simple or that  the set $\gamma_1, \ldots , \gamma_n$ is empty, and hence the segment $\Delta_{s^1}$ is of the form $\Delta_{s^1}=\{\gamma_1,\ldots, \gamma_n\}$.
Then from formula (\ref{comults}) it follows that $\Delta_s(U_{\varepsilon}^s(\m_+))\subset U_{\varepsilon}^s(\m_+)\otimes U_{\varepsilon}^s(\b_+)$.

Similarly to Section \ref{qplproups}  we deduce  that the left action ${\rm Ad}\circ S_s$ of $U_{\varepsilon}^s(\m_+)$ on ${\mathbb{C}}_{\varepsilon}[G_*]$ induces an action on $Q_{\varepsilon}^r$ which we denote by ${\rm Ad}\circ S_s$.
One can also define the corresponding W--algebra by
$$
W_\varepsilon^s(G)^r={\rm Hom}_{U_{\varepsilon}^s(\m_+)}(\mathbb{C}_{\varepsilon},Q_\varepsilon^r),
$$
where the multiplication in $W_\varepsilon^s(G)^r$ is induced from ${\mathbb{C}}_{\varepsilon}[G_*]$.

As in Proposition \ref{ZWimb} we have an embedding
$$
Z({\mathbb{C}}_\varepsilon[G_*])\rightarrow W_\varepsilon^s(G)^r.
$$

Note that by Proposition \ref{locfin} ${\mathbb{C}}_\varepsilon[G_*]\simeq U_{\varepsilon}^s(\g)^{fin}$.
Let $Z_\lambda$ be the kernel of the character $\chi_\lambda: Z({\mathbb{C}}_\varepsilon[G_*])\rightarrow \mathbb{C}$.
Consider the quotient
$$
W_\varepsilon^s(G)^r_\lambda=W_\varepsilon^s(G)^r/W_\varepsilon^s(G)^rZ_\lambda.
$$

Observe that for generic $\varepsilon$ we have an algebra isomorphism  ${\mathbb{C}}_{\varepsilon}[M_-]=U_{\varepsilon}^s(\m_+)$ and that $U_{\varepsilon}^s(\m_+)$ is a coideal in $U_{\varepsilon}^s(\g)$. In particular, $\chi^{s}_\varepsilon$ is a character of $U_{\varepsilon}^s(\m_+)$. Therefore one can define the category
$_{U_{\varepsilon}^s(\m_+)}^{\qquad\chi^{s}_\varepsilon}{\rm mod}-U_{\varepsilon}^s(\g)^\lambda$ of $(U_{\varepsilon}^s(\m_+),\chi^{s}_\varepsilon)$--equivariant $U_{\varepsilon}^s(\g)^\lambda$--modules. Consider the full  subcategory $_{U_{\varepsilon}^s(\m_+)}^{\qquad \chi^{s}_\varepsilon}{\rm mod}-U_{\varepsilon}^s(\g)^\lambda_{loc}$ of $_{U_{\varepsilon}^s(\m_+)}^{\qquad\chi^{s}_\varepsilon}{\rm mod}-U_{\varepsilon}^s(\g)^\lambda$ objects of which are finitely generated over $U_{\varepsilon}^s(\g)$ objects of $_{U_{\varepsilon}^s(\m_+)}^{\qquad\chi^{s}_\varepsilon}{\rm mod}-U_{\varepsilon}^s(\g)^\lambda$ such that the $\gamma$--action of the augmentation ideal of $U_{\varepsilon}^s(\m_+)$ on them is locally nilpotent.

Let ${Q_{\varepsilon}^r}_\lambda$ be the image of $U_{\varepsilon}^s(\g)^{\lambda}$ under the projection $\rho_{\chi^{s}_\varepsilon}$.
We have the following straightforward analogue of Theorem \ref{sqeq} for the category $_{~U_{\varepsilon}^s(\m_+)}^{~~~\chi^{s}_\varepsilon}{\rm mod}-U_{\varepsilon}^s(\g)^\lambda_{loc}$.
\begin{proposition}\label{sqeq1}
Assume that the roots $\gamma_1,\ldots, \gamma_n$ (or $\gamma_{n+1},\ldots, \gamma_{l'}$) are simple or one of the sets $\gamma_1,\ldots, \gamma_n$ or $\gamma_{n+1},\ldots, \gamma_{l'}$ is empty. Suppose also that the numbers $t_{i}$ defined in (\ref{defu}) are not equal to zero for all $i$. Then for generic $\varepsilon\in \mathbb{C}$ the functor $E\mapsto E \otimes_{W_\varepsilon^s(G)^r_\lambda} {Q_{\varepsilon}^r}_\lambda$, is an equivalence of the category of finitely generated right $W_\varepsilon^s(G)^r_\lambda$--modules and the category $_{U_{\varepsilon}^s(\m_+)}^{\qquad\chi^{s}_\varepsilon}{\rm mod}-U_{\varepsilon}^s(\g)^\lambda_{loc}$. The inverse equivalence is given by the functor $V\mapsto {\rm Hom}_{U_{\varepsilon}^s(\m_+)}(\mathbb{C}_\varepsilon,V)={\rm Wh}(V)$. In particular, the latter functor is exact.

Every module $V\in _{U_{\varepsilon}^s(\m_+)}^{\qquad\chi^{s}_\varepsilon}{\rm mod}-U_{\varepsilon}^s(\g)^\lambda_{loc}$ is isomorphic to ${\rm hom}_{\mathbb{C}}(U_{\varepsilon}^s(\m_+),\mathbb{C})\otimes {\rm Wh}(V)$ as a left $U_{\varepsilon}^s(\m_+)$--module. In particular, $V$ is $U_{\varepsilon}^s(\m_+)$--injective, and ${\rm Ext}^\bullet_{U_{\varepsilon}^s(\m_+)}(\mathbb{C}_{\varepsilon},V)={\rm Wh}(V)$.
\end{proposition}

Let $\mathbb{C}_\varepsilon[M_+]'$ be the coalgebra which is the quotient of $\mathbb{C}_\varepsilon[G]$ by the coalgebra ideal generated by elements vanishing on $U_{\varepsilon}^s(\m_+)$.
Proposition \ref{sqeq1} implies that the $U_{\varepsilon}^s(\m_+)$--action on the objects of the category $_{U_{\varepsilon}^s(\m_+)}^{\qquad\chi^{s}_\varepsilon}{\rm mod}-U_{\varepsilon}^s(\g)^\lambda_{loc}$ is induced by the adjoint $U_{\varepsilon}^s(\g)$--action on $U_{\varepsilon}^s(\g)^{\lambda}$ which is locally finite. Therefore this action gives rise to a right coaction of $\mathbb{C}_\varepsilon[M_+]'$ on objects of $_{U_{\varepsilon}^s(\m_+)}^{\qquad\chi^{s}_\varepsilon}{\rm mod}-U_{\varepsilon}^s(\g)^\lambda_{loc}$. Conversely, a right $\mathbb{C}_\varepsilon[M_+]'$--coaction on any such object $V$ gives rise to a $U_{\varepsilon}^s(\m_+)$--action such that the action of the augmentation ideal of $U_{\varepsilon}^s(\m_+)$ on it is locally nilpotent. Indeed, the action of the augmentation ideal of $U_{\varepsilon}^s(\b_+)$ on any finite--dimensional $U_{\varepsilon}^s(\g)$--module is locally nilpotent, and hence the action of $U_{\varepsilon}^s(\m_+)\subset U_{\varepsilon}^s(\b_+)$ induced by the coaction of $\mathbb{C}_\varepsilon[M_+]'$ is locally nilpotent as well.

Now observe that in this case the $U_{\varepsilon}^s(\m_+)$--action defined by formula (\ref{gdef}) on the corresponding biequivariant $\mathcal{D}_\varepsilon$--module gives rise to a right $\mathbb{C}_\varepsilon[M_+]'$--coaction which is the tensor product of the right coaction of $\mathbb{C}_\varepsilon[M_+]'$ on $V$ described above and the right coaction of $\mathbb{C}_\varepsilon[M_+]'$ on $\mathcal{D}_\varepsilon^\lambda$ induced by the regular action $(u,a)\mapsto (a)S_s(u)$, $u\in U_{\varepsilon}^s(\g)$, $a\in \mathbb{C}_\varepsilon[G]$, of $U_{\varepsilon}^s(\g)$ on $\mathbb{C}_\varepsilon[G]$ which is locally finite by definition.

Conversely, if $M$ is an object of the category $_{U_{\varepsilon}^s(\m_+)}^{\qquad\chi^{s}_\varepsilon}{\mathcal{D}^\lambda_{U_{\varepsilon}^s(\b_+)}}$ such that the $\gamma$--action of $U_{\varepsilon}^s(\m_+)$ on it is induced by a right $\mathbb{C}_\varepsilon[M_+]'$--coaction then the induced $U_{\varepsilon}^s(\m_+)$--action on $\Gamma(M)$ corresponds to a right $\mathbb{C}_\varepsilon[M_+]'$--coaction on $\Gamma(M)$.

Now consider the full  subcategory $_{U_{\varepsilon}^s(\m_+)}^{\qquad\chi^{s}_\varepsilon}{\mathcal{D}^\lambda_{U_{\varepsilon}^s(\b_+)}}_{loc}$ of $_{U_{\varepsilon}^s(\m_+)}^{\qquad\chi^{s}_\varepsilon}\mathcal{D}^\lambda_{U_{\varepsilon}^s(\b_+)}$ objects of which are finitely generated over $\mathcal{D}_\varepsilon$ objects of $_{U_{\varepsilon}^s(\m_+)}^{\qquad\chi^{s}_\varepsilon}\mathcal{D}^\lambda_{U_{\varepsilon}^s(\b_+)}$ such that for each $M\in {_{U_{\varepsilon}^s(\m_+)}^{\qquad\chi^{s}_\varepsilon}{\mathcal{D}^\lambda_{U_{\varepsilon}^s(\b_+)}}_{loc}}$ the $\gamma$--action of $U_{\varepsilon}^s(\m_+)$ on $M$ is induced by a right $\mathbb{C}_\varepsilon[M_+]'$--coaction.
From Propositions \ref{BBeq} and \ref{sqeq1} and the discussion above we immediately obtain the following statement.
\begin{theorem}
Assume that the roots $\gamma_1,\ldots, \gamma_n$ (or $\gamma_{n+1},\ldots, \gamma_{l'}$) are simple or one of the sets $\gamma_1,\ldots, \gamma_n$ or $\gamma_{n+1},\ldots, \gamma_{l'}$ is empty. Suppose also that the numbers $t_{i}$ defined in (\ref{defu}) are not equal to zero for all $i$. Suppose also that $\lambda$ is regular dominant. Then for generic transcendental $\varepsilon\in \mathbb{C}$ the category $_{U_{\varepsilon}^s(\m_+)}^{\qquad \chi^{s}_\varepsilon}{\mathcal{D}^\lambda_{U_{\varepsilon}^s(\b_+)}}_{loc}$ is equivalent to the category of finitely generated right $W_\varepsilon^s(G)^r_\lambda$--modules.
\end{theorem}

\end{document}